\documentclass[a4paper, 12pt]{article}
\usepackage{amssymb, amsthm,amsmath}
\usepackage{enumerate}
\usepackage{stmaryrd}
\usepackage{mathrsfs}
\PassOptionsToPackage{hyphens}{url} \usepackage{hyperref}
\usepackage[all]{xy}

\makeatletter
\def\blfootnote{\gdef\@thefnmark{}\@footnotetext}
\makeatother

\numberwithin{equation}{section}
\newtheorem{thm}[equation]{Theorem}
\newtheorem{lem}[equation]{Lemma}
\newtheorem{prop}[equation]{Proposition}
\newtheorem{cor}[equation]{Corollary}
\newtheorem{claim}[equation]{Claim}

\theoremstyle{definition}

\theoremstyle{remark}
\newtheorem{rem}[equation]{Remark}

\newtheorem{example}[equation]{Example}
\newtheorem{examples}[equation]{Examples}

\newcommand{\C}{\mathcal{C}}
\newcommand{\D}{\mathcal{D}}
\newcommand{\K}{\mathcal{K}}
\newcommand{\M}{\mathcal{M}}
\newcommand{\Q}{\mathcal{Q}}
\newcommand{\U}{\mathcal{U}}
\newcommand{\V}{\mathcal{V}}
\newcommand{\proj}{\mathcal{P}}
\newcommand{\F}{\operatorname{F}}
\newcommand{\FP}{\operatorname{FP}}
\newcommand{\id}{\operatorname{id}}
\newcommand{\coker}{\operatorname{coker}}
\newcommand{\im}{\operatorname{im}}
\newcommand{\dom}{\operatorname{dom}}
\newcommand{\Hom}{\operatorname{Hom}}
\newcommand{\End}{\operatorname{End}}
\newcommand{\colim}{\operatorname{colim}}
\newcommand{\disc}{\operatorname{disc}}
\newcommand{\eq}{\operatorname{eq}}
\newcommand{\coeq}{\operatorname{coeq}}
\newcommand{\sd}{\operatorname{sd}}
\newcommand{\Ex}{\operatorname{Ex}}
\newcommand{\Ext}{\operatorname{Ext}}
\newcommand{\Tor}{\operatorname{Tor}}
\newcommand{\ob}{\operatorname{ob}}
\newcommand{\Sing}{\operatorname{Sing}}
\newcommand{\Kan}{\operatorname{Kan}}
\newcommand{\diag}{\operatorname{diag}}

\begin{document}

%

\title{Homotopical Algebra in Categories with Enough Projectives}
\author{Ged Corob Cook\thanks{The author was supported by EPSRC grant EP/N007328/1.}}
\date{}
\renewcommand\footnotemark{}
\maketitle

\begin{abstract}
For a complete and cocomplete category $\C$ with a well-behaved class of `projectives' $\bar{\proj}$, we construct a model structure on the category $s\C$ of simplicial objects in $\C$ where the weak equivalences, fibrations and cofibrations are defined in terms of $\bar{\proj}$. This holds in particular when $\C$ is $\U$, the category of compactly generated, weakly Hausdorff spaces, and $\bar{\proj}$ is the class of compact Hausdorff spaces.

We also construct a new model structure on $\U$ itself, where the cofibrant spaces are generalisations of CW-complexes allowing spaces, rather than sets, of $n$-cells to be attached. The singular simplicial complex and geometric realisation functors give a Quillen adjunction between these model structures.

For a space in $\U$, these structures allow the definition of homotopy group objects in the exact completion of $\U$, which are invariant under weak equivalence and have a lot of the nice properties usually expected of homotopy groups. There is a long exact sequence of homotopy group objects arising from a fibre sequence in $\U$.

Working along similar lines, we study homological algebra in categories of internal modules in $\U$, getting in particular a Lyndon--Hochschild--Serre spectral sequence for extensions of topological groups in $\U$.
\end{abstract}

%


\newpage
\tableofcontents
\newpage

\section*{Introduction}

For a few decades, there has been a good understanding of the categorical structure needed to use homological and homotopical algebra in studying algebraic objects. There are various approaches to extending this structure to study objects with extra structure, such as topological groups, but naive attempts to generalise, say, group cohomology to topological groups only obtain essential tools like long exact sequences in very limited circumstances. See \cite{Stasheff} for a summary of these attempts. In another direction, progress has been made in studying homotopical algebra (via model categories) and homological algebra (via triangulated categories) in more general category-theoretic terms. But there has been little attempt to unify these abstract and concrete ways of thinking to get results in topological algebra.

We present here a systematic framework for defining homotopy and homology group objects for a space, and homology and cohomology group objects for a topological group, which preserve topological information that is lost by existing approaches. A lot of the tools for working with them hold in this context: we obtain long exact sequences in very general circumstances, for example. The highlight of this approach in the current work is a Lyndon--Hochschild--Serre spectral sequence for topological groups, as mentioned above. Further work by the author, in \cite{Myself}, uses this framework to prove topological versions of many of the foundational theorems of algebraic topology: we prove a Seifert--van Kampen Theorem, an Excision Theorem, a Mayer--Vietoris sequence, and we construct Eilenberg--Mac Lane spaces for a class of topological groups.

We will usually work in the convenient category $\U$ of compactly generated, weakly Hausdorff spaces, which we call \emph{$k$-spaces}. Among other nice properties, this is cartesian closed, unlike the category of all topological spaces; it is regular and coregular, which turns out to be crucial to defining well-behaved homotopy group objects on simplicial objects, and well-behaved (co)homology group objects on (co)chain complexes; every space in $\U$ can be seen canonically as a quotient of a disjoint union of compact Hausdorff spaces; compact Hausdorff spaces are finite relative to closed inclusions in $\U$.

The theory of CW-complexes and Eilenberg--Mac Lane spaces is well adapted to the study of abstract groups, but to build spaces with topological information on the homotopy groups is not so easy. There has been some work on the `topological fundamental group' of a space -- though it should be noted (see \cite{Fabel}) that this construction does not always give a topological group -- but there has been no good understanding of how to build such spaces systematically. In this direction, we construct a new model structure on the category $\U$ of spaces and show that this has the potential to give more information about $k$-groups than the usual Quillen model structure. All spaces are fibrant in this structure, and the class of cofibrant spaces may be thought of as a generalisation of CW-complexes, where we may attach a set of compact Hausdorff spaces of $n$-cells, rather than a discrete set of $n$-cells, for each $n$. We call these spaces KW-complexes. This model structure is crucial for allowing the construction of Eilenberg--Mac Lane spaces for topological groups in \cite{Myself}.

The usual approach to doing homological algebra using projective resolutions is to start with a well-behaved class of epimorphisms and then specify that projectives should be objects which have a lifting property with respect to maps in this class. For example this is the way things are done in Quillen-exact categories (for which see \cite{Buhler}). We attack from the opposite direction: start with a class of projectives and take a class of morphisms to be those that satisfy a lifting property with respect to projective objects. This class is automatically well-behaved in the relevant sense. This gives rise to a concept of a category having enough projectives; it turns out that having enough projectives is enough to define derived functors in an additive context. In a general category $\C$, we show that if $\C$ has enough projectives and the projectives satisfy an appropriate smallness condition, then there is a model structure on $s\C$, the category of simplicial objects in $\C$. This generalises \cite[Theorem II.4.4]{Quillen}: crucially, it allows model structures to be cofibrantly generated by a proper class of objects satisfying an additional condition, rather than just a set of objects. This condition is satisfied by the class of compact Hausdorff spaces, giving a model structure on $s\U$, the category of simplicial spaces. There is a Quillen adjunction between $s\U$ with this structure and $\U$ with the structure mentioned above, but we show it is not a Quillen equivalence.

This paper should be considered a framework for further research, not a complete theory. With the definitions given here, there are many possible avenues for further investigation, not all of which are dealt with in the body of the paper, including: further explicit calculations of homotopy group objects and homology group objects; using these to study spaces, such as wild knots, whose usual homotopy groups are too coarse to capture local structure; studying properties of topological groups using their cohomology group objects; defining new homology theories in topological algebra, and in other categories; giving a notion of enriched $\infty$-categories; etc. A sample of the results which can be obtained within this framework can be seen in the work of \cite{Myself} -- but there is much more to be done.

We summarise here the structure of the paper. Section \ref{kspace} gives results from the literature on $k$-spaces and $k$-groups that will be needed later. In Section \ref{modelU} we construct a new model structure on $\U$ whose cofibrant objects are retracts of KW-complexes; we compare the model structure to previously known ones on $\U$ and study the properties of KW-complexes.

Section \ref{se} contains the category-theoretic results needed to make smoother the proofs of Section \ref{mct}, where we study the structure of $s\C$ for a category $\C$ with a fixed class $\bar{\proj}$ of `projective objects'. We show that the full subcategory of Kan complexes in $s\C$ is a category of fibrant objects, and construct a model structure on $s\C$ for any complete and cocomplete category $\C$ which has enough projectives satisfying a smallness condition.

In Section \ref{reghomotopy} we give some background on regular categories, in order to define, for an internal Kan complex in a regular category $\C$, homotopy group objects in the exact completion of the category. We show that, if $\C$ also has a class of projective objects, and morphisms that satisfy a lifting property with respect to projective objects are regular epimorphisms, then these homotopy group objects are invariants under weak equivalence of objects in the category of fibrant objects given by Kan complexes in $s\C$. We apply this in Section \ref{calcs} to investigate the properties of these invariants, and their relation to other definitions of homotopy groups.

In Section \ref{kmodules} we investigate the category of $k$-modules for a $k$-ring (that is, module objects for a ring object internally to $\U$) and show that it has many of the nice properties usually associated with categories of modules for an abstract ring: it is complete and cocomplete, quasi-abelian, has free modules given by a left adjoint to the forgetful functor to $\U$, has a tensor product $\otimes_R$, and has an internal $\Hom$ which is right adjoint to $\otimes_R$.

In Sections \ref{additive} and \ref{derived} we show how our theory works in the additive case, in particular for categories of $k$-modules, and derive some of the usual tools of abstract group cohomology. In particular we get a Lyndon--Hochschild--Serre spectral sequence in this context.

\section{\texorpdfstring{$k$-Spaces, $k$-Groups}{k-Spaces, k-Groups}}
\label{kspace}

The default category of spaces in this paper, unless stated otherwise, will be the category $\U$ of compactly generated weakly Hausdorff spaces, which we call $k$-spaces (though note there are different naming conventions in the literature: what we call $k$-spaces are sometimes called compactly generated, and what we call compactly generated spaces are sometimes called $k$-spaces). All first-countable spaces, and all locally compact Hausdorff spaces, are compactly generated. We give here a summary of the results we will need. A good reference for these facts is \cite{Strickland}.

Given a topological space $X$, we say a subset $U \subseteq X$ is \emph{$k$-closed} if $u^{-1}(U)$ is closed in $K$ for all compact Hausdorff $K$ and all continuous maps $u: K \to X$. We say $X$ is \emph{compactly generated} or CG if every $k$-closed subset of $X$ is closed. Given a topology $\tau$ on $X$, we can define a new topology $k(\tau)$ on $X$ by defining the closed subsets in $k(\tau)$ to be subsets which are $k$-closed in $(X,\tau)$. We write $k(X)$ for $(X,k(\tau))$.

By definition, $k(X)$ is CG, and this gives a functor $k( )$ from the category $Top$ of topological spaces to the category $\K$ of compactly generated spaces, where the morphisms in both cases are the continuous maps. $k( )$ is right adjoint to the inclusion $\K \to Top$.

Colimits of diagrams of CG spaces, calculated in $Top$, are CG; it follows that this procedure gives colimits in $\K$. The adjunction of $k( )$ implies that limits of diagrams of CG spaces can be calculated by applying $k( )$ to the limit of the diagram in $Top$. We may write the $k$-subspace topology on a subset $Y$ of $X$ for $k( )$ applied to the usual subspace topology. So $\K$ is complete and cocomplete (that is, has all small limits and colimits), because $Top$ is.

Note in particular that, for CG spaces $X,Y$, if we write $X \times Y$ for the product in $\K$ and $X \times_0 Y$ for the product in $Top$, then $X \times Y \neq X \times_0 Y$ in general. The notation $- \times -$ will always mean the product in $\K$. If $X$ is locally compact and Hausdorff then we do have $X \times Y = X \times_0 Y$.

A topological space $X$ is called \emph{weakly Hausdorff} or WH if $u(K)$ is closed in $X$ for all compact Hausdorff $K$ and all continuous maps $u: K \to X$. In terms of separation conditions, $T_1 \Rightarrow WH \Rightarrow \text{Hausdorff}$.

A CG space $X$ is WH if and only if the diagonal $$\Delta_X = \{(x,x) \in X \times X\}$$ is closed in $X \times X$. Note the similarity to the result that $X$ is Hausdorff if and only if $\Delta_X$ is closed in $X \times_0 X$. Thus given $X \in \K$, we can functorially define a WH space $(X)_{WH}$ by quotienting $X$ by the smallest closed equivalence relation: see this exists by taking the intersection of all the closed equivalence relations. The functor $( )_{WH}$ from $\K$ to the category $\U$ of $k$-spaces, which we call \emph{weak Hausdorffisation}, is left adjoint to inclusion $\U \to \K$.

Limits of diagrams of $k$-spaces, calculated in $\K$, are WH, so this procedure gives limits in $\U$. The adjunction of $( )_{WH}$ implies that colimits of diagrams of $k$-spaces can be calculated by applying $( )_{WH}$ to the colimit of the diagram in $\K$. So $\U$ is complete and cocomplete. Note these statements show that limits in $\U$ are preserved by the forgetful functor to $Set$.

The main reason for using $\K$ or $\U$ instead of $Top$ is that the \emph{compact-open topology} on spaces of maps is better-behaved. The definition is a modification of the usual definition of compact-open topology.

For $X,Y \in \K$, write $\K(X,Y)$ for the set of morphisms $X \to Y$. For every compact Hausdorff space $K$ and continuous $u: K \to X$, and every open $U \subseteq Y$, define $$W(u,U) = \{f: X \to Y : fu(K) \subseteq U\} \subseteq \K(X,Y).$$ Then we can endow $\K(X,Y)$ with the coarsest compactly generated topology in which the sets $W(u,U)$ are open; we will call this topology the \emph{$k$-compact-open topology} and write $\underline{\K}(X,Y)$ for the space. This makes $\underline{\K}(-,-)$ into a bifunctor $\K^{op} \times \K \to \K$. When $Y$ is WH, $\underline{\K}(X,Y)$ is WH too, so that $\underline{\K}(-,-)$ is also a bifunctor $\U^{op} \times \U \to \U$, for which we will write $\underline{\U}(-,-)$. We will also write $\underline{Top}(X,Y)$ for the continuous maps $X \to Y$ with the usual compact-open topology.

If $X$ is WH, any subspace $Z$ of $X$ is WH, so the diagonal is closed in $Z \times Z$. If $Z$ is compact, $Z \times Z = Z \times_0 Z$, so the closure of the diagonal shows $Z$ is Hausdorff. Hence every continuous $u: K \to X$ with $K$ compact Hausdorff factors through its compact Hausdorff image, and the $k$-compact-open topology on $\K(X,Y)$ is determined by the open sets $$W(K,U) = \{f: X \to Y : f(K) \subseteq U\}$$ with $K \subseteq X$ compact Hausdorff and $U$ open. Summarising, we have:

\begin{lem}
\label{WHhom}
If $X$ is WH, $\underline{\K}(X,Y) = k(\underline{Top}(X,Y))$.
\end{lem}

\begin{thm}
For $X,Y,Z \in \K$, define the maps
\begin{align*}
&ev: X \times \underline{\K}(X,Y) \to Y, && (x,f) \mapsto f(x), \\
&inj: Y \to \underline{\K}(X,X \times Y), && y \mapsto (x \mapsto (x,y)), \\
&adj: \underline{\K}(X,\underline{\K}(Y,Z)) \to \underline{\K}(X \times Y,Z), && f \mapsto ((x,y) \mapsto f(x)(y)).
\end{align*}
Then:
\begin{enumerate}[(i)]
\item $ev$ is continuous;
\item $inj$ is continuous; and
\item $adj$ is naturally a homeomorphism.
\end{enumerate}
\end{thm}

So in particular $\K$ and $\U$ are cartesian closed. Note that these results fail in $Top$. Products in cartesian closed categories are left adjoints, so we get:

\begin{cor}
For $X \in \K$ (respectively, $\U$) the functor $X \times -$ commutes with colimits in $\K$ (respectively, $\U$).
\end{cor}

We now restrict ourselves to the consideration of $k$-spaces. A map in a category $\C$ with all finite limits and colimits is called a \emph{regular epimorphism}, written $\twoheadrightarrow$, if it is a coequaliser of a pair of maps, and a \emph{regular monomorphism}, written $\rightarrowtail$, if it is an equaliser of a pair of maps. Then $\C$ is called \emph{regular} if pull-backs preserve regular epimorphisms, and \emph{coregular} if push-outs preserve regular monomorphisms.

\begin{thm}
\phantomsection
\label{currying}
\begin{enumerate}[(i)]
\item Regular epimorphisms in $\U$ are exactly the quotient maps.
\item Regular monomorphisms in $\U$ are exactly the closed inclusions.
\item $\U$ is regular and coregular.
\end{enumerate}
\end{thm}

I do not know whether $\K$ is regular or coregular.

Now we consider algebraic objects in $\U$. The idea is to start with a Lawvere theory $T$, and then define $\C^T$, the $T$-objects in a category $\C$, to be the category of product-preserving functors $T \to \C$ and natural transformations between them. We give an explicit description of group objects in $\U$, on the understanding that `group' and `$\U$' generalise in an obvious way. We call these group objects \emph{$k$-groups}.

Write $\{\ast\}$ for the one-point space (categorically: the terminal object). Then a $k$-group is an object $G \in \U$ together with maps $$m: G^2 \to G, i: G^1 \to G, e: G^0 = \{\ast\} \to G$$ satisfying the usual relations of the multiplication, inversion and identity maps of groups:
\begin{enumerate}[(i)]
\item $m \circ (m \times \id_G) = m \circ (\id_G \times m): G \times G \times G \to G$;
\item $m \circ (\id_G \times e)$ is the canonical projection $G \times \{\ast\} \to G$; and
\item $m \circ (\id_G \times i) \circ d$ and $m \circ (i \times \id_G) \circ d$ are the trivial map $G \to \{\ast\} \xrightarrow{e} G$, where $d$ is the diagonal map $G \to G \times G$.
\end{enumerate}
Then a morphism of $k$-groups is a morphism in $\U$ which is compatible with the maps $m,i,e$, in an appropriate sense.

Alternatively, and perhaps more explicitly, we can define a $k$-group to be a group equipped with a $k$-space topology making the multiplication and inversion maps continuous: here morphisms are the group homomorphisms which are continuous.

A good reference for $k$-groups, providing proofs for the statements below, is \cite{Lamartin}.

Write $TopGrp$, $\K Grp$ and $\U Grp$ for the categories of group objects and morphisms in $Top$, $\K$ and $\U$, respectively. All three are complete and cocomplete. Given $G \in TopGrp$, $k(G)$ is automatically a CG group, though not all CG groups arise in this way. Of course, this gives a functor $TopGrp \to \K Grp$. On the other hand there is no obvious inclusion $\K Grp \to TopGrp$, as the multiplication in a CG group might not be a continuous map $G \times_0 G \to G$.

A CG group $G$ is a $k$-group if and only if the identity is closed in $G$, and the functor $$( )_{WH}: \K Grp \to \U Grp$$ is given by $(G)_{WH} = G/\bar{\{1_G\}}$, with the quotient topology. This is left adjoint to inclusion $\U Grp \to \K Grp$.

Many properties of CG groups and $k$-groups are similar to those of topological groups: given $H \leq G$ in either of these categories, $\bar{H}$ is again a subgroup of $G$, which is normal if $H$ is (normal here should be interpreted as a condition on the underlying groups); for $H \lhd G$ in either category, $G/H$ is a CG-group and $G/\bar{H}$ is a $k$-group; the quotient map $G \to G/H$ is open. LaMartin in \cite[Proposition 2.2]{Lamartin} claims that the multiplication map $G \times G \to G$ is open, though without proof, and the proof for topological groups only shows that the image of an open subset of $G \times_0 G$ is open; so the question may still be open.

On the other hand, in $TopGrp$ $T_0 \Rightarrow T_{3\frac{1}{2}}$, and in $\K Grp$ $T_0$ does not imply Hausdorff.

As in $TopGrp$, there are free group objects, that is, the forgetful functors $\K Grp \to \K$ and $\U Grp \to \U$ have left adjoints. Here, however, unlike in $TopGrp$, these left adjoints can be easily given by an explicit construction, because, as noted above, products in $\K$ and $\U$ commute with colimits. Indeed, free group objects exist very generally in cartesian closed topological categories, as shown in \cite{Seal}. The construction in \cite{Lamartin} shows that the free CG-group on a $k$-space is actually the free $k$-group on that space.

Finally, we also have free abelian group objects, constructed similarly, and the free abelian CG-group on a $k$-space is the free abelian $k$-group on that space.

\section{The model structure on \texorpdfstring{$\U$}{U}}
\label{modelU}

We use the definitions of model structures and model categories given in \cite[Section 1.1]{Hovey}. That is, a model category is a complete and cocomplete category together with a model structure, and factorisations (unless stated otherwise) are required to be functorial. The definitions of other model-theoretic terminology will also be taken from \cite{Hovey}.

Before going on to more abstract proofs in more abstract categories, we start by constructing a new model structure on $\U$. As well as being of independent interest, this argument gives an intuition for the approaches that will be used later on.

In a (locally small) category $\C$, we write $\C(X,Y)$ for the set of morphisms $X \to Y$ in $\C$. Recall, for a cardinal $\kappa$, a category $\C$ and a class $\M$ of morphisms in $\C$, that an object $X$ in $\C$ is said to be \emph{$\kappa$-small relative to $\M$} if, for every $\lambda$-sequence of morphisms $(f_i)_{i \in \lambda}$ in $\M$, where $\lambda$ is an ordinal with cofinality greater than $\kappa$, $$\C(X,\colim f_i) \cong \colim \C(X,f_i).$$ In the special case where $\kappa$ is finite, we say $X$ is finite relative to $\M$.

The following construction uses a generalisation of Quillen's small object argument, due to \cite{Chorny}. The point is simply that a model structure may look cofibrantly generated, except for the intervention of set-theoretic issues, and in well-behaved cases these may be overcome.

\begin{prop}
\label{smallobject}
\cite[Theorem 1.1]{Chorny}: Suppose $\C$ is a category containing all small colimits, and $I$ is a class of maps in $\C$ satisfying the following conditions:
\begin{enumerate}[(i)]
\item There exists a cardinal $\kappa$, such that each element $A \in \dom(I)$ is $\kappa$-small relative to $I$-cell.
\item For every map $f$ in $\C$ there exists a (functorially assigned) map $S(f) \in I$-cell, equipped with a (natural) morphism of maps $t_f: S(f) \to f$, such that any morphism of maps $i \to f$ with $i \in I$ factors through the (natural) map $t_f$.
\end{enumerate}
Then there is a (functorial) factorisation $(\alpha,\beta)$ on $\C$ such that, for all morphisms $f$ in $\C$, the map $\alpha(f)$ is in $I$-cell and the map $\beta(f)$ is in $I$-inj.
\end{prop}

We say such a class $I$ \emph{permits the} (\emph{functorial}) \emph{generalised small object argument}.

\begin{rem}
We have changed the statement from \cite{Chorny} slightly to require that each domain in $I$ is $\kappa$-small relative to $I$-cell, and that $S(f)$ be in $I$-cell, rather than $I$-cof; the proof of \cite[Theorem 1.1]{Chorny} then constructs $\alpha(f)$ as a transfinite composition of pushouts by maps in $I$-cell, which is once again in $I$-cell as required. Indeed, in our applications it seems to be much easier to prove initially that the objects we are interested in are small relative to $I$-cell than to $I$-cof.
\end{rem}

Thanks to the factorisations constructed in Proposition \ref{smallobject}, we may rewrite the statement of \cite[Theorem 2.1.19]{Hovey} to allow for classes of generating (trivial) cofibrations; the proof goes through \textit{mutatis mutandis}.

\begin{prop}
\label{cofibgen}
Suppose $\C$ is a category with all small colimits and limits. Suppose $W$ is a subcategory of $\C$, and $I$ and $J$ are classes of maps of $\C$ which permit the functorial generalised small object argument. Then there is a class-cofibrantly generated model structure on $\C$ with $I$ as the class of generating cofibrations, $J$ as the class of generating trivial cofibrations, and $W$ as the subcategory of weak equivalences if and only if the following conditions are satisfied:
\begin{enumerate}[(i)]
\item The subcategory $W$ has the two-out-of-three property and is closed under retracts.
\item $J$-cell $\subseteq W \cap I$-cof.
\item $I$-inj $\subseteq W \cap J$-inj.
\item Either $W \cap I$-cof $\subseteq J$-cof or $W \cap J$-inj $\subseteq I$-inj.
\end{enumerate}
\end{prop}

Write $\bar{\proj}$ for the class of compact Hausdorff spaces in $\U$, and $\proj$ for the disjoint unions of compact Hausdorff spaces. Given $X \in \U$, define $d(X)$ to be the disjoint union of the compact Hausdorff subspaces of $X$. Then any $f: Y \to X$ with $Y \in \bar{\proj}$ factors canonically (though \textit{not} uniquely) through the obvious map $d(X) \to X$, by sending $Y$ to the $\im(f)$-component of $d(X)$. If we take $Y \in \proj$, $f$ still factors through $d(X)$ but this is no longer canonical, as there is no canonical way of deciding how to decompose $Y$ as a disjoint union of compact Hausdorff spaces. On the other hand, given a map $g: X \to Z$ in $\U$, we can define $d(g):d(X) \to d(X)$ functorially, by sending each compact Hausdorff subspace $K$ of $X$ to the component of $d(Z)$ corresponding to the image of $K$ in $Z$. This makes $d$ into a functor, which we call the \emph{CH-subspaces functor}.

Though all compactly generated spaces can be written as quotients of weakly Hausdorff ones, I do not know of any way to extend $d$ functorially to $\K$.

Let $I$ be the class of maps in $\U$ of the form $\iota \times \id_K: \partial \Delta_n \times K \to \Delta_n \times K$, $K \in \bar{\proj}$, and let $J$ be the class of maps of the form $\iota' \times \id_K: \Lambda_n \times K \to \Delta_n \times K$, $K \in \bar{\proj}$, where $\Lambda_n$ is the $n$-horn, $\Delta_n$ is the $n$-simplex and $\iota, \iota'$ are the inclusion maps.

\begin{thm}
Let $W$ be the class of maps $f$ in $\U$ such that $\underline{\U}(K,f)$ is a weak homotopy equivalence for all $K \in \bar{\proj}$. In the terminology of \cite[Definition 1.3]{Chorny}, with $I$ as the generating cofibrations, $J$ as the generating trivial cofibrations, and $W$ as the weak equivalences, $\U$ is a class-cofibrantly generated right-proper symmetric monoidal model category.
\end{thm}

We will show this as a series of lemmas, using Proposition \ref{cofibgen}.

\begin{lem}
\label{IandJwork}
$I$ and $J$ permit the functorial generalised small object argument.
\end{lem}
\begin{proof}
We give a proof for $I$; $J$ holds similarly.

The domains in $I$ are all compact and Hausdorff, so they are finite relative to closed inclusions in $\U$ by \cite[Proposition 2.4.2]{Hovey}. Note that every map in $I$ is a closed inclusion. Closed inclusions are closed under pushouts and transfinite composition, by the proof of \cite[Lemma 2.4.5]{Hovey}.

Suppose $f: X \to Y$ is a map in $\U$. Consider the pullback $$A_n = \underline{\U}(\partial\Delta_n,X) \times_{\underline{\U}(\partial\Delta_n,Y)} \underline{\U}(\Delta_n,Y)$$ in $\U$, and write $B_n$ for the disjoint union of the compact Hausdorff subspaces of $A_n$. Finally, let $S(f)$ to be the coproduct over $n$ of the maps $\partial\Delta_n \times B_n \to \Delta_n \times B_n$. Then $S(f)$ is in $I$-cell by \cite[Lemma 2.1.13]{Hovey}, and is functorially assigned. Now define $e_1$ to be the composite $$\bigsqcup_n\partial\Delta_n \times B_n \to \bigsqcup_n\partial\Delta_n \times A_n \to \bigsqcup_n\partial\Delta_n \times \underline{\U}(\partial\Delta_n,X) \to X,$$ where the first two maps are the canonical ones on the second factor and the identity on the first factor, and the third map is evaluation. Similarly, define $e_2$ to be the composite $$\bigsqcup_n\Delta_n \times B_n \to \bigsqcup_n\Delta_n \times A_n \to \bigsqcup_n\Delta_n \times \underline{\U}(\Delta_n,Y) \to Y.$$ Then the square
\[
\xymatrix{\bigsqcup_n\partial\Delta_n \times B_n \ar[r]^{e_1} \ar[d] & X \ar[d]^{f} \\ \bigsqcup_n\Delta_n \times B_n \ar[r]^{e_2} & Y}
\]
commutes, as required, giving the morphism of maps $t_f$. This is natural by construction.

There is a natural bijection between commutative squares of the form
\[
\xymatrix{\partial\Delta_n \times K \ar[r] \ar[d] & X \ar[d]^{f} \\ \Delta_n \times K \ar[r] & Y,}
\]
$K \in \bar{\proj}$, and elements of the set $\underline{\U}(\partial\Delta_n \times K,X) \times_{\underline{\U}(\partial\Delta_n \times K,Y)} \underline{\U}(\Delta_n \times K,Y)$; this set is naturally isomorphic to $$\underline{\U}(K,\underline{\U}(\partial\Delta_n,X) \times_{\underline{\U}(\partial\Delta_n,Y)} \underline{\U}(\Delta_n,Y)) = \underline{\U}(K,A_n),$$ using cartesian closedness and the fact that $\underline{\U}(K,-)$ commutes with limits. So each such commutative square identifies a map $K \to A_n$, which lifts canonically to $g: K \to B_n$, giving a diagram
\[
\xymatrix{\partial\Delta_n \times K \ar[r] \ar[d] & \partial\Delta_n \times B_n \ar[r] \ar[d] & \bigsqcup_n\partial\Delta_n \times B_n \ar[r]^{e_1} \ar[d] & X \ar[d]^{f} \\ \Delta_n \times K \ar[r] & \Delta_n \times B_n \ar[r] & \bigsqcup_n\Delta_n \times B_n \ar[r]^{e_2} & Y.}
\]
This diagram commutes and gives the required factorisation.
\end{proof}

\begin{lem}
$W$ has the two-out-of-three property and is closed under retracts.
\end{lem}
\begin{proof}
This is immediate from the corresponding property for weak homotopy equivalences.
\end{proof}

\begin{lem}
\phantomsection
\label{equivalences}
\begin{enumerate}[(i)]
\item Homotopy equivalences in $\U$ are weak equivalences.
\item Weak equivalences are weak homotopy equivalences.
\item If $f$ is a weak equivalence, so is $\underline{\U}(K,f)$ for $K \in \bar{\proj}$.
\end{enumerate}
\end{lem}
\begin{proof}
\begin{enumerate}[(i)]
\item For homotopic maps $f,g: X \to Y$, the induced maps $\underline{\U}(K,X) \to \underline{\U}(K,Y)$ are homotopic, using the natural isomorphism $$\underline{\U}(K,\underline{\U}([0,1],Y)) \cong \underline{\U}([0,1],\underline{\U}(K,Y)).$$ The result follows.
\item Take $K = \{\ast\}$.
\item For $K' \in \bar{\proj}$, the induced map $\underline{\U}(K',\underline{\U}(K,f))$ is a weak homotopy equivalence as required, because it is isomorphic to $\underline{\U}(K' \times K,f)$, and $K' \times K \in \bar{\proj}$.
\end{enumerate}
\end{proof}

\begin{lem}
$J$-cell $\subseteq W \cap I$-cof.
\end{lem}
\begin{proof}
The map $\Lambda_n \to \Delta_n$ can be written as a finite composition of pushouts by maps of the form $\partial \Delta_m \to \Delta_m$; indeed, it is a finite CW-complex. Taking products of all the maps involved with some $K \in \bar{\proj}$ expresses each map $\Lambda_n \times K \to \Delta_n \times K$ in $J$ as an element of $I$-cell (recall that $- \times K$ preserves pushouts in $\U$), so $J$-cell $\subseteq I$-cell $\subseteq I$-cof.

Pushouts by elements of $J$ are weak equivalences by Lemma \ref{equivalences}, because they are inclusions of deformation retracts and hence homotopy equivalences. So elements $f$ of $J$-cell are transfinite compositions $(f_\lambda)$ of weak equivalences that are also closed inclusions. Therefore, for every $K \in \bar{\proj}$, $(\underline{\U}(K,f_\lambda))$ is again a transfinite composition, where each map is a weak homotopy equivalence by definition, and each map is a closed inclusion by \cite[Proposition 2.37]{Strickland}. Hence the transfinite composition of $(\underline{\U}(K,f_\lambda))$ is a weak equivalence by \cite[Lemma 2.4.8]{Hovey}, and $\colim_\lambda \underline{\U}(K,f_\lambda) = \underline{\U}(K, \colim_\lambda f_\lambda)$ for limit ordinals by \cite[Proposition 2.4.2]{Hovey}, so we are done.
\end{proof}

\begin{lem}
$I$-inj $= W \cap J$-inj.
\end{lem}
\begin{proof}
A map $f$ is in $I$-inj if and only if $\underline{\U}(K,f)$ has the right lifting property with to the maps $\partial\Delta_n \to \Delta_n$ for all $n$ and $K \in \bar{\proj}$, if and only if $C(K,f)$ is a Serre fibration and a weak homotopy equivalence for all $K \in \bar{\proj}$ by \cite[Proposition 2.4.10, Theorem 2.4.12]{Hovey}, if and only if $f \in W \cap J$-inj.
\end{proof}

This completes the proof of the model structure. We can now prove the other claims.

\begin{lem}
Pullbacks preserve weak equivalences.
\end{lem}
\begin{proof}
Recall that $\underline{\U}(K,-)$ commutes with pullbacks for all $K$. Then given a weak equivalence $f$ in $\U$, and a pullback $g$ of $f$ by another map $h$, $\underline{\U}(K,g)$ is the pullback of $\underline{\U}(K,f)$ by $\underline{\U}(K,h)$, for all $K \in \bar{\proj}$. By definition, $\underline{\U}(K,f)$ is a weak homotopy equivalence, and the Quillen model structure is right proper by \cite[Proposition 2.4.18]{Hovey}, so $\underline{\U}(K,g)$ is a weak homotopy equivalence for all $K \in \bar{\proj}$, and hence $g$ is a weak equivalence.
\end{proof}

\begin{lem}
This model structure on $\U$ is symmetric monoidal.
\end{lem}
\begin{proof}
The symmetric monoidal structure is just the usual cartesian closedness of $\U$. Since $\{\ast\}$ is easily seen to be cofibrant, we only need to check the first condition in \cite[Definition 4.2.6]{Hovey}; by \cite[Corollary 4.2.5]{Hovey}, it suffices to check the condition on maps in our classes of generating cofibrations. Since $\U$ with the Quillen model structure is symmetric monoidal, this is an exercise in writing out definitions and commuting $- \times K$ with pushouts, which we leave for the reader.
\end{proof}

\begin{lem}
\begin{enumerate}[(i)]
\item $- \times K$ preserves cofibrations and trivial cofibrations, for all $K \in \bar{\proj}$.
\item $\underline{\U}(K,-)$ preserves fibrations and trivial fibrations, for all $K \in \bar{\proj}$.
\end{enumerate}
\end{lem}
\begin{proof}
Using cartesian closedness of $\U$, these statements are equivalent, since the classes in each part are defined by lifting properties relative to classes in the other. We prove (ii). If $f$ is a fibration, $\underline{\U}(K',\underline{\U}(K,f)) = \underline{\U}(K' \times K,f)$ is a Serre fibration for all $K' \in \bar{\proj}$, as required. Similarly for trivial fibrations.
\end{proof}

We now want to relate our new model structure, which we will call the compact Hausdorff model structure, to the more common model structures on $\U$, namely the Quillen and Hurewicz model structures. We will write $\U_{CH}, \U_Q$ and $\U_H$ respectively for the three model categories.

\begin{rem}
By combining these model structures, we can make new ones, using \cite[Theorem 4.1.1]{MS}. So for example we get one with the weak equivalences of $\U_{CH}$, Hurewicz fibrations, and cofibrations given by closed Hurewicz cofibrations which are the composite of a cofibration in $\U_{CH}$ and a homotopy equivalence.
\end{rem}

\begin{prop}
The identity on $\U$ gives left Quillen functors $\U_Q \to \U_{CH} \to \U_H$.
\end{prop}
\begin{proof}
Certainly $\id_\U$ is left adjoint to itself. For the first left Quillen functor, it is trivial that the generating cofibrations and generating trivial cofibrations of the Quillen model structure are in $I$ and $J$, respectively, and it follows that cofibrations and trivial cofibrations in the Quillen model structure are in $I$-cof and $J$-cof, respectively.

For the second, we use the alternative characterisation of Quillen functors given in \cite[Lemma 1.3.4]{Hovey}. A map $f$ which is a Hurewicz fibration has the right lifting property with respect to maps of the form $X \to X \times [0,1]$, $X \in \U$, so it certainly has the right lifting property with respect to maps in $J$. So $\id_\U: \U_H \to \U_{CH}$ preserves fibrations, and it preserves weak equivalences too by Lemma \ref{equivalences}, so it preserves trivial fibrations.
\end{proof}

However, we will see later that neither of these functors is a Quillen equivalence.

\begin{lem}
All spaces in $\U_{CH}$ are fibrant.
\end{lem}
\begin{proof}
All spaces are fibrant in the Quillen model structure, so for all $X$, $\underline{\U}(K,X) \to \underline{\U}(K,\{\ast\}) = \{\ast\}$ is a Serre fibration for all $K \in \bar{\proj}$, so $X$ is fibrant.
\end{proof}

\begin{example}
All spaces are fibrant in $\U_H$. By \cite[Proposition 1.3.13]{Hovey}, to see that the homotopy categories of $\U_{CH}$ and $\U_H$ are not equivalent it suffices to find some cofibrant $X \in \U_{CH}$ and a weak equivalence $X \to Y$ which is not a homotopy equivalence. Take $X$ to be the point $\{\ast\}$, and take $Y$ to be the long line (with your favourite map $X \to Y$, whose image we take as a basepoint): it is well-known that the long line is not contractible. But using the usual smallness results, any map $f$ from a compact Hausdorff space to $Y$ has image contained in some subspace of $Y$ which is homeomorphic to $\mathbb{R}$, so $f$ is homotopic to a constant map. Therefore $$\pi_n(\underline{\U}(K,Y)) = \underline{\U}(\partial\Delta_n,\underline{\U}(K,Y))/\sim \cong \underline{\U}(\partial\Delta_n \times K,Y)/\sim$$ is trivial for all $n$, and hence $\underline{\U}(K,Y)$ is weakly homotopic to $\underline{\U}(K,\{\ast\}) = \{\ast\}$ for all $K \in \bar{\proj}$, and $X \to Y$ is a weak equivalence.
\end{example}

We now define a special class of objects in $\U$: \emph{KW-complexes}, where the K stands roughly for \textit{kompakt}. We want to build spaces, as for CW-complexes, by adding cells in order of dimension.

A KW-complex is a space $X$ constructed in the following way: start with a space $X^0 \in \proj$, which we call the \emph{$0$-skeleton} of $X$. Now inductively define the \emph{$n$-skeleton} $X^n$ of $X$ as the pushout in the category of all topological spaces of $X^{n-1}$ by a map of the form $$\partial\Delta_n \times Y_n \to \Delta_n \times Y_n,$$ some $Y_n \in \proj$. Finally, let $X = \colim_n X^n$. Note that CW-complexes are automatically KW-complexes, but all spaces in $\proj$ are also KW-complexes, by letting $X = X^0$. For example, if $X$ is a totally disconnected, locally compact group, it is in $\proj$, and is a KW-complex, by van Dantzig's theorem.

\begin{rem}
In the expressions $\partial\Delta_n \times Y_n$ and $\Delta_n \times Y_n$, it does not matter whether we are taking the product as $k$-spaces or topological spaces: the two coincide because $\partial\Delta_n$ and $\Delta_n$ are compact Hausdorff.
\end{rem}

A priori, these KW-complexes are just topological spaces; we will see that they are in $\U$.

\begin{lem}
KW-complexes are compactly generated.
\end{lem}
\begin{proof}
The subcategory of compactly generated spaces is closed under colimits in $Top$, and $X^0$ and every $Y_n$ is compactly generated.
\end{proof}

The following is essentially a simplified version of \cite[Proposition A.3]{Hatcher}, adapted for our situation.

Let $X$ be the KW-complex described above. We may think of $\Delta_n$, after a homeomorphism, as the closed unit ball in $\mathbb{R}^n$, and denote the points of $\Delta_n \times Y_n$ by the coordinates $(r,\theta,y)$, where $(r,\theta)$ parametrises the closed unit ball in spherical coordinates and $y \in Y_n$. For each $n$, we have a map $\Phi^n: \Delta_n \times Y_n \to X^n$ which restricts to a homeomorphism from the Hausdorff space $(\Delta_n \setminus \partial\Delta_n) \times Y_n$ to $X^n \setminus X^{n-1}$, in the same way as for CW-complexes.

\begin{lem}
KW-complexes are Hausdorff.
\end{lem}
\begin{proof}
Given $x, x' \in X$, both points must be contained in some $X^n$, with $n$ minimal. We start by constructing disjoint open neighbourhoods of $x$ and $x'$ in $X^n$. There are two cases to consider:
\begin{enumerate}[(i)]
\item $x$ and $x'$ are both in $X^n \setminus X^{n-1}$.
\item $x \in X^n \setminus X^{n-1}$ (without loss of generality) and $x' \in X^{n-1}$.
\end{enumerate}
In the first case, think of $x$ and $x'$ as points in $(\Delta_n \setminus \partial\Delta_n) \times Y_n$, and take disjoint open neighbourhoods there. In the second case, we think of $x$ as a point in $(\Delta_n \setminus \partial\Delta_n) \times Y_n$, with coordinates $(r,\theta,y)$. Then $x$ is contained in the open set $B_{n,r+\varepsilon} \times Y_n$, where $B_{n,r+\varepsilon}$ is the open ball in $\mathbb{R}^n$ of radius $r+\varepsilon$, for some $0 < \varepsilon < (1-r)/2$. On the other hand, $x'$ is in the open set $$(B_{n,1} \setminus B_{n,1-\varepsilon}) \times Y_n \cup X^{n-1},$$ and these open sets have empty intersection.

Now, given disjoint open sets $U_m \ni x$, $V_m \ni x'$ in $X^m$, some $m \geq n$, think of $(\Phi^{m+1})^{-1}(U_m)$ as an open subset of $\partial\Delta_{m+1} \times Y_{m+1}$; in spherical coordinates as before, define $$U_{m+1} = (1 - \delta,1] \times (\Phi^{m+1})^{-1}(U_m) \cup U_m$$ for some small $\delta > 0$. Similarly for $V_{m+1}$. Then $U_{m+1}$ and $V_{m+1}$ are open in $X^{m+1}$ with $U_{m+1} \cap X^m = U_m$ and $V_{m+1} \cap X^m = V_m$. Applying this process inductively and taking $U = \bigcup U_m, V = \bigcup V_m$ gives the required disjoint open sets.
\end{proof}

It is now clear from the construction of KW-complexes that they are cofibrant in $\U_{CH}$. Since all spaces are fibrant, we get from \cite[Theorem 1.2.10]{Hovey} that a weak equivalence between KW-complexes is actually a homotopy equivalence. This may be seen as a generalisation of Whitehead's theorem.

With a little more care, the argument of \cite[Proposition 1.19]{Hatcher2} can be adapted to apply to KW-complexes, by extending partitions of unity to compact spaces of $n$-cells instead of single cells. One can then show:

\begin{prop}
KW-complexes are paracompact.
\end{prop}

Finally, we state the main result that makes KW-complexes useful.

\begin{thm}
\label{kwwe}
For every $Y \in \U$, there is a KW-complex $X$ and a weak equivalence $X \to Y$.
\end{thm}

The proof of this theorem is given in Section \ref{mct}.

We leave to the reader the exercise of defining relative KW-complexes and investigating their properties.

\section{Simplicial enrichment}
\label{se}

We now introduce some category theory. My main source for the category-theoretic ideas in this section is \cite{Riehl}, particularly Chapters 3 and 7 on enriched category theory; see there for details. To fix some notation: if $\C$ is a $\V$-category, that is, enriched over a symmetric monoidal category $\V$, write $\underline{\C}_\V(X,Y)$ for the hom-object in $\V$ (or just $\underline{\C}(X,Y)$ when there is no ambiguity).

From now on, assume that $(\V,\otimes)$ is a closed symmetric monoidal category. Recall that $\C$ is said to be \emph{tensored} (or \emph{copowered}) over $\V$ if, for all $X,Y \in \C$ and $U \in \V$, there is an object $U \odot X$ such that there is a natural isomorphism $\underline{\V}(U,\underline{\C}(X,Y)) \cong \underline{\C}(U \odot X,Y)$. Dually, $\C$ is said to be \emph{cotensored} (or \emph{powered}) over $\V$ if, for all $X,Y \in \C$ and $U \in \V$, there is an object $U \pitchfork Y$ such that there is a natural isomorphism $\underline{\V}(U,\underline{\C}(X,Y)) \cong \underline{\C}(A,U \pitchfork Y)$.

Every category $\C$ with products and coproducts is enriched, tensored and cotensored over $Set$: tensors are coproducts and cotensors are products.

We also write $\{-,-\}$ for weighted limits and $- \star -$ for weighted colimits. If $\C$ is complete, cocomplete, and tensored and cotensored over $\V$, all small weighted limits and colimits exist.

From now on, $\C$ will be a complete and cocomplete $\V$-category, and we write $\ast$ for the terminal object. Much of the work here can be done in categories with fewer limits and colimits; we leave it to the careful reader to check which are needed. However, the model category theorists assume completeness and cocompleteness, so I do not feel guilty about joining them.

Write $\Delta$ as usual for the simplex category. By a \emph{simplicial object} in $\C$ we mean an object in the functor category $\C^{\Delta^{op}}$, which we will also denote $s\V$. We will use the standard simplicial sets $\Delta^n$, $\partial\Delta^n$ and $\Lambda^n_k$; see any source on simplicial sets for definitions. We will also use the barycentric subdivision functor $\sd: sSet \to sSet$ -- details are in \cite[Section III.4]{GJ}, for example.

There is a functor $\disc: \C \to s\C$ which sends $X \in \C$ to the constant simplicial object on $X$, and a right adjoint $-_0: Y \mapsto Y_0$.

For $X \in \C$ and a set $S$, we write $S \cdot X$ or $X \cdot S$ for $\coprod_S X$ (this is the usual tensoring over $Set$). If $S'$ is a simplicial set and $X' \in s\C$, we can define an object $S' \cdot X$ in $s\C$ by $(S' \cdot X)_n = S'_n \cdot X_n$, with the obvious face and degeneracy maps; for $X \in \C$, we let $S' \cdot X = S' \cdot \disc X$.

\begin{lem}
Suppose $X \in s\C$. Then $X$ is isomorphic to the coend $\int^{n \in \Delta} X_n \cdot \Delta^n$, and to the end $\int_{n \in \Delta} \Delta^\bullet_n \pitchfork X_n$ (that is, $X_m = \int_n \Delta^m_n \pitchfork X_n$, naturally in $m$).
\end{lem}
\begin{proof}
We prove the first statement; the second can be proved similarly.

It is enough by the Yoneda lemma to show that $$s\C(X,Y) \cong s\C(\int^{n \in \Delta} X_n \cdot \Delta^n,Y)$$ naturally for all $Y$. Now $s\C(X,Y) = \int_{n \in \Delta}\C(X_n,Y_n)$ and $$s\C(\int^{n \in \Delta} X_n \cdot \Delta^n,Y) = \int_{n \in \Delta}s\C(X_n \cdot \Delta^n,Y),$$ so we just need to check that $\C(X_m,Y_n) \cong s\C(X_m \cdot \Delta^n,Y)$, naturally in $X,Y,m,n$. This can be seen by hand. Given a map $X_m \to Y_n$, think of it as a map from the component of $(X_m \cdot \Delta^n)_n$ corresponding to the unique non-degenerate cell of $\Delta^n_n$ to $Y_n$. This extends to a map $X_m \cdot \Delta^n \to Y$ by composing with face and degeneracy maps. Conversely, given $X_m \cdot \Delta^n \to Y$, we can restrict to the copy of $X_m$ given by the non-degenerate cell.
\end{proof}

My thanks go to Emily Riehl for explaining to me how to show the next two results.

\begin{prop}
\label{simpclosedmonoidal}
$s\V$ is a closed symmetric monoidal category, making the functor $\disc: \V \to s\V$ strong monoidal.
\end{prop}
\begin{proof}
Suppose $U,V,W \in s\V$. The tensor $U \otimes V$ is given by $(U \otimes V)_n = U_n \otimes V_n$. We take $\underline{s\V}(V,W)$ to be given in degree $n$ by the end $\int_m \underline{\V}((\Delta^n \cdot V)_m,W_m)$. The only thing that needs checking is that $s\V(U \otimes V,W) \cong s\V(U,\underline{s\V}(V,W))$:
\begin{align*}
s\V(U \otimes V, W)_k &= \int_n \V((U \otimes V)_n,W_n) \\
&= \int_n \V((\int^m \Delta^m \cdot U_m)_n \otimes V_n,W_n) \\
&= \int_n \int_m \V((\Delta^m \cdot U_m)_n \otimes V_n,W_n) \\
&= \int_n \int_m \V(U_m \otimes (\Delta^m_n \cdot V_n),W_n) \\
&= \int_n \int_m \V(U_m, \underline{\V}((\Delta^m \cdot V)_n,W_n)) \\
&= \int_m \V(U_m, \int_n \underline{\V}((\Delta^m \cdot V)_n,W_n)) \\
&= \int_m \V(U_m,\underline{s\V}(V,W)_m) \\
&= s\V(U,\underline{s\V}(V,W)).
\end{align*}
\end{proof}

Now let $\C$ be enriched, tensored and cotensored over $\V$.

\begin{thm}
\label{tensorcotensor}
$s\C$ is complete and cocomplete, and enriched, tensored and cotensored over $s\V$.
\end{thm}
\begin{proof}
Suppose $X,Y \in s\C$ and $U \in s\V$. The tensor $U \odot X$ is given by $(U \odot X)_n = U_n \otimes X_n$ in degree $n$, with the obvious maps. We take $\underline{s\C}_{s\V}(X,Y)$ to be given in degree $n$ by $\int_m \underline{\C}_\V((\Delta^n \cdot X)_m,Y_m)$, with maps given by the functoriality in $n$. Similarly, the cotensor $U \pitchfork Y$ is given in degree $n$ by $\int_m (\Delta^n \cdot U)_m \pitchfork Y_m$.

There are plenty of axioms to check, but the work boils down to proving two natural isomorphisms: $$\underline{s\C}(X,U \pitchfork Y) \cong \underline{s\C}(U \odot X,Y) \cong \underline{s\V}(U,\underline{s\C}(X,Y)).$$ We start by showing the unenriched versions $$s\C(X,U \pitchfork Y) \cong s\C(U \odot X,Y) \cong s\C(U,\underline{s\C}(X,Y)).$$

\begin{align*}
s\C(U \odot X, Y)_k &= \int_n \C((U \odot X)_n,Y_n) \\
&= \int_n \C((\int^m \Delta^m \cdot U_m)_n \odot X_n,Y_n) \\
&= \int_n \int_m \C((\Delta^m \cdot U_m)_n \odot X_n,Y_n) \\
&= \int_n \int_m \C(U_m \odot (\Delta^m_n \cdot X_n),Y_n) \\
&= \int_n \int_m \V(U_m, \underline{\C}((\Delta^m \cdot X)_n,Y_n)) \\
&= \int_m \V(U_m, \int_n \underline{\C}((\Delta^m \cdot X)_n,Y_n)) \\
&= \int_m \V(U_m,\underline{s\C}(X,Y)_m) \\
&= s\V(U,\underline{s\C}(X,Y)),
\end{align*}
and similarly for the cotensor.

For the enriched isomorphism, we use the Yoneda lemma, together with the fact that, for all $V \in \V$, $V \odot (U \odot X) \cong V \otimes U) \odot X$.
\begin{align*}
s\V(V,\underline{s\C}(U \odot X, Y)) &= s\C(V \odot (U \odot X),Y) \\
&= s\C((V \otimes U) \odot X,Y) \\
&= s\V(V \otimes U,\underline{s\C}(X,Y)) \\
&= s\V(V, \underline{s\V}(U,\underline{s\C}(X,Y)),
\end{align*}
as required; similarly for the cotensor.
\end{proof}

Hence $s\C$ is enriched, tensored and cotensored over $sSet$, and over $\V$ too by \cite[Theorem 3.7.11]{Riehl}. We note for later that, thanks to the adjunction between $\disc$ and $-_0$, the weighted limit $\{S,X\}$ for $S \in sSet$ and $X \in s\C$ is given by $(S \pitchfork X)_0$.

\begin{lem}
For $X \in \C$, $Y \in s\C$, $\underline{s\C}(\disc X,Y)_n = \underline{\C}(X,Y_n)$.
\end{lem}
\begin{proof}
$\underline{s\C}(\disc X,Y)_n = \int_m \underline{\C}(\Delta^n_m \cdot X,Y_m) = \underline{\C}(X,\int_m \Delta^n_m \pitchfork Y_m) = \underline{\C}(X,Y_n)$.
\end{proof}

\section{Model category theory}
\label{mct}

We consider the following situation. Fix a class of objects $\proj$ of $\C$, closed under retracts and coproducts; write $(\C,\proj)$ for the pair. When there is no ambiguity the $\proj$ may be suppressed. We call objects in $\proj$ \emph{projectives}. We call a morphism $f$ in $\C$ \emph{$\proj$-split}, and write $\xrightarrow{\proj}$, if $\C(P,f)$ is surjective for all $P \in \proj$. We say $\C$ \emph{has enough projectives} if, for every $X \in \C$, there is a map $u: P \xrightarrow{\proj} X$ with $P$ in $\proj$. We say $\C$ \emph{has enough functorial projectives} if we have a functor giving a choice of $u$, naturally in $X$.

\begin{rem}
\phantomsection
\label{projectivesinnature}
\begin{enumerate}[(i)]
\item Note that this situation arises quite often in nature. Suppose we have a functor $G: \C \to \D$ and a left adjoint $F: \D \to \C$. The unit of this adjunction is $X \to GF(X)$, so we get $G(X) \to GFG(X)$; the counit is $FG(X) \to X$, so we get $GFG(X) \to G(X)$; it is not hard to check that the composite $G(X) \to GFG(X) \to G(X)$ is the identity. Therefore, for $Y \in \D$, the induced map $$\D(Y,G(X)) \to \D(Y,GFG(X)) \to \D(Y,G(X))$$ is the identity, and in particular $$\C(F(Y),FG(X)) \cong \D(Y,GFG(X)) \to \D(Y,G(X) \cong \C(F(Y),X)$$ is surjective. So if $\bar{\proj}$ is the class of objects in $\C$ in the image of $F$, the counit is $\bar{\proj}$-split, so $\C$ has enough functorial projectives.

More generally, suppose $(\D,\Q)$ has enough functorial projectives. Then $(\C,F(\Q))$ does too: for $Q \in \Q$, a map $f: F(Q) \to X$ corresponds to a map $Q \to G(X)$. This factors through a functorially defined $\Q$-split map $Q \to Q' \xrightarrow{\Q} G(X)$, and by the adjunction this corresponds to maps $F(Q)\to F(Q') \to X$ whose composite is $f$.
\item Given an object $A \in \C$, there is a canonical map $\coprod_{f: P \to A} P \to A$ for each $P \in \proj$, where the coproduct is over all maps $P \to A$. If the coproduct $\coprod_{P \in \P} (\coprod_{f: P \to A} P)$ exists, this induces a canonical map $\coprod_{P \in \P} (\coprod_{f: P \to A} P) \to A$ which is necessarily $\proj$-split. So this question of existence is really the only obstacle here.
\end{enumerate}
\end{rem}

It will sometimes be convenient (for reasons that will become clear later) to consider a subclass $\bar{\proj}$ of $\proj$, such that the closure of $\bar{\proj}$ under retracts and coproducts is $\proj$. We may abuse notation by writing $(\C,\bar{\proj})$ for this pair. In this case, if $\C$ has enough (functorial) projectives we say $\bar{\proj}$ (\emph{functorially}) \emph{generates} $\C$. Clearly a morphism is $\bar{\proj}$-split, in the obvious sense, if and only if it is $\proj$-split.

From now on, to avoid a clash of notations, we will write $CH$ for the class in $\U$ of compact Hausdorff spaces, and $\sqcup CH$ for the class of disjoint unions of compact Hausdorff spaces.

The following is immediate.

\begin{lem}
Suppose we have two composeable morphisms $f$ and $g$.
\begin{enumerate}
\item If $f$ and $g$ are $\xrightarrow{\proj}$, so is the composite $gf$.
\item If $gf$ is $\xrightarrow{\proj}$, so is $g$.
\item $\xrightarrow{\proj}$ is closed under pullbacks.
\end{enumerate}
\end{lem}

\begin{lem}
\phantomsection
\label{regularproj}
\begin{enumerate}[(i)]
\item Every split epimorphism is $\proj$-split.
\item Suppose $\C$ has enough projectives, and that for every $X \in \C$ there is a regular epimorphism $P \xrightarrow{\proj} X$ with $P$ in $\proj$. Then every $\proj$-split map is a regular epimorphism.
\item Suppose $\C$ has enough projectives, and that for every $X \in \C$ there is an epimorphism $P \xrightarrow{\proj} X$ with $P$ in $\proj$. Then every $\proj$-split map is an epimorphism.
\end{enumerate}
\end{lem}
\begin{proof}
\begin{enumerate}[(i)]
\item A section $g$ of $f$ induces a section $\C(P,g)$ of $\C(P,f)$.
\item The argument is the same as \cite[Proposition II.4.2]{Quillen}. We reproduce the relevant part here, \textit{mutatis mutandis}; we will assume the standard facts about effective epimorphisms stated there. Note that we use regular epimorphisms where \cite{Quillen} uses effective epimorphisms: since $\C$ has all kernel pairs, the two definitions are equivalent.

Suppose $f: X \to Y$ is $\proj$-split and choose an epimorphism $u: P \xrightarrow{\proj} X$ with $P$ in $\proj$. Then $fu$ is $\proj$-split, so we are reduced to the case where $X$ is in $\proj$. Now choose a regular epimorphism $v: Q \xrightarrow{\proj} Y$. As $X$ is projective, there is some $\alpha: X \to Q$ such that $v\alpha = f$. As $f$ is $\proj$-split, there is some $\beta: Q \to X$ such that $f\beta = v$. Then $\alpha$ and $\beta$ induce sections in the pullback
\[
\xymatrix{X \times_Y Q \ar[r]^{\pi_2} \ar[d]^{\pi_1} & Q \ar[d]^{v} \\ X \ar[r]^{f} & Y.}
\]
Then $f\pi_1 = v\pi_2$ is a regular epimorphism (because $v$ is), and hence $f$ is.
\item This holds by essentially the same argument, as the reader is invited to check.
\end{enumerate}
\end{proof}

\begin{example}
Consider $(\U,\sqcup CH)$. For $X \in \U$, apply the CH-subspaces functor $d$: the canonical map $d(X) \to X$ is a quotient map, so by Theorem \ref{currying} it is a regular epimorphism. As noted earlier, $d(X) \to X$ is $\proj$-split. So $\U$ has enough functorial projectives, and by the lemma every $\proj$-split map in $\U$ is a regular epimorphism.
\end{example}

\begin{rem}
\label{notepic}
In general, $\proj$-split maps are not epic, without any extra assumptions. To give a non-trivial example, let $\C$ be the category of groups, and $\D$ the coreflective subcategory of torsion groups (the coreflection functor takes the torsion subgroup). So $(\C,\D)$ has enough functorial projectives. The counit is the inclusion map into a group of its torsion subgroup.
\end{rem}

Now fix a subclass $\bar{\proj}$ of $\proj$ whose closure under retracts and coproducts is the whole of $\proj$.

Let $I$ be the class of maps in $s\C$ of the form $\iota \cdot \id_P: \partial \Delta^n \cdot P \to \Delta^n \cdot P$, $P \in \bar{\proj}$, and let $J$ be the class of maps of the form $\iota' \cdot \id_P: \Lambda^n_k \cdot P \to \Delta^n \cdot P$, $P \in \bar{\proj}$, where $\iota: \Lambda^n_k \to \Delta^n, \iota': \partial \Delta^n \to \Delta^n$ are the inclusion maps.

It is worth briefly considering a few equivalent characterisations for a map to have the right lifting property with respect to $I$ or $J$.

\begin{lem}
\label{rightlifting}
The following are equivalent for map $f: X \to Y$ in $s\C$:
\begin{enumerate}[(i)]
\item $f$ has the right lifting property with respect to $I$;
\item $\underline{s\C}_{sSet}(\disc P,f)$ is a trivial fibration of simplicial sets for all $P \in \bar{\proj}$;
\item the induced map $\{\Delta^n,X\} \to \{\partial \Delta^n,X\} \times_{\{\partial \Delta^n,Y\}} \{\Delta^n,Y\}$ is $\proj$-split.
\end{enumerate}
\end{lem}
\begin{proof}
It is immediate from the definitions that $f$ has the right lifting property with respect to $I$ if and only if the induced map $$s\C(\Delta^n \cdot P,X) \to s\C(\partial \Delta^n \cdot P,X) \times_{s\C(\partial \Delta^n \cdot P,Y)} s\C(\Delta^n \cdot P,Y)$$ is surjective for all $n$ and $P$. This is equivalent to (ii) by the hom-tensor adjunction of Theorem \ref{tensorcotensor}; it is equivalent to (iii) by the tensor-cotensor adjunction (recall that being $\proj$-split is equivalent to being $\bar{\proj}$-split).
\end{proof}

$J$ is similar and is left to the reader.

In particular, by (iii), the classes defined by having right lifting properties with respect to $I$ and $J$ depend only on $\proj$, not on $\bar{\proj}$.

We now say a map $f$ in $s\C$ is a fibration if it has the right lifting property with respect to $J$, a weak equivalence if $\underline{s\C}_{sSet}(\disc P,f)$ is a weak equivalence in $sSet$ for all $P \in \bar{\proj}$, and a trivial fibration if it is a fibration and a weak equivalence.

\begin{lem}
\phantomsection
\label{simpfib}
\begin{enumerate}[(i)]
\item Weak equivalences satisfy the two-out-of-three property and are closed under retracts.
\item Trivial fibrations are exactly the maps that have the right lifting property with respect to $I$.
\item Fibrations, trivial fibrations and weak equivalences are preserved by pullbacks.
\end{enumerate}
\end{lem}
\begin{proof}
\begin{enumerate}[(i)]
\item This holds because the same is true for weak equivalences in $sSet$.
\item By the same argument as the previous lemma, using $J$ instead of $I$, $f$ is a trivial fibration if and only if $\underline{s\C}_{sSet}(\disc P,f)$ is a fibration and a weak equivalence for all $P \in \bar{\proj}$, if and only if $\underline{s\C}_{sSet}(\disc P,f)$ is a trivial fibration for all $P \in \bar{\proj}$, if and only if $f$ has the right lifting property with respect to $I$.
\item This is immediate from the corresponding properties in $sSet$.
\end{enumerate}
\end{proof}

It is also worth noting the following result. We define $\diag: (s\C)^{\Delta^{op}} \to s\C$ to be the functor induced by the diagonal map $\Delta \to \Delta \times \Delta$.

\begin{lem}
Suppose $X_\bullet \in (s\C)^{\Delta^{op}}$ is a bisimplicial object such that, for all maps $[k] \to [n]$ in $\Delta$, the induced map $(X_n)_\bullet \to (X_k)_\bullet$ is a weak equivalence. Then the map $(X_0)_\bullet \to \diag X_\bullet$ is a weak equivalence.
\end{lem}
\begin{proof}
This is true for $\C = Set$, by \cite[Lemma 5.3.1]{Hovey}. Now the argument is the same as the previous lemma: just apply `$\underline{s\C}_{sSet}(\disc P,-)$ for all $P \in \bar{\proj}$'.
\end{proof}

\begin{lem}
\label{modelcotensor}
If $f: X \to Y$ is a fibration in $s\C$, and $g: K \to L$ is a cofibration of simplicial sets, the induced map $$L \pitchfork X \to (K \pitchfork X) \times_{K \pitchfork Y} (L \pitchfork Y)$$ is a fibration, which is trivial if $f$ or $g$ is.
\end{lem}
\begin{proof}
This follows from the corresponding fact for simplicial sets, using the fact that $\underline{s\C}_{sSet}(\disc P,-)$ commutes with pullbacks.
\end{proof}

Define $X \in s\C$ to be a \emph{Kan complex} (with respect to $\proj$) if the unique map $X \to \ast$ to the terminal object is a fibration, and write $\Kan(\C,\proj)$, or just $\Kan\C$, for the full subcategory of $s\C$ whose objects are Kan complexes. So $X \in \Kan(\C,\proj)$ if and only if $\{\Delta^n,X\} \to \{\Lambda^n_k,X\}$ is $\proj$-split for all $n,k$.

In particular, for $X \in \C$, the face and degeneracy maps in $\disc X$ are all $\id_X$, and it follows that $\{\Delta^n,\disc X\} \to \{\Lambda^n_k,\disc X\}$ is $\id_X$ too. So $\disc X$ is a Kan complex.

We also define a \emph{fibre sequence} to be a pair of maps $F \to X \xrightarrow{f} Y$ in the pointed category $s(\C_\ast)$ such that $f$ is a fibration and $F$ is the equaliser of $f$ and the trivial map $X \to \ast \to Y$. Of course, for any fibration $f$ in $s\C$ and any map $x:\ast \to X_0$, which we think of as a `choice of basepoint', we can use $f$, and the face maps, to get compatible maps $\ast \to X_n$, $\ast \to Y_n$ for all $n$, and think of $f$ as a fibration in the pointed category. Then we may take the \emph{fibre of $f$ at $x$}, the equaliser of $f$ and $X \to \ast \xrightarrow{fx} Y$, to construct a fibre sequence.

\begin{lem}
\label{Kanfibre}
Suppose $f: X \to Y$ is a fibration in $s\C$.
\begin{enumerate}[(i)]
\item Choose a basepoint $x \in X_0$. The fibre of $f$ at $x$ is a Kan complex.
\item If $X$ is a Kan complex and each map $X_n \to Y_n$ is $\proj$-split, then $Y$ is a Kan complex.
\item If $Y$ is a Kan complex, so is $X$.
\end{enumerate}
\end{lem}
\begin{proof}
The proof for each part is the usual one: apply `$\underline{s\C}_{sSet}(\disc P,-)$ for all $P \in \bar{\proj}$' to reduce to proving it in $sSet$, by noting that $\underline{s\C}_{sSet}(\disc P,-)$ preserves equalisers and hence fibre sequences. The corresponding statements in $sSet$ are proved in \cite[Proposition 7.3, Proposition 7.5]{May2}.
\end{proof}

\begin{thm}
\label{Kan}
With these classes of weak equivalences and fibrations $\Kan\C$ is a category of fibrant objects, in the sense of \cite{Brown}.
\end{thm}
\begin{proof}
Thanks to Lemma \ref{simpfib}, it only remains to show that for every $X \in \Kan\C$ there is some $X^I$ and a factorisation of the diagonal map $X \to X \times X$ into a weak equivalence $X \to X^I$ and a fibration $X^I \to X \times X$.

Let $X^I$ be given (functorially) by $\Delta^1 \pitchfork X$; the factorisation is induced by $\Delta^0 \sqcup \Delta^0 \xrightarrow{\delta_0 \sqcup \delta_1} \Delta^1 \xrightarrow{\sigma_0} \Delta^0$. Then (by the corresponding results for simplicial sets) $$\underline{s\C}_{sSet}(\disc P,X) \to \underline{s\C}_{sSet}(\disc P,\Delta^1 \pitchfork X) \cong \underline{sSet}(\Delta^1,\underline{s\C}_{sSet}(\disc P,X))$$ is a weak equivalence for all $P$, so $X \to \Delta^1 \pitchfork X$ is a weak equivalence. Similarly, $$\underline{sSet}(\Delta^1,\underline{s\C}_{sSet}(\disc P,X) \cong \underline{s\C}_{sSet}(\disc P,\Delta^1 \pitchfork X) \to \underline{s\C}_{sSet}(\disc P,X \times X)$$ is a fibration for all $P$, so $\Delta^1 \pitchfork X \to X \times X$ is.
\end{proof}

\begin{rem}
This does not require that $\C$ have enough projectives; any class $\proj$ will do.
\end{rem}

Immediately from the definitions we get:

\begin{prop}
Suppose $\proj, \Q$ are classes of objects in $\C$, closed under retracts and coproducts. Suppose $\proj \subseteq \Q$. Then weak equivalences (respectively, fibrations) in $(s\C,\Q)$ are weak equivalences (respectively, fibrations) in $(s\C,\proj)$. So $\id_{s\C}$ induces a canonical inclusion functor $i: \Kan(\C,\Q) \to \Kan(\C,\proj)$ which preserves weak equivalences and fibrations.
\end{prop}

The reader may entertain themselves by investigating the poclass of possible classes of projectives. Notably, it has a maximal element:

\begin{examples}
\phantomsection
\label{globalKan}
\begin{enumerate}[(i)]
\item Take $\proj = \ob \C$. Then, by the Yoneda lemma and Lemm \ref{rightlifting}, fibrations in $(s\C,\proj)$ are maps $X \to Y$ such that the induced map $\{\Delta^n,X\} \to \{\Lambda^n_k,X\} \times_{\{\Lambda^n_k,Y\}} \{\Delta^n,Y\}$ is a split epimorphism for all $n,k$. Such maps are also known as \emph{global Kan fibrations}, and objects in $\Kan(\C,\ob \C)$ are known as \emph{global Kan complexes}.
\item At the other extreme, consider the case where $\proj$ is just retracts of the initial object $\emptyset$. Then $\C(\emptyset,X) = \{\ast\}$ for all $X \in \C$, so all maps are $\emptyset$-split. It follows $(\C,\emptyset)$ has enough functorial projectives, and that all maps in $(s\C,\emptyset)$ are trivial fibrations.
\item In the case $\C = Set$, the closure of any class of objects containing a non-empty set under retracts and coproducts is $\ob Set$. So in this case $\Kan(Set,\proj)$ is independent of the choice of $\proj$, and its objects/fibrations are just the usual definition of Kan complexes/fibrations in $sSet$. We will just write $\Kan Set$ for this category.
\end{enumerate}
\end{examples}

In the factorisation of the diagonal map given by the axioms for a category of fibrant objects, we will generally label the maps $X \xrightarrow{s} X^I \rightrightarrows^{d_0}_{d_1} X$, without further comment.

The structure of a category of fibrant objects allows us to define homotopies between maps in this category, in a well-behaved way. We say $f,g:X \to Y$ are \emph{right homotopic}, and write $f \sim g$, if there is a diagram $X \xrightarrow{h} Y^I \rightrightarrows^{d_0}_{d_1} Y$ such that $d_0h = f$ and $d_1h = g$. A standard result of categories of fibrant objects is that $\sim$ gives an equivalence relation on maps $X \to Y$.

As usual, a \emph{right homotopy equivalence} between $X$ and $Y$ is a map $f:X \to Y$ such that there is some $g:Y \to X$ with $gf \sim \id_X$ and $fg \sim \id_Y$.

\begin{lem}
Right homotopy equivalences in $\Kan\C$ are weak equivalences.
\end{lem}
\begin{proof}
The right homotopy $X \xrightarrow{h} \Delta^1 \pitchfork X \rightrightarrows^{d_0}_{d_1} Y$ gives a right homotopy
\begin{align*}
\underline{s\C}_{sSet}(\disc P,X) &\xrightarrow{\underline{s\C}_{sSet}(\disc P,h)} \underline{sSet}(\Delta^1,\underline{s\C}_{sSet}(\disc P,X)) \\
&\rightrightarrows^{\underline{s\C}_{sSet}(\disc P,d_0)}_{\underline{s\C}_{sSet}(\disc P,d_1)} \underline{s\C}_{sSet}(\disc P,Y)
\end{align*}
of simplicial sets for all $P \in \bar{\proj}$. So right homotopy equivalences in $\Kan\C$ give right homotopy equivalences of simplicial sets, which are weak equivalences by \cite[Theorem 1.2.10]{Hovey}.
\end{proof}

We will also want to use \emph{strong deformation retracts}. Following \cite[Definition 1.5.3]{JT}, we say $f:X \to Y$ is a strong deformation retract if it has a retraction $r:Y \to X$ such that $rf = \id_X$, and there is a right homotopy $h: Y \to Y^I$ such that $d_0h=fr$, $d_1h=\id_Y$, and the following diagram commutes, where the map $X \to X^I$ is the canonical one:
\[
\xymatrix{X \ar[r]^{f} \ar[d] & Y \ar[d]^{h} \\
X^I \ar[r]^{f^I} & Y^I.}
\]
Clearly strong deformation retracts are right homotopy equivalences, and hence weak equivalences.

\begin{lem}
\label{factor}
Every map $f:X \to Y$ in a category of fibrant objects can be factored functorially as $f = pi$, where $p: X \times_Y Y^I \to Y$ is a fibration and $i: X \to X \times_Y Y^I$ is a strong deformation retract.
\end{lem}
\begin{proof}
The factorisation is shown in \cite[p.421]{Brown}. Note that functoriality is not mentioned in the reference, but is clear from the construction. The proof given there shows that $p$ is a fibration and that $i$ has a right inverse $r$ which is a trivial fibration. Now $X \times_Y Y^I$ is the mapping cocylinder of $f$, and the required right homotopy between $ir$ and $\id_{X \times_Y Y^I}$ can be constructed by analogy to the topological case; this is left to the reader.
\end{proof}

The last tool we need to introduce for the main result is Kan's $\Ex$-functor $s\C \to s\C$, given by $\Ex(X)_n = \{\sd \Delta^n,X\}$. There is a canonical map $\sd \Delta^n \to \Delta^n$ described in \cite{Low} inducing a map $X \to \Ex(X)$ which we will refer to as $\varepsilon_X$, and we define $\Ex^\lambda$ to be the functor given by letting $\Ex^\lambda(X)$ be the composition of the transfinite sequence $$X \to \Ex(X) \to \Ex^2(X) \to \cdots$$ over all ordinals $< \lambda$. That is, if $\lambda$ is a successor, let $\Ex^\lambda(X)=\Ex(\Ex^{\lambda-1}(X))$; if $\lambda$ is a limit, let $\Ex^\lambda(X)=\colim_{\mu < \lambda} \Ex^\mu(X)$.

The following lemma is well-known in the case $\lambda = \omega$; nothing changes for larger limit ordinals.

\begin{lem}
\label{Ex}
Let $S$ be a simplicial set, and $\lambda$ a limit ordinal.
\begin{enumerate}[(i)]
\item $S \to \Ex(S)$ is a trivial cofibration.
\item $S \to \Ex^\lambda(S)$ is a trivial cofibration.
\item $\Ex^\lambda(S)$ is fibrant.
\end{enumerate}
\end{lem}
\begin{proof}
\begin{enumerate}[(i)]
\item \cite[Theorem III.4.6]{GJ} says the map is a weak equivalence, and all injective maps are cofibrations in $sSet$.
\item Trivial cofibrations are closed under transfinite compositions.
\item A map to $\Ex^\lambda(S)$ from the finite simplicial set $\Lambda^n_k$ factors through $\Ex^\mu(S)$ for some $\mu < \lambda$, by the usual smallness argument. By \cite[Lemma III.4.7]{GJ} this map extends to a diagram
\[
\xymatrix{\Lambda^n_k \ar[r] \ar[d] & \Ex^\mu(S) \ar[d] \\
\Delta^n \ar[r] & \Ex^{\mu+1}(S),}
\]
so the composite of the bottom row with the canonical map $\Ex^{\mu+1}(S) \to \Ex^\lambda(S)$  gives the required lifting in
\[
\xymatrix{\Lambda^n_k \ar[r] \ar[d] & \Ex^\lambda(S) \ar[d] \\
\Delta^n \ar[r] & \ast.}
\]
\end{enumerate}
\end{proof}

The proof of the next proposition is given by exactly the same argument as \cite[Proposition II.4.3]{Quillen}. It is worth noting, for later, exactly what assumptions are needed to make the result work: it is much more general than the statement given in \cite{Quillen}.

\begin{prop}
\label{factorisation}
Suppose we have a class $\M$ of morphisms in $\C$ containing all isomorphisms, closed under composition, such that $gf \in \M$ implies $g \in \M$. Suppose for every $X \in \C$ there is a $P \in \proj$ and a map $P \to X$ in $\M$. Then there is a factorisation of any map $f:Y \to Z$ in $s\C$ into a map $i: Y \to W$ in $I$-cell followed by a map $p: W \to Z$ such that the induced map $\{\Delta^n,W\} \to \{\partial\Delta^n,W\} \times_{\{\partial\Delta^n,Z\}} \{\Delta^n,Z\}$ is in $\M$ for all $n$. If the choice of map $P \to X$ is functorial in $X$, the factorisation is functorial too.
\end{prop}

We are particularly interested in the case when the objects of $\bar{\proj}$ are `small', in the appropriate sense. This is the justification for using $\bar{\proj}$ instead of $\proj$: to use the generalised small object argument of Proposition \ref{smallobject}, our objects must be uniformly small in the sense of satisfying (i) of the proposition, which need not hold for all of $\proj$. Let $\M_\Lambda$ be the class of canonical inclusion maps $X \to X \sqcup P$ in $\C$ with $P \in \bar{\proj}$. Similarly, let $\M_{\Ex}$ be the class of split monomorphisms in $\C$. Suppose there is some cardinal $\kappa$ such that every $P \in \bar{\proj}$ is $\kappa$-small relative to maps in $\M_\Lambda$, and $\kappa$-small relative to maps in $\M_{\Ex}$. Then we say $\bar{\proj}$ is a \emph{class of $\kappa$-small} (\emph{functorial}) \emph{generators} for $\C$. Of course, in a lot of situations, such as pointed categories, the maps $X \to X \sqcup P$ are canonically split, so $\M_\Lambda \subseteq \M_{\Ex}$ in this case.

For the next theorem, we follow the approach of Quillen in \cite[Theorem II.4.4]{Quillen}. Note that Quillen requires that his $\proj$-split maps are exactly the effective epimorphisms, and in (ii) that $\C$ has a (small) set of finite projective generators, so our result is more general; in particular it allows results like Example \ref{sU}.

\begin{thm}
\label{modelcat}
Suppose $\C$ has enough functorial projectives. Let the classes of weak equivalences and fibrations in $s\C$ be as defined above, and let the cofibrations be retracts of maps in $I$-cell. Then $s\C$ is a right-proper simplicial model category if $\C$ satisfies at least one of:
\begin{enumerate}[(i)]
\item every object in $s\C$ is fibrant;
\item $\bar{\proj}$ is a class of $\kappa$-small functorial generators for $\C$.
\end{enumerate}
In the second case, $s\C$ is class-cofibrantly generated with $I$ as the generating cofibrations and $J$ as the generating trivial cofibrations.
\end{thm}
\begin{proof}
We start by presenting the elements of the argument that apply to both cases, before giving the case-specific parts.

Every map factors functorially into a cofibration in $I$-cell followed by a trivial fibration, by Proposition \ref{factorisation}: take $\M$ to be the $\proj$-split epimorphisms.

Therefore, by \cite[Lemma 1.1.9]{Hovey}, the cofibrations are exactly the maps with the left lifting property with respect to the trivial fibrations.
\begin{enumerate}[(i)]
\item Here we are in the situation of Theorem \ref{Kan}, so we can apply Lemma \ref{factor}. So every map $f: A \to B$ factors functorially as $$A \xrightarrow{i} A \times_B \Delta^1 \pitchfork B \xrightarrow{p} B$$ where $i$ is a weak equivalence and $p$ is a fibration. Now factor $i$ functorially, as above, as $qj$ where $j$ is a cofibration and $q$ is a trivial fibration. By two-out-of-three, $j$ is a weak equivalence, so $f = (pq)j$ gives a functorial factorisation into a trivial cofibration followed by a fibration.

It remains to show that the trivial cofibrations are exactly the maps with the left lifting property with respect to the fibrations. If $f$ has the left lifting property with respect to fibrations, it is a cofibration; factoring $f=pi$ as before, by \cite[Lemma 1.1.9]{Hovey} again, $f$ is a retract of $i$, so it is a weak equivalence too, and hence a trivial cofibration.

Conversely, if $f$ is a trivial cofibration, $p$ is a trivial fibration, so $f$ is a retract of $i$, and hence a strong deformation retract. Write $r$ for the retraction $Y \to X$ and $h$ for the right homotopy making $f$ a strong deformation retract. Then use \cite[Lemma II.3.4]{Quillen}: given a fibration $g: W \to Z$, a lifting $u: Y \to W$ in the square
\[
\xymatrix{X \ar[r]^{\alpha} \ar[d]^{f} & W \ar[d]^{g} \\
Y \ar[r]^{\beta} & Z}
\]
is given by taking a lifting $H: Y \to W^I$ in
\[
\xymatrix{X \ar[r]^{s\alpha} \ar[d]^{f} & W^I \ar[d]^{(d_0,g^I)} \\
Y \ar[r]^{(\alpha r,h)} & W \times_{Z} Z^I,}
\]
which exists because $(d_0,g^I)$ is a fibration by Lemma \ref{modelcotensor}, and setting $u=d_1H$.

\item Thanks to Proposition \ref{cofibgen} and Lemma \ref{simpfib}, for $s\C$ to be a model category we just need to show that $J$ permits the functorial generalised small object argument and that $J$-cell $\subseteq I$-cof $\cap W$.

The idea for the functorial generalised small object argument is the same as Lemma \ref{IandJwork}.

Note that if $f$ is a pushout of a map in $J$, $f_n$ is in $\M_\Lambda$ for each $n$. If every $P \in \bar{\proj}$ is $\kappa$-small with respect to maps in $\M_\Lambda$ and $(f_\mu)_{\mu < \lambda}$ is a transfinite sequence of pushouts of maps in $J$ with $\lambda$ having cofinality greater than $\kappa$,
\begin{align*}
\underline{s\C}_{sSet}(\disc P,\colim_{\mu < \lambda} f_\mu)_m &= \C(P,(\colim_{\mu < \lambda} f_\mu)_m) \\
&= \C(P,\colim_{\mu < \lambda} (f_\mu)_m) \\
&= \colim_{\mu < \lambda} \C(P,(f_\mu)_m) \\
&= \colim_{\mu < \lambda} \underline{s\C}_{sSet}(\disc P,f_\mu)_m.
\end{align*}
Hence
\begin{align*}
s\C(\Lambda^n_k \cdot P,\colim_{\mu < \lambda} f_\mu) &= sSet(\Lambda^n_k,\underline{s\C}_{sSet}(\disc P,\colim_{\mu < \lambda} f_\mu)) \\
&= sSet(\Lambda^n_k,\colim_{\mu < \lambda} \underline{s\C}_{sSet}(\disc P,f_\mu)) \\
&= \colim_{\mu < \lambda} sSet(\Lambda^n_k,\underline{s\C}_{sSet}(\disc P,f_\mu)) \\
&= \colim_{\mu < \lambda} s\C(\Lambda^n_k \cdot P,f_\mu),
\end{align*}
because $\Lambda^n_k$ is a finite simplicial set. So $\Lambda^n_k \cdot P$ is $\kappa$-small with respect to pushouts of maps in $J$.

Now suppose $f: X \to Y$ is a map in $s\C$. Consider the pullback $$A_{n,k} = \{\Lambda^n_k,X\} \times_{\{\Lambda^n_k,Y\}} \{\Delta_n,Y\},$$ and write $B_{n,k} \xrightarrow{\proj} A_{n,k}$ for the functorially assigned $\proj$-split map with $B_{n,k} \in \proj$. Finally, define $S(f)$ to be the coproduct over $n,k$ of the maps $\Lambda^n_k \cdot B_{n,k} \to \Delta_n \cdot B_{n,k}$. Then $S(f)$ is in $J$-cell by \cite[Lemma 2.1.13]{Hovey}, and is functorially defined.

Define $e_1$ to be the composite $$\coprod_{n,k} \Lambda^n_k \cdot B_{n,k} \to \coprod_{n,k} \Lambda^n_k \cdot A_{n,k} \to \coprod_{n,k} \Lambda^n_k \cdot \{\Lambda^n_k,X\} \to X,$$ where the first two maps are the canonical ones on the second factor and the identity on the first factor, and the third map comes from the identity on $\{\Lambda^n_k,X\}$, via the adjunction $$s\C(\Lambda^n_k \cdot \{\Lambda^n_k,X\},X) \cong \C(\{\Lambda^n_k,X\},\{\Lambda^n_k,X\}).$$ Similarly, define $e_2$ to be the composite $$\coprod_{n,k} \Delta_n \cdot B_{n,k} \to \coprod_{n,k} \Delta_n \cdot A_{n,k} \to \coprod_{n,k} \Delta_n \cdot \{\Delta_n,Y\} \to Y.$$ Then the square
\[
\xymatrix{\coprod_{n,k} \Lambda^n_k \cdot B_{n,k} \ar[r]^{e_1} \ar[d] & X \ar[d]^{f} \\ \coprod_{n,k} \Delta_n \times B_{n,k} \ar[r]^{e_2} & Y}
\]
commutes, as required, giving the morphism of maps $t_f$. This is natural by construction.

There is a natural bijection between commutative squares of the form
\[
\xymatrix{\Lambda^n_k \cdot P \ar[r] \ar[d] & X \ar[d]^{f} \\ \Delta_n \cdot P \ar[r] & Y,}
\]
$P \in \bar{\proj}$, and elements of the set $$s\C(\Lambda^n_k \cdot P,X) \times_{s\C(\Lambda^n_k \cdot P,Y)} s\C(\Delta_n \cdot P,Y);$$ this set is naturally isomorphic to $$\C(P,\{\Lambda^n_k,X\}) \times_{\{\Lambda^n_k,Y\}} \{\Delta_n,Y\}) = \C(P,A_{n,k}).$$ So each such commutative square identifies a map $P \to A_{n,k}$, which lifts to $g: P \to B_{n,k}$, giving a diagram
\[
\xymatrix{\Lambda^n_k \cdot P \ar[r] \ar[d] & \Lambda^n_k \cdot B_{n,k} \ar[r] \ar[d] & \coprod_n \Lambda^n_k \cdot B_{n,k} \ar[r]^{e_1} \ar[d] & X \ar[d]^{f} \\ \Delta_n \cdot P \ar[r] & \Delta_n \cdot B_{n,k} \ar[r] & \coprod_n \Delta_n \cdot B_{n,k} \ar[r]^{e_2} & Y.}
\]
This diagram commutes and gives the required factorisation, so $J$ permits the small object argument.

A map of simplicial sets of the form $\Lambda^n_k \to \Delta^n$ may be written as a finite composition of pushouts by maps of the form $\partial \Delta^m \to \Delta^m$; indeed, any injective map of simplicial sets is a relative $I$-cell complex in the standard model structure, by \cite[Proposition 3.2.2]{Hovey}. Tensoring with $P \in \bar{\proj}$, we find that maps in $J$ are in $I$-cell. Hence $J$-cell $\subseteq I$-cell $\subseteq I$-cof.

To see $J$-cell $\subseteq W$, we follow \cite[Section II.4]{Quillen}. First, for any $X \in \C$ and $Y \in s\C$, $\underline{s\C}_{sSet}(\disc X,\Ex(Y)) = \Ex(\underline{s\C}_{sSet}(\disc X,Y)$:
\begin{align*}
\underline{s\C}_{sSet}(\disc X,\Ex(Y))_n &= \C(X,(\Ex(Y)_n) \\
&= \{\sd \Delta^n,\underline{s\C}_{sSet}(\disc X,Y)\} \\
&= \Ex(\underline{s\C}_{sSet}(\disc X,Y)_n.
\end{align*}
Note that the canonical map $\sd\Delta^n \to \Delta^n$ has a section, by the remarks at the beginning of \cite[Section 2]{Low}, so the map $Y_n \to \Ex(Y)_n = \{\sd\Delta^n,Y\}$ is split monic for each $n$. Now, since every $P \in \bar{\proj}$ is $\kappa$-small with respect to maps in $\M_{\Ex}$, we get $$\underline{s\C}_{sSet}(\disc P,\Ex^\lambda(Y)) = \Ex^\lambda(\underline{s\C}_{sSet}(\disc P,Y))$$ for any $\lambda$ with cofinality greater than $\kappa$. Thanks to Lemma \ref{Ex}, it follows that $\Ex^\lambda(Y)$ is fibrant and $Y \to \Ex^\lambda(Y)$ is a weak equivalence.

Suppose $f: X \to Y$ is in $J$-cell, so it has the left lifting property with respect to fibrations. Then we may construct a lift $u:Y \to \Ex^\lambda(X)$ in
\[
\xymatrix{X \ar[r]^{\varepsilon_X} \ar[d]^{f} & \Ex^\lambda(X) \ar[d] \\
Y \ar[r] & \ast,}
\]
and a lift $H: B \to \Delta^1 \pitchfork \Ex^\lambda(B)$ in
\[
\xymatrix{X \ar[r]^{s\varepsilon_Yf} \ar[d]^{f} & \Delta^1 \pitchfork \Ex^\lambda(Y) \ar[d]^{(d_0,d_1)} \\
Y \ar[r]^{(\varepsilon_Y,\Ex^\lambda(f)u)} & \Ex^\lambda(Y) \times \Ex^\lambda(Y),}
\]
to get $uf = \varepsilon_X$, $\Ex^\lambda(f)u \sim \varepsilon_Y$ and $\varepsilon_Yf = \Ex^\lambda(f)\varepsilon_X$. Hence, writing $\gamma$ for the canonical functor $sSet \to Ho(sSet)$ to the homotopy category, $\gamma\underline{s\C}_{sSet}(\disc P,u)$ is an isomorphism for all $P \in \bar{\proj}$, so $\underline{s\C}_{sSet}(\disc P,u)$ is a weak equivalence for all $P$ (because $sSet$ is a model category), so $u$ is a weak equivalence. Therefore $f$ is a weak equivalence.
\end{enumerate}
That the resulting model structure in both these cases is right-proper is immediate from the definitions. It is simplicial by \cite[Proposition III.2.3]{Quillen}, because it satisfies Quillen's SM7(a) axiom by Lemma \ref{modelcotensor}.
\end{proof}

\begin{rem}
If we require that $\bar{\proj}$ generates $\C$, but not functorially, then we still get a model structure on $s\C$, as long as we drop the requirement that the factorisations be functorial.
\end{rem}

We will abuse notation again by writing $(s\C,\bar{\proj})$ for the category $s\C$ equipped with the model structure arising from the class $\bar{\proj}$ of objects in $\C$, via Theorem \ref{modelcat}.

\begin{prop}
\label{monoidal}
Suppose the closed symmetric monoidal category $(\V,\bar{\proj})$ satisfies the conditions of Theorem \ref{modelcat}. Suppose the monoidal identity $1$ is in $\bar{\proj}$, and that, for $P,Q \in \bar{\proj}$, we have $P \otimes Q \in \proj$. Then $s\V$, with the model structure described above, is a monoidal model category, in the sense of \cite[Definition 4.2.6]{Hovey}.
\end{prop}
\begin{proof}
The monoidal identity in $s\V$ is $\disc 1 = \Delta^0 \pitchfork 1$, so it is cofibrant. So we need to show, for a cofibration $f:W \to X$ and a fibration $g:Y \to Z$, that the induced map $$\underline{s\V}_{s\V}(X,Y) \to \underline{s\V}_{s\V}(X,Z) \times_{\underline{s\V}_{s\V}(W,Z)} \underline{s\V}_{s\V}(W,Y)$$ is a fibration that is trivial if $f$ or $g$ is; equivalently, that
\begin{align*}
&s\V(\disc P \otimes X,Y) \cong s\V(\disc P,\underline{s\V}_{s\V}(X,Y)) \\
\to &s\V(\disc P,\underline{s\V}_{s\V}(X,Z)) \times_{s\V(\disc P,\underline{s\V}_{s\V}(W,Z))} s\V(\disc P,\underline{s\V}_{s\V}(W,Y)) \\
\cong &s\V(\disc P \otimes X,Z)) \times_{s\V(\disc P \otimes W,Z))} s\V(\disc P \otimes W,Y))
\end{align*}
is a fibration, that is trivial if $f$ or $g$ is, for all $P \in \bar{\proj}$: since $s\V$ is a simplicial model category, for this it suffices to show $\disc P \otimes f$ is a cofibration that is trivial if $f$ is.

Since $\disc P \otimes -$ commutes with retracts and colimits (because it is a left adjoint), we may reduce to the case where $f$ is in $I$. Then $$\partial \Delta^n \cdot (P \otimes Q) \cong \disc P \otimes (\partial \Delta^n \cdot Q) \to \disc P \otimes (\Delta^n \cdot Q) \cong \Delta^n \cdot (P \otimes Q)$$ is in $I$-cell, by the hypothesis. Argue similarly using $J$ for the case where $f$ is trivial.
\end{proof}

Similarly, \cite[Definition 4.2.18]{Hovey} describes $\V$-model categories, and similarly one can show:

\begin{prop}
\label{Vmod}
Suppose $(\V,\bar{\proj})$, and the tensored and cotensored $\V$-category $(\C,\bar{\Q})$, satisfy the conditions of Theorem \ref{modelcat}. Suppose the monoidal identity $1$ is in $\bar{\proj}$, and that, for $P \in \bar{\proj}, Q \in \bar{\Q}$, we have $P \odot Q \in \Q$. Then $s\C$, with the model structure described above, is an $s\V$-model category.
\end{prop}

\begin{prop}
Suppose we have categories $\C,\D$ and functors $F:\C \rightleftarrows \D:G$ with $F$ left adjoint to $G$. Suppose $(\C,\bar{\proj})$ and $(\D,\bar{\Q})$ satisfy the conditions of Theorem \ref{modelcat}, and that $F(\bar{\proj}) \subseteq \Q$. Then we get a Quillen adjunction $sF: s\C \rightleftarrows s\D :sG$ between the resulting model categories.
\end{prop}
\begin{proof}
The adjunction extends to an adjunction between $s\C$ and $s\D$; the hypothesis ensures it is Quillen, by the same argument as Proposition \ref{monoidal}.
\end{proof}

The reader may also construct a true statement, along the same lines, about $s\V$-Quillen adjunctions when $\C$ and $\D$ are tensored and cotensored $\V$-categories.

In particular, suppose $\C=\D$ and $F,G$ are $\id_\C$. Suppose $\bar{\proj}$ and $\bar{\Q}$ are two classes of objects in $\C$ satisfying the conditions of Theorem \ref{modelcat}, with $\bar{\proj} \subseteq \bar{\Q}$. Then as well as a Quillen adjunction, we get a mixed model structure combining the two, by \cite[Theorem 4.1.1]{MS}.

\begin{examples}
\label{sU}
The following are all class-cofibrantly generated right-proper simplicial model categories.
\begin{enumerate}[(i)]
\item $(Set,\text{finite sets})$ satisfies condition (ii) of Theorem \ref{modelcat}, and the conditions of Proposition \ref{monoidal} (finite sets are finite in $Set$). So $sSet$ is a monoidal model category. (Unfortunately, this is circular: we assumed the result to prove the theorem.)
\item Similarly, $(\U, CH)$ makes $s\U$ a monoidal model category. We have seen that the $CH$ is finite relative to closed inclusions, and that the CH-subspaces functor gives the functorial choice of $\proj$-split map from a projective. Readers should satisfy themselves that maps in $\M_\Lambda$ and $\M_{\Ex}$ are closed inclusions: recall that split monomorphisms are reguar, and closed monomorphisms in $\U$ are exactly the closed inclusions.
\item $\U$ with finite discrete spaces also works, in the same way as (i).
\item Consider $\U Grp$, the category of group objects in $\U$, with the class $\bar{\proj}$ given by the free $k$-groups $F(X)$ on spaces $X \in CH$. The adjunction of $F$ with the forgetful functor to $\U$ shows objects in $\bar{\proj}$ are finite relative to closed inclusions, and for $G \in \U Grp$ every map $F(X) \to G$ factors through the canonical (functorial) map $\coprod_{X \subseteq G, X \in CH} F(X) \to G$. This makes $s\U Grp$ a simplicial model category. In fact, $\U Grp$ is enriched, tensored and cotensored over $\U$, by giving $\underline{\U Grp}(G,H)$ the $k$-compact-open topology, and defining, for $X \in \U$, $X \odot G = F(X) \sqcup G$ and $X \pitchfork G = \underline{\U}(X,G)$, with the group structure induced by multiplication in $G$. So by Proposition \ref{Vmod}, $\U Grp$ is an $s\U$-model category.
\item Consider $\U$ with the class of disjoint unions of at most $\kappa$ compact Hausdorff spaces, for some cardinal $\kappa$. Such spaces are $\kappa$-small relative to closed inclusions (check!), and the CH-subspaces functor gives the functorial choice of $\proj$-split map from a projective, as before. This gives the same model structure on $s\U$ as (ii), with a different class of small generators.
\item Suppose $\C$ is a cartesian closed category. Consider the category $G$-$\C$ of $G$-objects (that is, objects $X$ with a morphism $G \times X \to X$) and $G$-maps, for $G \in \C Grp$. We can alternatively think of $G$ as a one-object $\C$-category, with morphisms given by $G$, so that $G$-$\C$ is the enriched functor category $\C^G$. Just like the category of $G$-sets for an abstract group, $G$-$\C$ is a cartesian closed category and the forgetful functor to $\C$ has a left adjoint, $X \mapsto G \times X$. Suppose $\C$ has a class of $\kappa$-small generators $\bar{\proj}$. From the adjunction, it follows (see \cite[Theorem 7.8]{GJ}) that $(s(G$-$\C), G \times \bar{\proj})$ satisfies the conditions of Theorem \ref{modelcat}, where $G \times \bar{\proj}$ is the class of $G$-objects of the form $G \times X$ where $X \in \bar{\proj}$ and $G$ acts by left-multiplication on the first variable; this makes $s(G$-$\C)$ into a monoidal model category.
\end{enumerate}
\end{examples}

The reader is invited to expand this list.

\begin{rem}
Suppose $(\C,\bar{\proj})$ has enough projectives. Consider the pair $(s\C,\Q)$, where $\Q$ is the class of objects $X$ in $s\C$ such that the map $\disc \emptyset \to X$ from the initial object is in $I$-cell. The factorisation constructed in Proposition \ref{factorisation} shows that $s\C$ has enough functorial projectives. Indeed, trivial fibrations are $\Q$-split because they have the right lifting property with respect to cofibrations. But even if the objects of $\bar{\proj}$ satisfy condition (ii) of Theorem \ref{modelcat}, it is not clear how to choose a subclass of $\Q$ generating $s\C$ whose objects satisfy it.
\end{rem}

\begin{lem}
For a cofibrant object $X$ in $(s\C,\bar{\proj})$, satisfying the conditions of Theorem \ref{modelcat}, $X_n \in \proj$ for all $n$.
\end{lem}
\begin{proof}
For a cofibrant object $X$, we may write $X$ as a retract of a composition of pushouts of maps $f_\lambda: W^\lambda \to W^{\lambda+1}$ in $I$, starting from $\emptyset$. We argue by induction on $\lambda$. Retracts, compositions and pushouts are calculated levelwise, so we may fix $n$ and suppose $W^\lambda_n \in \proj$. Then $W^{\lambda+1}_n$, the pushout of $W^\lambda_n$ by $(\Lambda^m_k)_n \cdot P \to (\Delta^m)_n \cdot P$, is easily seen to be the coproduct of $W^\lambda_n$ with $|(\Delta^m)_n \setminus (\Lambda^m_k)_n|$ copies of $P$, as required.
\end{proof}

We now return to considering $(s\U,CH)$ and $\U_{CH}$. There is a \emph{singular complex} functor $\Sing: \U \to s\U$, $\Sing(X) = \underline{\U}(\Delta_\bullet,X)$, where $\Delta_n$ is the topological $n$-simplex. We also have the abstract singular complex functor $\Sing^{abs}: \U \to sSet$. Thanks to the cartesian closure of $\U$, we have $$\Sing^{abs}(\underline{\U}(K,-)) = \underline{s\U}_{sSet}(\disc K,\Sing(-)),$$ so $\Sing$ preserves fibrations and trivial fibrations. There is also a \emph{geometric realisation} functor $|-|: s\U \to \U$ given by the coend $|X| = \int^n \Delta_n \times X_n$, and an adjunction
\begin{align*}
\U(|X|,Y) &= \U(\int^n \Delta_n \times X_n,Y) &&= \int_n \U(X_n,\underline{U}(\Delta_n,Y)) \\
&= \int_n \U(X_n,\Sing(Y)_n) &&= s\U(X,\Sing Y).
\end{align*}
We immediately get from the definitions that $|-|$ preserves cofibrations and trivial cofibrations, so these functors give a Quillen adjunction between $s\U$ and $\U$. In particular $\Sing(X)$ is fibrant for all $X \in \U$, and hence by Ken Brown's lemma \cite[Lemma 1.1.12]{Hovey} $\Sing$ preserves weak equivalences.

\begin{rem}
\phantomsection
\label{structuresonsU}
\begin{enumerate}[(i)]
\item In fact $\Sing(X)$ is a global Kan complex: a retraction of $\Lambda_n \to \Delta_n$ induces a section of $\underline{\U}(\Delta_n,X) \to \underline{\U}(\Lambda_n,X)$.
\item We can see $\Sing$ and $|-|$ as giving a correspondence between $(s\U,CH)$ and $\U_{CH}$ (that is, we have a Quillen adjunction and the same generating class of projectives). We have a similar correspondence between $(s\U,\text{finite discrete spaces})$ and $\U_Q$. The fibrations (respectively, weak equivalences) in the Hurewicz model structure $\U_H$ are exactly the maps $f$ in $\U$ such that $\underline{\U}(X,f)$ is a fibration (respectively, weak equivalence) in $\U_Q$ for all $X \in \U$. There ought (on aesthetic grounds) to be a model structure on $s\U$ whose fibrations and weak equivalences come from $(\U,\ob \U)$. But the existence of such a model structure is not clear.
\end{enumerate}
\end{rem}

Geometric realisation from $s\U$ seems to be rather more delicate than geometric realisation from $sSet$. A useful reference is \cite[Section 11]{May}. In particular, \cite[Theorem 11.5]{May} gives:

\begin{prop}
Geometric realisation commutes with finite products.
\end{prop}

Hence $|-|$ is strong monoidal. Therefore, by \cite[Theorem 3.7.11]{Riehl}, $\U$ is enriched, tensored and cotensored over $s\U$. It follows that $\U$ is a monoidal $s\U$-model category in the sense of \cite[Definition 4.2.20]{Hovey}.

The analogous pair of functors, geometric realisation $|-|: sSet \to \U$ and the abstract singular complex functor $\Sing^{abs}: \U \to sSet$, are well-known to give a Quillen equivalence between the two categories. We get the following result.

\begin{prop}
The composite $|-| \circ \Sing$ is naturally homotopic to the identity, .
\end{prop}
\begin{proof}
As far as I know, this is not in the published literature, but a proof is given in Strickland's answer to the question at \cite{Strickland2}.
\end{proof}

We should mention at this point that it is not clear whether or not $|-|$ preserves weak equivalences in $(s\U,CH)$. So in general we will want to replace it with its \emph{total left derived functor} $L|-|$ (see \cite[Definition 1.3.6]{Hovey}) to get a functor on the homotopy categories. Explicitly, $L|-| = |-| \circ Q$, where $Q$ is a cofibrant replacement functor on $(s\U,CH)$; this gives an adjunction $$L|-|: Ho(s\U,CH) \rightleftarrows Ho(\U_{CH}): Sing.$$ It is clear from the definitions that applying $L|-|$ always produces a KW-complex.

$L|-| \circ \Sing$ need not be homotopic to the identity (indeed, it is not well-defined up to homotopy). On the other hand, we do have:

\begin{prop}
\label{geomsingwe}
$L|-| \circ \Sing$ is weakly equivalent in $\U_{CH}$ to the identity.
\end{prop}
\begin{proof}
For $X \in \U$, we have to show that the induced map $q: \underline{\U}(K,L|Sing(X)|) \to \underline{\U}(K,|Sing(X)|)$ is a weak homotopy equivalence for all compact Hausdorff $K$.

Now we have a homotopy equivalence $\underline{\U}(K,|Sing(X)|) \to \underline{\U}(K,X)$, so given a map $f: S^n \to \underline{\U}(K,|Sing(X)|)$ (where $S^n$ is the $n$-sphere), choose a homotopic map $S^n \to \underline{\U}(K,X)$, or equivalently a map $S^n \times K \to X = Sing(X)_0$. Since $S^n \times K$ is compact, this lifts to a map $S^n \times K \to Q(Sing(X))_0$ (because $Q(Sing(X)) \to Sing(X)$ is a trivial fibration). Composing with $Q(Sing(X))_0 \to L|Sing(X)|$, we get a map $S^n \to \underline{\U}(K,L|Sing(X)|)$ whose image in $\underline{\U}(K,|Sing(X)|)$ is homotopic to $f$ -- possibly with a different basepoint, but this is dealt with by standard techniques. Therefore $$\pi_n(\underline{\U}(K,L|Sing(X)|),x) \to \pi_n(\underline{\U}(K,|Sing(X)|),q(x))$$ is surjective, for all $x$ and $n$.

To see this map is injective, suppose we have based maps $f,g: S^n \to \underline{\U}(K,L|Sing(X)|)$, for some basepoint $x$, such that $q(f), q(g)$ are homotopic. We have a commutative diagram
\[
\xymatrix{\underline{\U}(K,L|Sing(X)|) \ar[r] \ar[d]^{q} & \underline{\U}(K,L|Sing(|Sing(X)|)|) \ar[d]^{q'} \\
\underline{\U}(K,|Sing(X)|) \ar[r] & \underline{\U}(K,|Sing(|Sing(X)|)|),}
\]
where the top and bottom maps are weak equivalences by the previous proposition, so we will identify $f$ and $g$ with their images in $\underline{\U}(K,L|Sing(|Sing(X)|)|)$ and show they are homotopic there. Identifying $q(f), q(g)$ with their images in $\underline{\U}(K,|Sing(|Sing(X)|)|)$, we see that it is homotopic to a constant map there via a map $S^n \times [0,1] \times K \to |Sing(|Sing(X)|)|$ whose image is contained in the $0$-skeleton $Sing(|Sing(X)|)_0 = |Sing(X)|$. So this map lifts to one $S^n \times [0,1] \times K \to Q(Sing(|Sing(X)|))_0$, which after geometric realisation gives a homotopy between $h_0: S^n \times {0} \to \underline{\U}(K,L|Sing(|Sing(X)|)|)$ and $h_1: S^n \times {1} \to \underline{\U}(K,L|Sing(|Sing(X)|)|)$.

Now we show $f$ is homotopic to $h_0$; $g$ is homotopic to $h_1$ similarly and the result will follow. Both $f$ and $h_0$ map into the subspace $$\underline{\U}(K,Q(Sing(|Sing(X)|))_0) = \underline{s\U}(\disc K,Q(Sing(|Sing(X)|)))_0,$$ and $q(f) = q(h_0)$ in $\underline{\U}(K,Sing(|Sing(X)|)_0) = \underline{s\U}(\disc K,Sing(|Sing(X)|))_0$; therefore, since $$\underline{s\U}(\disc S^n \times K,Q(Sing(|Sing(X)|))) \to \underline{s\U}(\disc S^n \times K,Sing(|Sing(X)|))$$ is a Serre fibration, $\underline{s\U}(\disc S^n \times K,Q(Sing(|Sing(X)|)))$ has a $1$-cell giving a homotopy from $f$ to $h_0$.

\end{proof}

Theorem \ref{kwwe} now follows: given a space $X$, we can find a weakly equivalent KW-complex by applying $L|-| \circ \Sing$.

$L|-|$ also has the nice property that it sends weak equivalences $f: X \to Y$ to bona fide homotopy equivalences: $L|f|$ is a weak equivalence between cofibrant, fibrant objects, because all objects are fibrant in $\U_{CH}$, so it is a homotopy equivalence by \cite[Theorem 1.2.10]{Hovey}.

We will see later in Proposition \ref{notwe} that the pair $(|-|,\Sing)$ is not a Quillen equivalence, in contrast to the case of simplicial sets -- but see Theorem \ref{tpd} below.

\section{Homotopy groups in a regular category}
\label{reghomotopy}

The right place to study homotopy groups is in the framework of regular categories, and the right place to give more detail on these categories is here. In this section, $\C$ will be a complete and cocomplete regular category; $\twoheadrightarrow$ will denote a regular epimorphism and $\rightarrowtail$ will denote a monomorphism. A \emph{subobject} $Y$ of $X$ will mean a map $Y \rightarrowtail X$.

The structure of a regular category is enough to prove various useful diagram lemmas, as in \cite{Barr}. In particular, morphisms $f:X \to Y$ factorise uniquely, up to unique isomorphism, into $X \twoheadrightarrow Z \rightarrowtail Y$. We write $\im(f)$ for this $Z$ and call it the image of $f$. We also define preimages: given $f: X \to Y$ and a subobject of $Y$, $f^{-1}(Z)$ is the pull-back of $X \xrightarrow{f} Y \leftarrowtail Z$. Finally, intersections: if $Y,Z$ are subobjects of $X$, we write $Y \cap Z$ for the pull-back of $Y \rightarrowtail X \leftarrowtail Z$.

Given objects $X,Y$ in a regular category $C$, we define a \emph{relation} from $X$ to $Y$ to be a subobject of $X \times Y$. An important class of examples is that of the graphs of morphisms in $C$: given $f:X \to Y$, its graph is the equaliser of $X \times Y \rightrightarrows^{fp_1}_{p_2} Y$, where $p_1,p_2$ are the projection maps. We will abuse notation by identifying a morphism of $C$ with its graph in this context. Another is the trivial relation $1_X$ from $X$ to $X$, given by the diagonal map $X \to X \times X$.

Relations can be composed: given relations $R$ from $X$ to $Y$ and $S$ from $Y$ to $Z$, we define the composite $SR$ from $X$ to $Z$ to be $$\im(p^{-1}_{X \times Y}(R) \cap p^{-1}_{Y \times Z}(S) \rightarrowtail X \times Y \times Z \twoheadrightarrow X \times Z),$$ where $p_{X \times Y}$ is the projection $X \times Y \times Z \to X \times Y$, and similarly for the others.

\begin{rem}
We write the composites of relations $R$ and $S$ as $SR$ to maintain the connection to morphisms. This convention is not universally followed: \cite{SC} writes the same composite as $RS$.
\end{rem}

Given a relation $R$ from $X$ to $Y$, we write $R^\circ$ for the relation from $Y$ to $X$ given by composing $R \to X \times Y$ with the map $X \times Y \to Y \times X$ which swaps the factors. $-^\circ$ is thus an involution on the relations in $C$. Given relations $R,S$ from $X$ to $Y$, we write $R \leq S$ if $R \to X \times Y$ factors through $S \to X \times Y$. Categories of relations, equipped with $-^\circ$ and $\leq$, can be profitably studied in their own right; see \cite{SC} for details and for properties of $-^\circ$ and $\leq$ which we will use later. In particular, a relation from $X$ to $Y$ is the graph of a map $f$ if $1_X \leq f^\circ f$ and $ff^\circ \leq 1_Y$.

We now define an equivalence relation in $C$ to be a relation $R$ from an object $X$ to itself such that $1_X \leq R$, $R \leq R^\circ$ and $RR \leq R$.

\begin{rem}
These three conditions correspond to the usual requirement in $Set$ that equivalence relations be reflexive, symmetric and transitive, respectively. It does not seem to be as widely known as it should be that equivalence relations in $Set$ can alternatively be characterised by two axioms: a relation $\sim$ on a set $X$ is an equivalence relation if and only if it is reflexive and $x\sim y$, $x\sim z$ implies $y\sim z$. It does not seem to be known that the same reduction to two axioms can be done for the more general equivalence relations defined here, but the following proposition shows that it can. We prefer the two-axiom definition on the basis that $2<3$.
\end{rem}

\begin{prop}
\label{relation}
A relation $R$ from $X$ to $X$ is an equivalence relation if and only if $1_X \leq R$ and $RR^\circ \leq R$.
\end{prop}
\begin{proof}
Suppose $R$ is an equivalence relation. Then $R^\circ \leq R^{\circ \circ} = R$ so $RR^\circ \leq RR \leq R$. Conversely, suppose $1_X \leq R$ and $RR^\circ \leq R$. $R$ is symmetric: $R^\circ = 1_XR^\circ \leq RR^\circ \leq R$, so $R = R^{\circ \circ} \leq R^\circ$. $R$ is transitive: by symmetry, $RR \leq RR^\circ \leq R$.
\end{proof}

Given a morphism $f$ in a regular category $C$, the \emph{kernel pair} of $f$ is the pullback of $f$ with itself. Then $C$ is called \emph{Barr-exact}, or just \emph{exact} when it is clear from the context, if all equivalence relations in $C$ are kernel pairs.

\begin{lem}
\label{Lawvereregular}
Let $C$ be a regular (respectively, Barr-exact) category. For a Lawvere theory $T$, $C^T$ is a regular (respectively, Barr-exact) category.
\end{lem}
\begin{proof}
\cite[Theorem 5.11]{Barr}
\end{proof}

Suppose $\C$ is a regular category. Then it is possible to study the homotopy theory of internal Kan complexes in $\C$. This work was initiated by \cite{vO}, but for a modern treatment see \cite{Low}. Here we define a map $X \to Y$ in $s\C$ to be an internal fibration if the induced map $\{\Delta^n,X\} \to \{\Lambda^n_k,X\} \times_{\{\Lambda^n_k,Y\}} \{\Delta^n,Y\}$ is a regular epimorphism for all $n,k$, and an internal trivial fibration if the induced map $\{\Delta^n,X\} \to \{\partial\Delta^n,X\} \times_{\{\partial\Delta^n,Y\}} \{\Delta^n,Y\}$ is a regular epimorphism for all $n$. Restricting now to fibrant objects gives a category which we will call $\Kan(\C,reg)$. As before, $\disc X \in \Kan(\C,reg)$ for all $X$.

Weak equivalences in $\Kan(\C,reg)$ are given by \emph{Dugger-Isaksen} weak equivalences. Define $D^{n+1}$ to be the pushout in $sSet$ of $$\partial\Delta^n \leftarrow \partial\Delta^n \times \Delta^1 \to \Delta^n \times \Delta^1,$$ where the first map is projection to the first factor and the second is inclusion on the first factor. Write $j_0,j_1$ for the compositions of the two maps $\Delta^n \to \Delta^n \times \Delta^1$ induced by the two maps $\Delta^0 \to \Delta^1$ with the map $\Delta^n \times \Delta^1 \to D^{n+1}$. Then a morphism $f:X \to Y$ in $\Kan(\C,reg)$ is called a Dugger-Isaksen weak equivalence if the map $DI_n(f)$ in the diagram
\[
\xymatrix{\hbox to 3em{\hss$\Delta^n,X\} \times_{\{\Delta^n,Y\}} \{D^{n+1},Y\}$\hss} \ar@{}[r]^(.85){}="a"^(1.75){}="b" \ar "a";"b" \ar[dd] \ar[dr]^(.65){DI_n(f)} & & \{D^{n+1},Y\} \ar[dd]^(.65){\{j_1,Y\}} \ar[dr]^{\{j_0,Y\}} \\
& \hbox to 3em{\hss$\{\partial\Delta^n,X\} \times_{\{\partial\Delta^n,Y\}} \{\Delta^n,Y\}$\hss} \ar@{}[r]^(.85){}="a"^(1.75){}="b" \ar "a";"b" \ar[dd] & & \{\Delta^n,Y\} \ar[dd] \\
\{\Delta^n,X\} \ar[rr] \ar[dr] & & \{\Delta^n,Y\} \ar[dr] \\
& \{\partial\Delta^n,X\} \ar[rr] & & \{\partial\Delta^n,Y\}}
\]
is a regular epimorphism for all $n$. The following proposition summarises the results we will need:

\begin{prop}
\phantomsection
\label{DI}
\begin{enumerate}[(i)]
\item Internal fibrant objects in $sSet$ are exactly the Kan complexes in $sSet$, so $\Kan(Set,reg) = \Kan Set$. Dugger-Isaksen weak equivalences, internal fibrations and internal trivial fibrations in $\Kan Set$ are exactly the usual weak homotopy equivalences, Kan fibrations and trivial Kan fibrations in $\Kan Set$.
\item Internal trivial fibrations in $\Kan(\C,reg)$ are exactly the internal fibrations that are also Dugger-Isaksen weak equivalences.
\item With these classes of weak equivalences and fibrations, $\Kan(\C,reg)$ is a category of fibrant objects.
\end{enumerate}
\end{prop}

\begin{lem}
\label{DIPsplit}
A map $f$ in $\Kan(\C,\proj)$ is a weak equivalence if and only if $DI_n(f)$ is $\proj$-split for all $n$.
\end{lem}
\begin{proof}
$f$ is a weak equivalence if and only if $\underline{s\C}_{sSet}(\disc P,f)$ is a weak equivalence for all $P \in \proj$, if and only if $\underline{s\C}_{sSet}(\disc P,f)$ is a Dugger-Isaksen weak equivalence for all $P$ and $n$, if and only if
\begin{align*}
\{\Delta^n,\underline{s\C}_{sSet}(\disc P,X)\} \times_{\{\Delta^n,\underline{s\C}_{sSet}(\disc P,Y)\}} \{D^{n+1},\underline{s\C}_{sSet}(\disc P,Y)\} \to \\
\{\partial\Delta^n,\underline{s\C}_{sSet}(\disc P,X)\} \times_{\{\partial\Delta^n,\underline{s\C}_{sSet}(\disc P,Y)\}} \{\Delta^n,\underline{s\C}_{sSet}(\disc P,Y)\}
\end{align*}
is a surjection for all $P$ and $n$, if and only if $$\C(P,\{\Delta^n,X\} \times_{\{\Delta^n,Y\}} \{D^{n+1},Y\}) \to \C(P,\{\partial\Delta^n,X\} \times_{\{\partial\Delta^n,Y\}} \{\Delta^n,Y\})$$ is a surjection for all $P$ and $n$, if and only if $DI(f)$ is $\proj$-split.
\end{proof}

\begin{cor}
\label{wesareregular}
Suppose $(\C,\proj)$ has enough projectives, and that for every $X \in \C$ there is a regular epimorphism $P \xrightarrow{\proj} X$ with $P \in \proj$. Then $\id_{s\C}$ induces a canonical inclusion functor $i: \Kan(\C,\proj) \to \Kan(\C,reg)$ which preserves weak equivalences and fibrations.
\end{cor}
\begin{proof}
By Lemma \ref{regularproj}, $\proj$-split maps are regular epimorphisms. In particular, if $X$ is a Kan complex in $(s\C,\proj)$, $\{\Delta^n,X\} \to \{\Lambda^n_k,X\}$ is $\proj$-split for all $n,k$, so it is regular, and $X$ is internally fibrant, and we get the inclusion functor $i$, as required. Similarly, if $f: X \to Y$ in $\Kan(\C,\proj)$ is a fibration (respectively, weak equivalence) then $\{\Delta^n,X\} \to \{\Lambda^n_k,X\} \times_{\{\Lambda^n_k,Y\}} \{\Delta^n,Y\}$ (respectively, $DI(f)$) is $\proj$-split, hence regular, so $f$ is an internal fibration (respectively, Dugger-Isaksen weak equivalence).
\end{proof}

\begin{examples}
\begin{enumerate}[(i)]
\item We may always take $\proj$ to be the maximal class $\ob \C$. Then, for every $X \in \C$, $\id_X$ is a functorial $\proj$-split regular epimorphism from an object in $\proj$.
\item Recall that $\U$ is regular. For $(\U,\sqcup CH)$, the CH-subspaces functor gives a functorial $\proj$-split regular epimorphism from an object in $\proj$.
\item By Lemma \ref{Lawvereregular}, $\U Grp$ is regular; the $\proj$-split maps from projectives constructed in $\U Grp$ in Examples \ref{sU}(iv) are regular epimorphisms.
\end{enumerate}
\end{examples}

This is a very useful property, because we can define homotopy groups for objects in $\Kan(\C,reg)$ which are invariant under weak equivalences, so under the hypotheses of Corollary \ref{wesareregular} these homotopy groups are invariant under weak equivalences in $\Kan(\C,\proj)$ too.

Explicitly: suppose $\C$ is a Barr-exact category. Then the most efficient way to define the homotopy groups of $X \in \Kan(\C,reg)$, following \cite{Low}, is to let $\pi_0(X)$ be the coequaliser of $X_1 \rightrightarrows^{d_0}_{d_1} X_0$. This makes $\pi_0: \Kan(\C,reg) \to \C$ left adjoint to $\disc: \C \to \Kan(\C,reg)$. Also $\pi_0(\disc X) = X$ for all $X \in \C$.

There are canonical maps $\Delta^0 \leftarrow \partial\Delta^n \to \Delta^n$, inducing maps $$X = \Delta^0 \pitchfork X \to \partial\Delta^n \pitchfork X \leftarrow \Delta^n \pitchfork X.$$ Write $\Omega^nX$ for the pullback, which is in $\Kan(\C,reg)$: $\Delta^n \pitchfork X \to \partial\Delta^n \pitchfork X$ is an internal fibration by \cite[Corollary 1.16]{Low}; its pullback is an internal fibration because $\C$ is regular; so $\Omega^nX \to X \to \ast$ is an internal fibration. Then let $\pi_n(X) = \pi_0(\Omega^nX)$. This construction is clearly functorial.

Because $\Omega^nX$ is constructed as a limit, and $\pi_0(X)$ is constructed as a colimit, in general $\pi_0$ will commute with colimits but $\pi_n$ will not commute with general limits or colimits.

\begin{prop}
Weak equivalences $f: X \to Y$ in $\Kan(\C,reg)$ induce isomorphisms $\pi_n(f): \pi_n(X) \to \pi_n(Y)$. Homotopic maps $f \sim g: X \to Y$ induce the same maps $\pi_n(X) \to \pi_n(Y)$.
\end{prop}
\begin{proof}
Use the embedding theorem \cite[Theorem 1.12]{Low}.
\end{proof}

We may also fix some $x: \disc T \to X$; this is `choosing a basepoint'. Define $(\Omega^nX,x)$ to be the pullback of $\disc T \xrightarrow{x} X \leftarrow \Omega^nX$ and $\pi_n(X,x) = \pi_0(\Omega^nX,x)$. In this case, $\pi_n(X,x)$ is a group object in the slice category $\C/T$ for $n \geq 1$, commutative for $n \geq 2$. By taking $T$ to be the terminal object in $\C$, we get that $\pi_n(X,x)$ is a group object in $\C$.

\begin{rem}
\cite[p.1196]{vO} suggests taking the unary map $e$ giving the identity in the definition of group object to be a map from a subobject of the terminal object.
\end{rem}

\begin{prop}
A map $f: X \to Y$ which induces isomorphisms $\pi_0(f): \pi_0(X) \to \pi_0(Y)$ and $\pi_n(f): \pi_n(X,x) \to \pi_n(Y,fx)$ for all $n$ and all basepoints $x$ is a weak equivalence in $\Kan(\C,reg)$.
\end{prop}
\begin{proof}
\cite[Theorem 1.21]{Low}
\end{proof}

Either by using an embedding theorem like \cite[Theorem 1.12]{Low}, or by working internally in $\Kan(\C,reg)$, a lot of the nice properties of the usual homotopy groups of spaces can be shown here: the $\pi_n(X,x)$ become modules for $\pi_1(X,x)$, fibre sequences give long exact sequences of homotopy group objects, $\pi_n(X,x) \cong \pi_n(X,y)$ for $x,y$ `in the same path-component', in an appropriate sense, and so on. See \cite{vO} for more on this.

These definitions recover the usual definitions of homotopy groups on spaces. Given $X \in \U$ with basepoint $x$, write $\pi^{abs}_n(X,x)$ for the usual $n$th homotopy group of $(X,x)$; then $\pi^{abs}_n(X,x) = \pi_n(\Sing^{abs}(X),x)$ (where the last $x$ is the map $$x:\disc\{\ast\} \to \Sing^{abs}(X), \ast \mapsto (\Delta^n \to \{x\})).$$ We would like to define topological homotopy groups similarly by $\pi_n(X,x) = \pi_n(\Sing(X),x)$, but $\U$ is not Barr-exact.

Suppose now that $\C$ is regular. We say a sequence $Z \rightrightarrows^p_q X \xrightarrow{f} Y$ is exact if $(p,q)$ is the kernel pair of $f$ and $f$ is the coequaliser of $(p,q)$. We say a functor between regular categories is exact if it preserves finite limits and regular epimorphisms (this is stronger than preserving exact sequences; see \cite{Barr}). We define the \emph{exact completion} $\C_{ex}$ of $\C$ by the following universal property: $\C_{ex}$ is Barr-exact, there is an exact functor $\varepsilon: \C \to \C_{ex}$, and every exact functor from $\C$ to a Barr-exact category factors uniquely through $\varepsilon$.

It is possible to give an explicit construction of $\C_{ex}$ in terms of $\C$; indeed, this construction is usually given as the definition. The source for statements here is \cite{SC}. The objects of $\C_{ex}$ are pairs $(X,R)$, where $X \in \C$ and $R$ is an equivalence relation on $X$. The morphisms $(X,R) \to (Y,S)$ are relations $f$ from $X$ to $Y$ such that $Sf = f = fR$, $R \leq f^\circ f$ and $f f^\circ \leq S$, and composition is the composition of relations.

The functor $\varepsilon$ is given by $X \mapsto (X,1_X)$ on objects, and sends a map $f$ to its graph.

It is proved that the construction satisfies the universal property in \cite[Proposizione 3.15]{SC} -- note that \cite{SC} uses extremal epimorphisms where we use regular epimorphisms; in our complete category the two are equivalent. Modulo set theory, this makes the `category' of Barr-exact categories a reflective sub`category' of regular categories.

The other facts we will need from \cite[Section 3]{SC} are contained in the following proposition.

\begin{prop}
\phantomsection
\label{excompletion}
\begin{enumerate}[(i)]
\item $\varepsilon$ makes $\C$ a full subcategory of $\C_{ex}$.
\item A morphism $f: (X,R) \to (Y,S)$ is a monomorphism if and only if $f^\circ f = R$, and is a regular epimorphism if and only if $f f^\circ = S$.
\item Every relation $R$ from $X$ to $Y$ in $\C$ has a \emph{canonical decomposition} $(Z,p,q)$: there are maps $p: Z \to X$, $q: Z \to Y$ in $\C$ such that $R = qp^\circ$, and for every other decomposition $X \xrightarrow{r^\circ} W \xrightarrow{s} Y$ of $R$, there is a unique $t: W \to Z$ such that $r=pt$ and $s=qt$.
\item If $(Z,p,q)$ is the canonical decomposition in $\C$ of an equivalence relation $R$ on $X \in \C$, then $(Z,1_Z) \rightrightarrows^p_q (X,1_X) \xrightarrow{R} (X,R)$ is an exact sequence in $\C_{ex}$.
\item For a map $(X,R) \to (Y,S)$ in $\C_{ex}$, its canonical decomposition is $$(X,R) \xleftarrow{Rp} (E,T) \xrightarrow{Sq} (Y,S),$$ where $X \xleftarrow{p} E \xrightarrow{q} Y$ is the canonical decomposition in $\C$ and $T = p^\circ R p \cap q^\circ S q$, and $Rp$ is an isomorphism.
\end{enumerate}
\end{prop}

\begin{thm}
\label{delta}
$\varepsilon: \C \to \C_{ex}$ has a left adjoint $\delta: \C_{ex} \to \C$, which is exact if and only if $\C$ is Barr-exact.
\end{thm}
\begin{proof}
For $(X,R) \in \C_{ex}$, by composing $R \to X \times X$ with the two projection maps to $X$ we may think of this equivalence relation as a pair of maps $R \rightrightarrows^{d_1}_{d_2} X$. Define $\delta(X,R) = \coeq(d_1,d_2)$.

Defining the morphisms in $\C_{ex}$ is more delicate. For a map $(X,R) \to (Y,S)$, take the canonical decomposition $(X,R) \xleftarrow{Rp} (E,T) \xrightarrow{Sq} (Y,S)$. Now $qTq^\circ \leq qq^\circ S qq^\circ \leq S$, so the composite $T \to E \times E \xrightarrow{q \times q} Y \times Y$ factors as $T \xrightarrow{q} S \to Y$, and we get a commutative diagram
\[
\xymatrix{E \ar[r]^{q} \ar[d] & S \ar[d] \\
T \times T \ar[r]^{q \times q} & Y \times Y.}
\]
By the same argument we also get a commutative diagram
\[
\xymatrix{E \ar[r]^{p} \ar[d] & R \ar[d] \\
T \times T \ar[r]^{p \times p} & X \times X.}
\]
Each of these maps of diagrams functorially induces a map of the coequalisers, $$\delta(X,R) \xleftarrow{p'} \delta(E,T) \xrightarrow{q'} \delta(Y,S);$$ indeed the whole construction has been functorial. I claim $p'$ is an isomorphism, so that defining $\delta((X,R) \to (Y,S)) = q'(p')^{-1}$ makes $\delta$ into a functor.

We will prove this claim and the adjunction of $\delta$ with $\varepsilon$ simultaneously. This is a notational device: strictly speaking, we are first using the following argument to prove the claim, then using it for the adjunction. Construct for $(X,R)$ the exact sequence $(Z,1_Z) \rightrightarrows^r_s (X,1_X) \xrightarrow{R} (X,R)$ given in Proposition \ref{excompletion}(iv). It is easy to check that $\delta(X,R)$ is actually the coequaliser in $\C$ of $(r,s)$ (here we are identifying maps in $\C$ with their image in $\C_{ex}$; $r$ and $s$ are in $\C$ because it is full). Then
\begin{align*}
\C(\delta(X,R),-) &= \eq(\C(X,-) \rightrightarrows \C(Z,-)) \\
&= \eq(\C_{ex}((X,1_X),\varepsilon(-)) \rightrightarrows \C_{ex}((Z,1_Z),\varepsilon(-))) \\
&= \C_{ex}((X,R),\varepsilon(-))
\end{align*}
naturally in all the arguments. The same argument shows $\C(\delta(E,T),-) = \C_{ex}((E,T),\varepsilon(-))$, and together with the induced map $\C_{ex}((X,R),\varepsilon(-)) \to \C_{ex}((E,T),\varepsilon(-))$, which is a natural isomorphism by Proposition \ref{excompletion}(v), we get that the induced map $\C(\delta(X,R),-) \to \C(\delta(E,T),-)$ is a natural isomorphism too, so $p'$ is an isomorphism by the Yoneda lemma.

Finally, if $\C$ is Barr-exact $\varepsilon$ and $\delta$ are equivalences of categories, so $\delta$ is clearly exact. Conversely, if $(X,R)$ is an ineffective equivalence relation in $\C$, $\delta$ sends the exact sequence $(Z,1_Z) \rightrightarrows^r_s (X,1_X) \xrightarrow{R} (X,R)$ to $Z \rightrightarrows X \to \delta(X,R)$, but $X \to \delta(X,R)$ is ineffective.
\end{proof}

This makes $\C$ into a reflective subcategory of $\C_{ex}$, and we will identify objects $X \in \C$ with their images $(X,1_X) \in \C_{ex}$.

$\varepsilon: \C \to \C_{ex}$ induces a functor $s\C \to s\C_{ex}$. Because $\varepsilon$ preserves finite limits and regular epimorphisms, this restricts to a functor $\Kan(\C,reg) \to \Kan(\C_{ex},reg)$ which preserves Dugger-Isaksen weak equivalences and internal fibrations, inducing $\gamma: Ho(\Kan(\C,reg)) \to Ho(\Kan(\C_{ex},reg))$ on the homotopy categories. (It is shown in \cite{Brown} that these homotopy categories exist.)

\begin{thm}
$\gamma$ is an equivalence of categories.
\end{thm}
\begin{proof}
Let $\M$ be the class of regular epimorphisms in $\C_{ex}$, and let $\proj = \ob \C$. For $(X,R) \in \C_{ex}$, we have a regular epimorphism $(X,1_X) \to (X,R)$. Unfortunately there seems to be no good way to make a functorial choice of regular epimorphism from $\C$ (we might have $(X,R) \cong (Y,S)$ with $X \neq Y$). Now we use Proposition \ref{factorisation} to factor the map from the initial object to some $Y \in s\C_{ex}$ through an internal trivial fibration $Q(Y) \to Y$ with $Q(Y) \in s\C$. In particular, for $Y \in \Kan(\C_{ex},reg)$, the composite $Q(Y) \to Y \to \ast$ is a fibration in $s\C_{ex}$, and hence a fibration in $s\C$, so $Q(Y) \in \Kan(\C,reg)$.

Therefore $\gamma$ is essentially surjective. We will show that $\gamma$ is essentially surjective, full and faithful, using the description of the homotopy category from \cite[Theorem 1]{Brown}: any map $Y \to Z$ in the homotopy category of a category of fibrant objects can be represented by $Y \leftarrow W \to Z$ where the map to the left is a weak equivalence.

It is clear that $\gamma$ is faithful. $\gamma$ is full: suppose $Y,Z \in \Kan(\C,reg)$. Suppose $Y \xleftarrow{t} W \xrightarrow{f} Z$ represents a map $Y \to Z$ in $Ho(\Kan(\C_{ex},reg))$. Take $Q(W) \in \C_{reg}$ and a trivial fibration $q: Q(W) \to W$: then $Y \xleftarrow{tq} Q(W) \xrightarrow{fq} Z$ represents the same map in $Ho(\Kan(\C_{ex},reg))$, and is in $Ho(\Kan(\C,reg))$.
\end{proof}

Given $X \in \Kan(\C,reg)$, and a map $x: \disc T \to X$, we define $\pi_n(X) = \pi_n(\varepsilon(X))$ and $\pi_n(X,x) = \pi_n(\varepsilon(X),\varepsilon(x))$. This is well-defined because $\varepsilon$ is an equivalence when $\C$ is Barr-exact. Because $\varepsilon$ is exact, we get for example long exact sequences of homotopy groups from fibre sequences in $\Kan(\C,reg)$. Notice that these `homotopy groups' are in the exact completion of $\C$, not in $\C$ itself.

We also define $\pi^\C_n(X) = \delta(\pi_n(X))$ and $\pi^\C_n(X,x) = \delta(\pi_n(X,x))$. If $\C$ is cartesian closed, colimits commute with finite products, and it follows that $\pi^\C_n(X,x)$ is actually a group object in $\C$ for $n \geq 1$. Since these are in $\C$, they might be easier to work with; however, by Theorem \ref{delta}, if $\C$ is not Barr-exact we can no longer expect long exact sequences of these objects.

Each $\pi_n(X)$ and $\pi^\C_n(X)$ is invariant under Dugger-Isaksen weak equivalences, because these are preserved by $\Kan(\C,reg) \to \Kan(\C_{ex},reg)$. In particular, for a fixed class $\proj$ in $\ob \C$ closed under retracts and coproducts, such that for every $X \in \C$ there is a regular epimorphism $P \xrightarrow{\proj} X$ with $P \in \proj$, each $\pi_n(X)$ and $\pi^\C_n(X)$ is invariant under weak equivalences in $\Kan(\C,\proj)$, and $\pi_n,\pi^\C_n$ give functors $Ho(\Kan(\C,\proj)) \to \C_{ex}, Ho(\Kan(\C,\proj)) \to \C$, respectively.

Let us now restrict attention to the case $\C = \U$. Then we have a model structure $(s\U,CH)$ on simplicial complexes in $\U$ and a category of fibrant objects structure $\Kan(\U,reg)$ on Kan complexes in $\U$. We would like to extend the weak equivalences of $\Kan(\U,reg)$ to weak equivalences on all of $s\U$, and we do this by defining a map $X \to Y$ to be a weak equivalence in the structure $(s\U,reg)$ if $\Ex^\infty(X) \to \Ex^\infty(Y)$ is a weak equivalence in $\Kan(\U,reg)$: recall $\Ex^\infty(X)$ is in $\Kan(\U,CH)$ by the proof of Theorem \ref{modelcat}, and hence in $\Kan(\U,reg)$ by Corollary \ref{wesareregular}.

\begin{lem}
For $X \in \Kan(\U,reg)$, $X \to \Ex^\infty$ is a weak equivalence. A map $X \to Y$ between objects in $\Kan(\U,reg)$ is a weak equivalence in $\Kan(\U,reg)$ if and only if it is a weak equivalence in $(s\U,reg)$.
\end{lem}
\begin{proof}
The first statement is \cite[Proposition 3.5]{Myself}. The rest follows easily.
\end{proof}

\begin{lem}
The identity on $s\U$ induces a functor $(s\U,CH) \to (s\U,reg)$ which preserves weak equivalences, fibrations and trivial fibrations.
\end{lem}
\begin{proof}
The functor preserves fibrations and trivial fibrations because $CH$-split epimorphisms are regular by Lemma \ref{regularproj}. If $X \to Y$ is a weak equivalence in $(s\U,CH)$, we know that $X \to \Ex^\infty(X)$ and $Y \to \Ex^\infty(Y)$ are the proof of Theorem \ref{modelcat}, so $\Ex^\infty(X) \to \Ex^\infty(Y)$ is one too. This is a weak equivalence in $(s\U,reg)$ by Lemma \ref{DIPsplit}.
\end{proof}

We can now state the following result, which may be thought of as a continuous version of the Seifert--van Kampen Theorem. Recall that, as a total left derived functor, $L|-|$ preserves \emph{homotopy colimits} (see \cite{Riehl} for definitions).

\begin{thm}
\label{singhoco}
If $C$ is an open cover of $X \in \U$, write $C'$ for the poset of finite intersections of sets in $C$, ordered by inclusion. Then $\Sing(X)$ is weakly equivalent (in $(s\U,reg)$) to the homotopy colimit (in the $(s\U,CH)$) of $\{\Sing(U)\}_{U \in C'}$.
\end{thm}
\begin{proof}
See \cite[Theorem 4.1]{Myself}.
\end{proof}

We can use this to prove the following partial converse to Proposition \ref{geomsingwe}.

\begin{thm}
\label{tpd}
Suppose $X \in s\U$ with $X_n$ totally path-disconnected for all $n$. Then $\underline{\Sing} \circ L\lvert X \rvert$ is weakly equivalent to $X$ in $(s\U,reg)$.
\end{thm}
\begin{proof}
See \cite[Theorem 6.3]{Myself}.
\end{proof}

This is important because it allows us to build spaces in certain weak equivalence classes in $\U$, by choosing a simplicial space and taking its geometric realisation to get a KW-complex. This idea underlies the construction of Eilenberg--Mac Lane spaces for totally path-disconnected groups in $\U$ in \cite{Myself}.

\section{Some calculations}
\label{calcs}

We now restrict attention again to $\U$. Since $\Sing: \U_{CH} \to \Kan(\U,\sqcup CH)$ and $i: \Kan(\U,\sqcup CH) \to \Kan(\U,reg)$ preserve weak equivalences, we may define, for $X \in \U$, $\pi_n(X) = \pi_n(\Sing(X))$, $\pi^\U_n(X) = \pi^\U_n(\Sing(X))$, and similarly for the definitions with basepoints. It will also be useful to let $\pi^\K_n(X)$ be the coequaliser of the equivalence relation $\pi_n(X)$ in $\K$, and similarly for $\pi^\K_n(X,x)$. As well as the usual homotopy group $\pi^{abs}_n(X)$, we also define $\pi^{Top}_n(X)$ as follows: consider $S \in sTop$ given by $S_n = \underline{Top}(\Delta_n,X)$. Let $\pi^{Top}_0(X) = \coeq_{Top}(S_1 \rightrightarrows^{d_0}_{d_1} S_0)$; the same definition as for Barr-exact categories in the previous section then gives $\pi^{Top}_n(X)$ for all $n$, as well as versions with basepoints.

Now if $(\U_{CH})_\ast$ is the category of pointed spaces in $\U$ with the induced model structure (see \cite[Proposition 1.1.8]{Hovey}), we get functors
\begin{align*}
&\pi_n:Ho(\U_{CH})_\ast \to \U_{ex}Grp, \\
&\pi^\U_n:Ho(\U_{CH})_\ast \to \U Grp, \\
&\pi^\K_n:Ho(\U_{CH})_\ast \to \K Grp, \\
&\pi^{abs}_n:Ho(\U_{CH})_\ast \to Grp
\end{align*}
for $n \geq 1$.

On the other hand, $\pi^{Top}_n(X,x)$ is not in general a topological group. It is shown in \cite[Theorem 1]{Fabel} that, for $X$ the Hawaiian earring (which is in $\U$ because it is compact Hausdorff), and a point $x \in X$, the multiplication map on the underlying sets $\pi^{Top}_1(X,x) \times \pi^{Top}_1(X,x) \to \pi^{Top}_1(X,x)$ is not continuous. Specifically, what goes wrong is that $Top$ is not cartesian closed, so that the product of two quotient maps need not be a quotient. This seems a strong argument in favour of using $\pi_n$, $\pi^\K_n$ and $\pi^\U_n$ instead; though there are no spaces which I know have the same homotopy groups using $\pi^{Top}_n$, but which I know are distinguished by their homotopy groups using $\pi_n$, $\pi^\K_n$ and $\pi^\U_n$.

Unravelling the definitions of all these homotopy objects, we can give more explicit descriptions. Let $(S^n,s)$ be the $n$-sphere with a basepoint. Given the space of based maps $\underline{\U_\ast}((S^n,s),(X,x))$, we can also put a topology on the set $H$ of based homotopies between maps $(S^n,s) \to (X,x)$. $H$ is the subset of maps $S^n \times [0,1] \to X$ satisfying the relevant conditions, so we may give it the $k$-subspace topology from $\underline{\U}(S^n \times I,X)$, which is again in $\U$. The inclusion $\{0,1\} \to [0,1]$ induces $H \rightrightarrows \underline{\U_\ast}((S^n,s),(X,x))$. We can factor $$H \to \underline{\U_\ast}((S^n,s),(X,x)) \times \underline{\U_\ast}((S^n,s),(X,x))$$ through its image $H'$ with the quotient topology, and it is not hard to check that $H' \rightrightarrows \underline{\U_\ast}((S^n,s),(X,x))$ is an equivalence relation in $\U$, in the categorical sense. Write $U$ for the forgetful functor $Top \to Set$. Then
\begin{align*}
&\pi_n(X,x) = (\underline{\U_\ast}((S^n,s),(X,x)),H'), \\
&\pi^\K_n(X,x) = \coeq_\K(H \rightrightarrows \underline{\U_\ast}((S^n,s),(X,x))), \\
&\pi^\U_n(X,x) = \coeq_\U(H \rightrightarrows \underline{\U_\ast}((S^n,s),(X,x))), \text{and} \\
&\pi^{abs}_n(X,x) = \coeq_{Set}(H \rightrightarrows \underline{\U_\ast}((S^n,s),(X,x))).
\end{align*}
Of course, from this perspective the definitions of $\pi^\K_n(X,x)$ and $\pi^{Top}_n(X,x)$ make at least as much sense for $X \in \K$ as for $X \in \U$; but we will largely restrict ourselves to $\U$ for consistency with the theory of the previous sections.

The description for $\pi^{Top}_n(X,x)$ is similar, using compact-open topologies instead of $k$-compact-open ones.

Then we have the following relations:

\begin{lem}
\phantomsection
\label{pirelations}
\begin{enumerate}[(i)]
\item $\pi^\U_n(X,x) = \delta\pi_n(X,x) = \pi^\K_n(X,x)_{WH}$.
\item $\pi^{abs}_n(X,x) = U\pi^\K_n(X,x) = U\pi^{Top}_n(X,x)$.
\item The identity map on the underlying sets gives a morphism $\pi^\K_n(X,x) \to \pi^{Top}(X,x)$ in $Top$, which is a homeomorphism when $X$ is metrisable.
\item $\pi_0, \pi^\U_0, \pi^\K_0, \pi^{Top}_0$ and $\pi^{abs}_0$ all commute with coproducts in $\U$.
\end{enumerate}
\end{lem}
\begin{proof}
\begin{enumerate}[(i)]
\item This is immediate.
\item Observe that $U$ preserves colimits in $\K$ and $Top$ (though not in $\U$).
\item We have a commutative square
\[
\xymatrix{\underline{\U_\ast}((S^n,s),(X,x)) \ar[r] \ar[d] & \underline{Top_\ast}((S^n,s),(X,x)) \ar[d] \\
\pi^\K_n(X,x) \ar[r] & \pi^{Top}_n(X,x)}
\]
in which the top row is continuous and the vertical maps are quotients (by definition). It follows by chasing open sets around that the bottom row is continuous. When $X$ is metrisable, the proof of \cite[Proposition 2.13]{Strickland} shows that the top row is a homeomorphism, because $S^n$ is compact and Hausdorff; the result follows.
\item We noted earlier that $\pi_0$ commutes with colimits in $s\U$; here we also use the fact that $\Sing$ commutes with coproducts.
\end{enumerate}
\end{proof}

A space is called \emph{totally path-disconnected} if its path components are one-point sets. Totally disconnected spaces are totally path-disconnected.

\begin{lem}
\begin{enumerate}[(i)]
\item For $X \in \U$, $\pi^\K_0(X) = \pi^{Top}_0(X)$ is the set of path components of $X$, with the quotient topology (which we will call the path-component space of $X$).
\item For $X \in \U$ totally path-disconnected, $\pi_0(X) = (X,1_X)$, and for $n>0$ we have $\pi_n(X,x) = \ast$, the terminal group object in $\U_{ex}$.
\item $\pi^\K_n$ and $\pi^\U_n$ commute with finite products.
\end{enumerate}
\end{lem}
\begin{proof}
\begin{enumerate}[(i)]
\item Coequalisers in $\K$ can be calculated as the coequaliser of the underlying maps of sets, with the quotient topology.
\item Continuous maps from $S^n$ to $X$, and homotopies between them, must be constant. The rest follows.
\item In $\K$ and $\U$, $- \times -$ is both a limit and a left adjoint. Hence it commutes with limits and colimits; in particular, with $\Sing$, $\Omega^n$, $\pi^\K_0$ and $\pi^\U_0$ (where the last two are considered as functors on $\Kan(\U,reg)$).
\end{enumerate}
\end{proof}

We can now use these invariants to get some concrete results. The next two propositions fulfil earlier promises.

\begin{prop}
\label{notwe}
$(|-|,\Sing)$ is not a Quillen equivalence between $(s\U,CH)$ and $\U_{CH}$.
\end{prop}
\begin{proof}
Unfortunately this is true for uninteresting reasons. Using the characterisation in \cite[Proposition 1.3.13(ii)]{Hovey}, and the fact that all objects in $\U_{CH}$ are fibrant, it is enough to find a cofibrant $X \in (s\U,CH)$ such that $X \to \Sing|X|$ is not a weak equivalence. Take $X = \disc K$, where $K$ is your favourite path-connected compact Hausdorff space with more than one point (or any other such space). Then $\pi^\K_0(|\disc K|) = \pi^\K_0(K)$ is a single point, but $\pi^\K_0(\disc K) = K$.
\end{proof}

\begin{prop}
$\id_\U$ is not a Quillen equivalence between $\U_Q$ and $\U_{CH}$.
\end{prop}
\begin{proof}
It is enough to find a space $X$ which is not weakly equivalent in $\U_{CH}$ to a CW-approximation. The following may be the simplest possible example in some well-defined sense. Let $X$ be the topologist's sine curve: the subset of $\mathbb{R}^2$ given by $$\{(0,y):-1 \leq y \leq 1\} \cup \{(x,y):y=\sin(1/x), 0 < x \leq 1/\pi\}.$$ $X$ is in $\U$ because it is compact and Hausdorff. It is not hard to check that the map from the discrete space on two points $$\{\ast_1,\ast_2\} \to X, \ast_1 \mapsto (0,0), \ast_2 \mapsto (1/\pi,0)$$ is a weak homotopy equivalence, but $\pi^\K_0(\{\ast_1,\ast_2\}) = \{\ast_1,\ast_2\}$ and $\pi^\K_0(X)$ is the Sierpinski space: the space with two points of which exactly one is open.
\end{proof}

We will now return to investigating homotopy groups of spaces $X$ in $\U$. There is a range of results in the literature, giving results and explicit calculations of $\pi^{Top}_n(X,x)$: see \cite{GHMM} for attempts to understand the topology on $\pi^{Top}_n(X,x)$ via the local behaviour of $X$; in another direction, see \cite{Brazas}, which replaces the topology on $\pi^{Top}_n(X,x)$ with a weaker one which makes it a topological group, and proves these topological homotopy groups are well-behaved in many respects. The fact that $\pi^{Top}_n(X,x)$ need not be a topological group seems to have caused some misunderstandings in the past (for example in \cite{GHMM}), so care must be taken in reading some of these and related papers.

It would be nice to be able to more easily relate $\pi^\K_n(X,x)$ to $\pi^{Top}_n(X,x)$, for example by applying $k( )$, to allow us to use these results. But consider the following situation: we have a commutative diagram
\[
\xymatrix{\underline{\U_\ast}((S^n,s),(X,x)) \ar[r] \ar[d] & k(\underline{Top_\ast}((S^n,s),(X,x))) \ar[d] \\
\pi^\K_n(X,x) \ar[r] & k(\pi^{Top}_n(X,x)),}
\]
in which the top row is a homeomorphism by Lemma \ref{pirelations}. If the right vertical map is a quotient map, the bottom row is a homeomorphism, as desired. But it is not clear why this should be so. On the other hand, when $\underline{Top_\ast}((S^n,s),(X,x))$ is compactly generated, its quotient $\pi^{Top}_n(X,x)$ is too, and we have $$\pi^\K_n(X,x) = k(\pi^{Top}_n(X,x)) = \pi^{Top}_n(X,x).$$ As noted in Lemma \ref{pirelations}, this occurs when $n=0$ or $X$ is metrisable. Note that even in this situation, $\pi^{Top}_n(X,x)$ may fail to be a topological group, when the multiplication map uses $\underline{Top_\ast}((S^n,s),(X,x)) \times_0 \underline{Top_\ast}((S^n,s),(X,x))$: indeed, the Hawaiian earring in \cite{Fabel} provides an example of just such a failure.

In the $n=0$ case, it is worth mentioning the remarkable result \cite[Corollary 2.6]{Harris}, which for $X \in Top$ constructs a new space $S(X) \in Top$, naturally in $X$, with the property that $\pi^{Top}_0(S(X)) = X$. Unfortunately it is not clear whether $S(X)$ is in $\U$ when $X$ is, or how to modify the strategy to produce an $S'(X) \in \U$ with $\pi^\K_0(S'(X)) = X$.

See \cite{GHMM} for definitions of \emph{$n$-semilocally simply connected} and \emph{locally $n$-connected}.

\begin{lem}
\begin{enumerate}[(i)]
\item If $X$ is metrisable and $\pi^\K_n(X,x)$ is discrete for all $x \in X$, $X$ is $n$-semilocally simply connected.
\item If $X$ is metrisable and locally $n$-connected, $\pi^\K_n(X,x)$ is discrete for all $x \in X$.
\item Suppose $X$ is a locally $(n - 1)$-connected metrisable space and $x \in X$. Then the following are equivalent:
\begin{enumerate}[(a)]
\item $\pi^\K_n(X,x)$ is discrete.
\item $X$ is $n$-semilocally simply connected at $x$.
\end{enumerate}
\end{enumerate}
\end{lem}
\begin{proof}
Thanks to the above observations, this is \cite[Theorem 3.2, Theorem 3.6, Theorem 3.7]{GHMM}.
\end{proof}

Suppose $X \in \U$ is \emph{locally path-connected}: that is, for every $x \in X$ and every open neighbourhood $U$ of $x$, there is a smaller open neighbourhood $V \subseteq U$ of $x$ which is path-connected. Every open subspace of a locally path-connected space is locally path-connected, and the path-components of a locally path-connected space are open. In particular $X$ is the disjoint union of its path-components, so we have seen $\pi^\K_0(X,x)$ is discrete, but we can say more. $\pi_0(X)$ is the equivalence relation $(X,H')$, in the notation from the beginning of the section.

\begin{lem}
$\pi_0(X)$ is (isomorphic in $\U_{ex}$ to the image under $\varepsilon$ of) $\pi^{abs}_0(X)$, with the discrete topology.
\end{lem}
\begin{proof}
Since $\pi_0$ commutes with colimits, we may assume $X$ is path-connected, so $H' \to X \times X$ is surjective. We need to show that $(X,H')$ is effective, that is, that $H' \to X \times X$ is open. Given an open set in $H'$, by definition its preimage in $H = \underline{\U}([0,1],X)$, is open; it suffices to show $f: H \to H' \to X \times X$ is open. So take a basic open set $W(K,U)$ of $\underline{\U}([0,1],X)$. Write $\{U_i\}_{i \in I}$ for the path-components of $U$: $U$ is locally path-connected so the $U_i$ are open in $X$. We now consider several cases; checking each one is left to the reader.

If $0,1 \notin K$, then $f(W(K,U)) = X \times X$. If, without loss of generality, $0 \in K$ and $1 \notin K$, then $f(W(K,U)) = U \times X$. If $0,1 \in K$ but $K \neq [0,1]$, then $f(W(K,U)) = U \times U$. If $K = [0,1]$, then $f(W(K,U)) = \bigcup_i U_i \times U_i$.
\end{proof}

It is shown in \cite[Theorem 3]{Milnor} that, if $X$ is homotopy equivalent to a CW-complex, $\underline{Top}((S^n,s),(X,x))$ is homotopy equivalent to a CW-complex. The argument goes via replacing $X$ with a homotopy equivalent metric space before applying $\underline{Top}((S^n,s),-)$. As noted in Lemma \ref{pirelations}, on metric spaces $\underline{Top}((S^n,s),-) = \underline{\U}((S^n,s),-)$, and so \cite{Milnor} shows that $\underline{\U}((S^n,s),(X,x))$ is homotopy equivalent to a CW-complex for all $n$. CW-complexes are locally path-connected, so we get:

\begin{thm}
Let $X$ be homotopy equivalent to a CW-complex. Then $\pi_n(X,x)$ is $\pi^{abs}_n(X,x)$, with the discrete topology.
\end{thm}

Since maps of spaces which induce isomorphisms of all the abstract homotopy groups are weak homotopy equivalences, we might wonder whether our new homotopy objects detect weak equivalences in $\U_{CH}$. The following example shows they do not.

\begin{example}
Consider $$X = ([0,1] \cap \mathbb{Q}) \cup [1,2] \cup ([2,3] \cap \mathbb{Q}), Y = [0,2] \cap \mathbb{Q},$$ and $f: X \to Y$ given by squeezing the $[1,2]$ segment to a point: $f(x) = x$ for $x\leq 1$, $f(x) = 1$ for $1 \leq x \leq 2$, and $f(x) = x-1$ for $2 \leq x$. As $Y$ is totally disconnected, we have already seen $\pi_0(Y) = Y$ and all higher homotopy groups are trivial; it is not hard to check that all the higher homotopy groups of $X$ are trivial too, and that the equivalence relation $\pi_0(X)$ is effective so that $\pi_0(X) = \pi^\K_0(X) = Y$.

However, $f$ is not a weak equivalence in $\U_{CH}$: let $K = \mathbb{N} \cup \infty$, the one-point compactification of the discrete space $\mathbb{N}$; we will show $\underline{\U}(K,X) \to \underline{\U}(K,Y)$ is not a weak homotopy equivalence. Indeed, since $Y$ is totally disconnected, $$\pi^{abs}_0(\underline{\U}(K,Y)) = \underline{\U}(K,Y) = \{\text{convergent sequences in Y}\}.$$ But the convergent sequence $((-1)^n/n)_{n \in \mathbb{N}}$ has no preimage in $\underline{\U}(K,X)$: that is, the preimage of the sequence $((-1)^n/n)_{n \in \mathbb{N}}$ does not converge. So $\pi^{abs}_0(\underline{\U}(K,X)) \to \pi^{abs}_0(\underline{\U}(K,Y))$ is not surjective.
\end{example}

The same example, applied to $\Sing(X) \to \Sing(Y)$, shows that weak equivalences in $(s\U,CH)$ are finer than weak equivalences in $(s\U,reg)$.

Here is a nice use of the long exact sequence of homotopy groups (which is not available for $\pi^\U_n$, $\pi^\K_n$ and $\pi^{Top}_n$). Suppose $\C$ is regular. For $G \in s\C Grp$, we can define a map $WG \to \bar{W}G$ in $s\C$ corresponding to the principal bundle over the classifying space, $EG \to BG$, for discrete groups. For background and definitions here a good source is \cite{Roberts}; we will just give the detail we need. $(WG)_n$ is given by $G_n \times \cdots \times G_0$, $G_i = G$, and $G$ acts on $WG$ by left multiplication on the first factor. The quotient by this action is written $\bar{W}G$.

We will also use the work of \cite{BS}, although note a difference in terminology: by global Kan complex, we mean an object in $\Kan(\C,\ob \C)$, whereas the definition given in \cite[Definition 2.1]{BS} is stronger. Explicitly, the maps $\lambda_{q,I}$ required by the definition there are only required to exist in our definition when $I$ contains $q$ elements. We will only need to know that global Kan complexes according to \cite[Definition 2.1]{BS} are in $\Kan(\C,\ob \C)$. In the same way, the fibrations defined in \cite[Definition 6.1]{BS} are global Kan fibrations in our sense (but not vice versa).

\begin{prop}
\begin{enumerate}[(i)]
\item $G \in \Kan(\C,\ob \C)$.
\item $WG \in \Kan(\C,\ob \C)$ and $WG$ is contractible (that is, it is homotopy equivalent to the terminal object $\disc\ast$).
\item The map $WG \to \bar{W}G$ is a global Kan fibration with fibre $G$ and $\bar{W}G \in \Kan(\C,\ob \C)$.
\end{enumerate}
\end{prop}
\begin{proof}
\begin{enumerate}[(i)]
\item Objects in $s\C Grp$ are global Kan complexes in $s\C$ by \cite[Theorem 3.8]{BS}, where the proof is given for simplicial topological groups, but holds in any category (with the relevant limits). Alternatively, we can make $\underline{s\C}_{sSet}(\disc P,G)$ into a simplicial group for all $P \in \C$, and observe the well-known fact that simplicial groups are fibrant.
\item $WG$ is a simplicial group object in $\C$, by \cite[Theorem]{Roberts}, so it is a global Kan complex. A contracting homotopy is given in \cite[Section 4]{Roberts}.
\item By \cite[Lemma 19]{RS} the map $WG \to \bar{W}G$ is a twisted cartesian product with fibre $G$; then \cite[Lemma 6.7]{BS} shows that this map is a global Kan fibration. Now each $(WG)_n \to (\bar{W}G)_n$ is clearly split epic, that is, $\ob\C$-split, so by Lemma \ref{Kanfibre} $\bar{W}G$ is a global Kan complex.
\end{enumerate}
\end{proof}

This fibre sequence gives a long exact sequence of homotopy groups (in the exact completion of $\C$). We will suppress the basepoints in the notation for these homotopy groups; asssume a basepoint has been chosen. Since $WG$ is contractible, $\pi_n(WG) = \pi_n(\ast) = \ast$ for all $n$. Therefore $\pi_0(\bar{W}G) = \ast$ and $\pi_n(\bar{W}G) = \pi_{n-1}(G)$ for $n>0$. In particular, for $G' \in \C Grp$ and $G = \disc G'$, the only non-zero homotopy group of $\bar{W}G$ is $\pi_1(\bar{W}G) = G'$. In this sense, $\bar{W}G$ is a \emph{simplicial Eilenberg--Mac Lane space} $K(G',1)$ for $G'$.

With a little more thought, when $G'$ is abelian we may construct $K(G',n)$ by noting that in this case $G$ is a normal subgroup object of $WG$, so $\bar{W}G$ has the structure of an $s\C$ group object. Then we let $K(G',n) = \bar{W}K(G',n-1)$. This is well-known in the case of abstract groups, and the strategy here is the same.

There are some important differences between these and the Eilenberg--Mac Lane spaces for abstract groups. First, it is hard to say anything in general about the homotopy groups of $|\bar{W}G|$ or $L|\bar{W}G|$ for $G' \in \U Grp$. However, we do have the following result:

\begin{thm}
For $G'$ a totally path-disconnected group, $L|\bar{W}G|$ is an Eilenberg--Mac Lane space for $G'$; that is, $\pi_1(L|\bar{W}G|) = G'$, with all other homotopy groups trivial.
\end{thm}
\begin{proof}
See \cite[Theorem 6.4]{Myself}.
\end{proof}

Similarly, when $G$ is totally path-disconnected and abelian, $L\lvert \bar{W}^nG\rvert$ is a $K(G,n)$.

\begin{rem}
It follows that for totally path-disconnected groups we also have $\pi^\K_1(L|\bar{W}G|) = G'$ and $\pi^\U_1(L|\bar{W}G|) = G'$ with all other homotopy groups trivial.
\end{rem}

Second, we know that the homotopy groups classify the objects of $\Kan(\C,reg)$ up to Dugger-Isaksen weak equivalence, in the sense of \cite[Theorem 1.21]{Low}, but for general $(\C,\proj)$ the weak equivalences in $\Kan(\C,\proj)$ may be finer (or coarser; see Lemma \ref{DIPsplit}).

At any rate, we may use $WG$ to define some homotopical finiteness conditions on $G'$. Suppose $\C$ is cartesian closed and has a class of $\kappa$-small generators $\bar{\proj}$, and consider $(s(G'$-$\C),G' \times \bar{\proj})$ with the model structure described in Examples \ref{sU}(vi): the $G'$-action on $WG$. By adding objects if necessary, we may assume $\bar{\proj}$ is closed under retracts and finite coproducts. We say $G'$ is \emph{of type $\bar{\proj}\F_n$} if $WG$ is weakly equivalent in $s(G'$-$\C)$ to some cofibrant $X \in I$-cell such that $X_0,\ldots,X_n$ are in $G' \times \bar{\proj}$, for $n \leq \infty$. We say $G'$ has \emph{simplicial dimension $n$}, written $\sd(G')=n$, $n \leq \infty$, if $WG$ is weakly equivalent in $s(G'$-$\C)$ to some $X \in I$-cell which has dimension $n$ (that is, $X$ can be written as a composition of pushouts by maps $\partial\Delta^m \pitchfork (G' \times P) \to \Delta^m \pitchfork (G' \times P), P \in \bar{\proj}$ in $I$ with $m \leq n$), and that is the smallest $n$ for which such an $X$ exists. We say $G'$ is \emph{of type $\bar{\proj}\F$} if $WG$ is weakly equivalent in $s(G'$-$\C)$ to some cofibrant $X \in I$-cell which has dimension $n$ for some finite $n$, such that $X_0,\ldots,X_n$ are in $G' \times \bar{\proj}$.

We observe that these definitions are analogous to type $F_n$, geometric dimension $n$, and type $F$ for abstract groups; studying the properties of these finiteness conditions is beyond the scope of the present work, but this seems to be the correct generalisation.

\section{\texorpdfstring{$R$-$k$}{R-k}-modules}
\label{kmodules}

When we deal with additive categories, of course, all the results above stated for general categories are still true. But there are some aspects of the theory that hold here but fail in the general case, or just get easier, and we will investigate these differences. First we will look at categories of $R$-module objects in $\U$. The results in this section may be known to experts, but I have been unable to find a source for them.

As for $k$-groups, we can define a category of $k$-rings, which we call $\U Ring$. Explicitly, objects are rings $R$ with a CGWH topology making $R$ an abelian $k$-group and making multiplication $R \times R \to R$ continuous; morphisms are continuous ring homomorphisms. Similarly, for $R \in \U Ring$, we write $R$-$\U Mod$ for the category of left $R$-module objects in $\U$, which we call $R$-$k$-modules (and $\U Mod$-$R$ for right $R$-module objects; modules will be left modules unless stated otherwise). We will also write $R$-$Mod$ for the category of abstract left $R$-modules (on the underlying ring of $R$), and $R$-$\K Mod$ for $R$-module objects in $\K$, which we call $R$-CG-modules. As for $\U$, we will show $R$-$\U Mod$ is regular and coregular. All this includes as a special case $\U Ab$, the category of abelian group objects in $\U$, which is $\mathbb{Z}$-$\U Mod$ where $\mathbb{Z}$ is given the discrete topology.

There are various forgetful functors, such as $R$-$\U Mod \to R$-$Mod$, $R$-$\U Mod \to \U$, $R$-$\U Mod \to \U Grp$, and so on. We will abuse notation by calling all these forgetful functors $U$, and by identifying a morphism in $R$-$\U Mod$ with its image under any of these functors.

For the rest of the section, fix a commutative $k$-ring $Q$ and a $Q$-$k$-algebra $R$, that is, $R$ is a $k$-ring together with a map of $k$-rings $Q \to R$, satisfying the usual properties. Note that (as usual for commutative rings) left $Q$-$k$-modules are the same thing as right $Q$-$k$-modules, and there is no need to differentiate.

As for $k$-groups, we have: an $R$-CG-module $A$ is an $R$-$k$-module if and only if the identity is closed in $A$, and the functor $$( )_{WH}: R\text{-}\K Mod \to R\text{-}\U Mod$$ is given by $(A)_{WH} = A/\bar{\{0_A\}}$, with the quotient topology. This is left adjoint to inclusion $R$-$\U Mod \to R$-$\K Mod$.

Given $A,B \in R$-$\U Mod$, write $\U_R(A,B)$ for the set of morphisms $A \to B$: we equip this with the structure of a(n abstract) $Q$-module in the usual way.

\begin{lem}
$R$-$\U Mod$ is additive and enriched over $Q$-$Mod$.
\end{lem}
\begin{proof}
After noting that the composition of two continuous maps is continuous, this is just the same situation as for categories of abstract modules. The biproduct $A \oplus B$ in $R$-$\U Mod$ is the biproduct of the underlying modules, with the ($k$-)product topology.
\end{proof}

\begin{lem}
$R$-$\U Mod$ is complete and cocomplete.
\end{lem}
\begin{proof}
By a standard argument, for completeness it is enough to show that $R$-$\U Mod$ has all kernels and products, and dually for cocompleteness it is enough to show that it has all cokernels and coproducts. We give the constructions, and leave it to the reader to check the details.

Given a morphism $f: A \to B$, $\ker(f)$ is just the kernel of the underlying map of modules, with the subspace topology; $\coker(f)$ is module $B/\bar{f(A)}$ with the quotient topology coming from $B$ (cokernels in $R$-$\K Mod$ are given by $B/f(A)$ with the quotient topology). Note that, with this topology, $\coker(f)$ is indeed an $R$-$k$-module -- the continuity of the addition and the $R$-action follow from the fact that the quotient map $B \to B/\bar{f(A)}$ is open.

Given a set $\{A_i : i \in I\}$ of $R$-$k$-modules, it is easy to check that the product of the underlying modules, with the product topology, is the product $\prod_I A_i$ in $R$-$\U Mod$.

The construction of $\bigoplus_I A_i$ is more interesting. We first construct the coproduct of the $A_i$ in $R$-$\K Mod$. Write $\bigsqcup_I A_i$ for the disjoint union of the spaces underlying the $A_i$, and $F_{Ab}(\bigsqcup_I A_i)$ for the free abelian CG-group on this space. The underlying group of $F_{Ab}(\bigsqcup_I A_i)$ is the abstract free abelian group on $\bigsqcup_I A_i$, so we have a canonical map to the abstract direct sum $\bigoplus_I A_i$: give $\bigoplus_I A_i$ the quotient topology coming from $F_{Ab}(\bigsqcup_I A_i)$ via this map. So $\bigoplus_I A_i$ is an abelian CG-group. The $R$-action descends from the diagonal action on $\bigsqcup_I A_i$, because products commute with colimits in $\K Space$, and one can then check that $\bigoplus_I A_i$ is indeed the coproduct in $R$-$\K Mod$.

By \cite[Corollary 2.15]{Lamartin}, $F_{Ab}(\bigsqcup_I A_i)$ is WH, and hence is the free abelian $k$-group on $\bigsqcup_I A_i$. It follows by the same argument as \cite[Theorem 2.38]{Lamartin} that the quotient topology on $\bigoplus_I A_i$ is then WH too, so that $\bigoplus_I A_i$ is actually the coproduct in $R$-$\U Mod$.
\end{proof}

So the forgetful functor $U: R$-$\U Mod \to R$-$Mod$ preserves limits and coproducts, while the forgetful functor $R$-$\K Mod \to R$-$Mod$ actually preserves limits and colimits.

\begin{lem}
\label{free}
$R$-$\U Mod$ has free modules. That is, the forgetful functor $R$-$\U Mod \to \U$ has a left adjoint.
\end{lem}
\begin{proof}
The free modules in $R$-$\K Mod$ are constructed in \cite{Seal}. By an adjunction argument, the free $R$-$k$-module on a $k$-space $X$ (that is, the image of $X$ under the left adjoint of the forgetful functor) is given by applying $( )_{WH}$ to the free $R$-CG-module on this space.

Alternatively, one can use an argument analogous to \cite[Corollary 2.15]{Lamartin} to show that the free $R$-CG-module on $X$ is already WH, so that this is the module we want. But we do not need this.
\end{proof}

We write $R[X]$ for the free $R$-$k$-module on a $k$-space $X$.

We can now introduce $k$-group rings: given a $k$-group $G$, we can form the free $Q$-$k$-module $Q[G]$ on the underlying space of $G$. Now write $1$ for the trivial $k$-group and $e$ for the unique morphism $1 \to G$. Consider the map of spaces $$G \sqcup G = G \times 1 \sqcup 1 \times G \xrightarrow{\id_G \times e \sqcup e \times \id_G} G \times G \xrightarrow{m} G.$$ The definition of free modules gives us a morphism $Q[G] \times Q[G] \to Q[G]$ in $Q$-$\U Mod$; taking this as multiplication makes the underlying $Q$-module $Q[G]$ into a ring as usual, and since the map is continuous $Q[G]$ is therefore a $k$-ring.

Suppose $A, B \in R$-$\U Mod$. We write $\underline{\U_R}(A,B)$ for the group of morphisms $A \to B$ with the $k$-subspace topology, thinking of $\U_R(A,B)$ as a subset of $\underline{\U}(A,B)$ -- note that the topology on $\underline{\U_R}(A,B)$ is automatically WH because $\underline{\U}(A,B)$ is.

We want to say that $R$-$\U Mod$ is enriched over $Q$-$\U Mod$, but that will have to wait until we have defined tensor products.

\begin{lem}
$R$-$\U Mod$ is enriched over $\U$, and $\underline{\U_R}(A,B)$ is a $Q$-$k$-module, naturally in $A$ and $B$.
\end{lem}
\begin{proof}
It follows from the cartesian closure of $\U$ that the composition map $\underline{\U}(Y,Z) \times \underline{\U}(X,Y) \to \underline{\U}(X,Z)$ is continuous, for all $X,Y,Z \in \U$. Because $\underline{\U_R}(-,-)$ is given the $k$-subspace topology coming from $\underline{\U}(-,-)$, the first statement follows.

Then $\underline{\U_R}(-,-)$ is functorial in both variables, and it follows easily that $\underline{\U_R}(A,B)$ is an abelian $k$-group. Now $adj$ from Theorem \ref{currying}(iii) gives a correspondence between continuous $Q$-actions $Q \times B \to B$ and continuous ring homomorphisms $Q \to \underline{\End_Q}(B) = \underline{\U_Q}(B,B)$. Then the continuity of composition $$\circ: \underline{\End_Q}(B) \times \underline{\U_R}(A,B) \to \underline{\U_R}(A,B)$$ shows the $Q$-action induced on $\underline{\U_R}(A,B)$ is continuous.
\end{proof}

As for abstract modules, given a right $R$-$k$-module $A$, a left $R$-$k$-module $B$ and a $Q$-$k$-module $C$, we can define the set $\operatorname{Bilin}_Q(A,B,C)$ of continuous bilinear maps $A \times B \to C$ compatible with the $Q$-action. Then the \emph{tensor product} $A \otimes_R B$ of $A$ and $B$ is a $Q$-$k$-module, together with a continuous bilinear map $A \times B \to A \otimes_R B$, such that the induced map $$\operatorname{Bilin}_Q(A,B,C) \to \U_Q(A \otimes_R B,C)$$ is an isomorphism for all $C \in Q$-$\U Mod$. This defines $A \otimes_R B$, when it exists, up to unique isomorphism, by the Yoneda lemma.

We can also define categories of $R$-$S$-$k$-bimodules in the usual way, where $R$ and $S$ are $Q$-$k$-algebras; we write $R$-$\U Mod$-$S$ for this category. Then given $A \in R$-$\U Mod$-$S$ and $B \in S$-$\U Mod$-$T$ (for $T$ another $Q$-$k$-algebra), $A \otimes_S B$ (if it exists) becomes an $R$-$T$-$k$-bimodule. Similarly, for $A \in S$-$\U Mod$-$R$ and $B \in S$-$\U Mod$-$T$, we get $\underline{\U_S}(A,B) \in R$-$\U Mod$-$S$. Details are left to the reader.

\begin{lem}
\phantomsection
\label{tensor}
\begin{enumerate}
\item Tensor products of $R$-$k$-modules exist, giving a tensor product functor $- \otimes_R -: \U Mod$-$R \times R$-$\U Mod \to Q$-$\U Mod$.
\item Suppose $R$ and $S$ are $Q$-$k$-algebras. Suppose $A \in R$-$\U Mod$, $B \in S$-$\U Mod$-$R$, and $C \in S$-$\U Mod$. Then there is a natural isomorphism $\U_S(B \otimes_R A, C) \cong \U_R(A,\underline{\U_S}(B,C))$ in $Q$-$Mod$.
\end{enumerate}
\end{lem}
\begin{proof}
\cite{Seal2} gives the construction of tensor products and the $\Hom$-$\otimes$ adjunction for CG-modules. We immediately get from the universal properties of tensor products and the $( )_{WH}$ functor that the tensor product of a right $R$-$k$-module $A$ and a left $R$-$k$-module $B$ in $Q$-$\U Mod$ is given by applying $( )_{WH}$ to the tensor product of $A$ and $B$ in $Q$-$\K Mod$. The case for bimodules is similar, and so the $\Hom$-$\otimes$ adjunction for $k$-modules follows.
\end{proof}

From the definition and this lemma, we get a lot of the usual properties of tensor products: $Q \otimes_Q -: Q$-$\U Mod \to Q$-$\U Mod$ is the identity functor; $\otimes_Q$ is additive, commutative and associative; $Q[X] \otimes_Q Q[Y] \cong Q[X \times Y]$, naturally in $X,Y \in \U$. Similarly, thinking of $R$ and $R[X]$ as $R$-$R$-$k$-bimodules in the obvious way, $R \otimes_R -: R$-$\U Mod \to R$-$\U Mod$ is the identity, $\otimes_R$ is additive, and $R[X] \otimes_R R[Y] \cong R[X \times Y]$.

We can now show:

\begin{thm}
\phantomsection
\label{kModenriched}
\begin{enumerate}[(i)]
\item $Q$-$\U Mod$ is a closed symmetric monoidal category.
\item There are strong monoidal functors $\U \to Q$-$\U Mod$ and $Q$-$Mod \to Q$-$\U Mod$, which are left adjoint to the forgetful functors.
\item $R$-$\U Mod$ is enriched, powered and copowered over $Q$-$\U Mod$.
\end{enumerate}
\end{thm}
\begin{proof}
\begin{enumerate}[(i)]
\item The proof that $\otimes_Q$ makes $Q$-$\U Mod$ into a symmetric monoidal category goes through as in the abstract case; by the remark above, the unit is given by $Q$. Then Lemma \ref{tensor} shows it is closed.
\item The functor $\U \to Q$-$\U Mod$ is the free module functor, which we have already seen is left adjoint to the forgetful functor. It is strong monoidal because $\{\ast\} \mapsto Q[\{\ast\}] = Q$ and $Q[X] \otimes_Q Q[Y] = Q[X \times Y]$, naturally in $X,Y \in \U$. All the required conditions go through easily.

We next construct a left adjoint to the forgetful functor $Q$-$\K Mod \to Q$-$Mod$. Given $A \in Q$-$Mod$, $B \in Q$-$\K Mod$ and a map $f: A \to B$ of the underlying modules, certainly if we give $A$ the discrete topology, and write $A_{disc}$ for this, the induced map $f: A_{disc} \to B$ is a morphism of abelian CG-groups; the obstruction is that the $Q$-action may not be continuous. So it is enough to show that there is a unique strongest topology $\tau$ on $A$ making the $Q$-action continuous: then the continuity of $f: (A,\tau) \to B$ is automatic. For such a $\tau$ the continuous map $A_{disc} \to (A,\tau)$ extends to a continuous surjection $\varepsilon: Q[A_{disc}] \to (A,\tau)$, so $\tau$ can be no stronger that the quotient topology $\tau_A$ coming from $\varepsilon$. But this quotient topology makes $(A,\tau_A)$ a $Q$-CG-module, so this is the topology we want. Thus the functor $A \mapsto (A,\tau_A)$ is the required left adjoint.

It now follows formally that the functor $A \mapsto (A,\tau_A)_{WH}$ is left adjoint to the forgetful functor $Q$-$\U Mod \to Q$-$Mod$. We check it is strong monoidal.

It sends the underlying ring of $Q$ to $Q$: by construction $\tau_Q$ is the strongest CG-group topology such that multiplication $Q \times (Q,\tau_Q) \to (Q,\tau_Q)$ is continuous. So it is at least as strong as the topology on $Q$ (and hence $\tau_Q$ is WH). We also have a continuous map $Q \to (Q,\tau_Q)$ given by restricting the domain of $Q[Q_{disc}] \rightarrow (Q,\tau_Q)$ to $Q[\{1_Q\}]$.

Given $A,B \in Q$-$Mod$, we check that $$(A \otimes_Q B, \tau_{A \otimes_Q B})_{WH} \cong (A,\tau_A)_{WH} \otimes_Q (B,\tau_B)_{WH},$$ naturally in $A,B$. This is just an adjoint functor/Yoneda lemma exercise, which we leave for the reader. The rest follows easily.
\item Enrichment: There are plenty of axioms to check, but after what we have already done there are no particular difficulties.

Powering and copowering: Suppose $A \in Q$-$\U Mod$, $B \in R$-$\U Mod$. Restricting the $R$-action on $B$ to $Q$, we think of $B$ as an $R$-$Q$-$k$-bimodule. Then by Lemma \ref{tensor}, letting $A \pitchfork B = \Hom_Q(A,B)$ and $A \odot B = B \otimes_Q A$ does the job.
\end{enumerate}
\end{proof}

In particular, by \cite[Theorem 3.7.11]{Riehl}, $R$-$\U Mod$ is enriched, powered and copowered over $\U$ and $Q$-$Mod$.

\begin{cor}
The isomorphism $\U_S(B \otimes_R A, C) \cong \U_R(A,\underline{\U_S}(B,C))$ of $Q$-modules in Lemma \ref{tensor} is actually an isomorphism $\underline{\U_S}(B \otimes_R A, C) \cong \underline{\U_R}(A,\underline{\U_S}(B,C))$ of $Q$-$k$-modules.
\end{cor}
\begin{proof}
This is \cite[Remark 3.3.9]{Riehl}.
\end{proof}

So far, everything we have shown works just as well for $R$-$\K Mod$ as for $R$-$\U Mod$. We now justify the presence of the WH requirement.

Observe that, in an additive category $C$, for any two morphisms $f,g: A \to B$ the equaliser of $f$ and $g$ is $\ker(f-g)$, and the coequaliser of $f$ and $g$ is $\coker(f-g)$. So regular epimorphisms in $C$ are exactly the cokernels, and regular monomorphisms in $C$ are exactly the kernels.

In particular, we get in $R$-$\U Mod$ that regular epimorphisms are exactly morphisms of the underlying $R$-modules such that the underlying map of spaces is a quotient, and regular monomorphisms are exactly morphisms of the underlying $R$-modules such that the underlying map of spaces is a closed inclusion.

\begin{lem}
Pullbacks of regular epimorphisms in $R$-$\U Mod$ are regular epimorphisms.
\end{lem}
\begin{proof}
Pullbacks and regular epics in $R$-$\U Mod$ are preserved by the forgetful functor to $\U$. So given a regular epimorphism $f$ in $R$-$\U Mod$ and a pullback $g$ of $f$, on the underlying spaces $f$ is regular epic, so $g$ is too. So $g$ is a quotient map of spaces, and it is certainly an $R$-module homomorphism, and we are done.
\end{proof}

\begin{prop}
Pushouts of regular monomorphisms in $R$-$\U Mod$ are regular monomorphisms.
\end{prop}
\begin{proof}
This takes more work. We argue dually to the proof in \cite{Strickland} that $\U$ is regular.

We will start by working with pushouts in $R$-$\K Mod$. So suppose we have a pushout diagram of the form
\[
\xymatrix{W \ar[r]^{f} \ar[d]^{p} & X \ar[d]^{q} \\ Y \ar[r]^{g} & Z,} \tag{$\ast$}
\]
where $W, X$ and $Y$ are $R$-$k$-modules. The proof proceeds in several stages.

\begin{claim}
If $p$ is a regular epimorphism, so is $q$.
\end{claim}
Observe, by the construction of colimits in $R$-$\K Mod$, that the underlying commutative square of $R$-modules is a pushout in $R$-$Mod$. So $q$ is epic in $R$-$Mod$ (this is standard in abelian categories) and in particular $q$ is surjective. Write $Z'$ for the underlying module of $Z$ endowed with the quotient topology coming from $q$: this makes $Z'$ into an $R$-CG-module. Now $gp=qf: W \to Z'$ is continuous, and $p$ is a quotient map, so by the usual argument (chase open sets around) $g: Y \to Z'$ is continuous. So $Z'$ is a cocone over ($\ast$), but by definition it has the strongest topology making $q$ continuous, so we get $Z' = Z$, and the claim follows.

\begin{claim}
If $f$ is injective, ($\ast$) is a pullback in $Set$.
\end{claim}
For this, we can forget the topologies and work with the pushout in $R$-$Mod$. We can construct $Z$ as $X \oplus Y/\{(f(w),p(w)): w \in W\}$. Then, from the usual construction of pullbacks in $Set$, the result follows.

\begin{claim}
If $f$ is a closed inclusion (of spaces), so is $g$.
\end{claim}
Consider first the case where $p$ is a quotient map, so $q$ is too. Since the pushout is created in the category of abelian groups, $g$ is injective, so for all subsets $F$ of $Y$ we have $q^{-1}g(F) = fp^{-1}(F)$ -- this is an easy check on the elements, using the fact that ($\ast$) is a pullback in $Set$. Suppose $F$ is closed. Then $p^{-1}(F)$ is closed in $W$, and as $f$ is a closed inclusion this means that $fp^{-1}(F) = q^{-1}g(F)$ is closed in $X$. But $q$ is a quotient map so $g(F)$ is closed in $Z$.

For the general case, it is formal that the square below is a pushout:
\[
\xymatrix{W \oplus Y \ar[r]^{f \oplus \id_Y} \ar[d]^{p+\id_Y} & X \oplus Y \ar[d]^{q+g} \\
Y \ar[r]^{g} & Z.}
\]
Here $p+\id_Y$ is a quotient map because it is split by $Y \xrightarrow{(0,\id_Y)} W \oplus Y$, and $f \oplus \id_Y$ is a closed inclusion by \cite[Proposition 2.32]{Strickland}. So the claim follows.

Finally, we observe that $Z$ is actually WH: its identity element $0$ is closed in $g(Y)$, which is closed in $Z$. The proposition follows.
\end{proof}

Therefore $R$-$\U Mod$ is regular and coregular. There is a word for additive categories with all kernels and cokernels which are regular and coregular: \emph{quasi-abelian}. See \cite{Prosmans} or \cite{Schneiders} for more details on such categories. We will also use \emph{left quasi-abelian} for an additive category with all kernels and cokernels which is regular.

\section{Additive categories}
\label{additive}

Suppose now that $\C$ is a complete and cocomplete additive category. As before, most of the results here require only certain limits and colimits, and the reader is invited to check which ones, but if we are working with model categories the assumption is a standard one.

\begin{thm}
\label{additivesc}
Suppose $(\C,\proj)$ has enough functorial projectives. Then we get a model structure on $s\C$ with the weak equivalences, cofibrations and fibrations defined in Section \ref{mct}.
\end{thm}
\begin{proof}
We will show that $(\C,\proj)$ satisfies condition (i) of Theorem \ref{modelcat}. Indeed, the additive structure means that $\underline{s\C}_{sSet}(\disc P,X)$ is a simplicial abelian group for all $P \in \proj$, and it is well-known that simplicial abelian groups are fibrant.
\end{proof}

In view of Remark \ref{projectivesinnature}, this gives a lot of model structures.

Note that this theorem holds for $\proj = \ob \C$, where it says that objects in $s\C$ are global Kan complexes. If $\C$ is regular, it follows that objects in $s\C$ are internal Kan complexes, because split epimorphisms are regular.

The convention for structures allowing homological algebra on additive categories is to make the definitions self-dual for aesthetic reasons; so for example quasi-abelian is regular and coregular. This convention is also seen in the more general \emph{Quillen-exact} categories, where we are given classes of \emph{inflations} and \emph{deflations}, required to satisfy some extra properties. In fact only one of these classes is needed to make most of the tools of homological algebra work: an investigation into this is carried out in \cite{BC}.

Inspecting \cite[Definition 3.2]{BC} shows that $\C$, with the class of $\proj$-split maps as deflations, is a strongly left exact category in the terminology used there, except that deflations are assumed to be cokernels (see Lemma \ref{regularproj} and Remark \ref{notepic}). As we have seen already in this paper, requiring that $\proj$-split maps be regular epimorphisms is necessary for homotopy groups to be invariant under weak equivalences, but is not necessary for working with simplicial objects. As such (via the Dold-Kan correspondence, below), a lot of the results of \cite{BC} will apply to our situation without change. We will use \emph{left exact} to mean a category which satisfies the axioms of \cite[Definition 3.1, Definition 3.2]{BC} except that deflations need not be cokernels.

We also mention that in the particularly nice case $\proj = \ob\C$, $\C$ is both left exact and right exact: split monomorphisms are closed under pushout by dualising the argument that split epimorphisms are closed under pullback, and similarly for the other axioms. Also, split epimorphisms are clearly cokernels, and dually. So in this case we actually do get a Quillen-exact category.

It may be more intuitive to use chain complexes in $\C$ in non-negative degree, instead of simplicial objects. We write $c\C$ for this category. Most of what we say about $c\C$ could be done for bounded below chain complexes, say, but we prefer $c\C$ here thanks to Theorem \ref{chainmodelcat} below. We leave to the reader the task of showing that $c\C$ is enriched over $cAb$, and that if $\C$ is enriched, tensored and cotensored over an additive category $\V$, in a way that is compatible with the additive structure in an appropriate sense, then $c\C$ is enriched, tensored and cotensored over $c\V$, in the spirit of Section \ref{se}.

Explicitly, $\underline{c\C}_{cAb}(A,B)$ is given by the $0$th truncation $\tau_{\geq 0}$ (as defined in \cite[Truncations 1.2.7]{Weibel}) of the product total complex $\operatorname{Tot}^{\scriptstyle \prod}_{m,n} \underline{C}_{Ab}(A_{-m},B_n)$.

We say that a complex $A \in (c\C,\proj)$ is \emph{exact in the $\proj$-split structure} if, for each $A_{n+1} \xrightarrow{d_n} A_n \xrightarrow{d_{n-1}} A_{n-1}$, the induced map $A_{n+1} \to \ker(d_{n-1})$ is $\proj$-split. If $\C$ is regular, we say $A$ is \emph{exact in the regular structure} if $A_{n+1} \to \ker(d_{n-1})$ is a regular epimorphism.

For each $P \in \proj$, let $S^n(P)$ be the chain complex with $P$ in degree $n$ and $0$ otherwise; let $D^n(P)$ be the chain complex with $P$ in degrees $n$ and $n-1$ and $0$ otherwise, with the map $P \to P$ given by the identity. Let $I$ be the class of maps $S^{n-1}(P) \to D^n(P)$ given by $\id_P$ in degree $n$, and let $J$ be the class of maps $0 \to D^n (P)$. Call a map $f$ in $c\C$ a cofibration if it is in $I$-cof, a fibration if it is in $J$-inj, and a weak equivalence if $\underline{c\C}_{cAb}(S^0P,f)$ is a weak equivalence in $cAb$, with the standard model structure, for all $P \in \proj$.

\begin{thm}
\label{chainmodelcat}
Suppose $(\C,\proj)$ has enough functorial projectives.
\begin{enumerate}[(i)]
\item The weak equivalences on $c\C$ described above are exactly the chain maps whose mapping cone is exact.
\item We get a model structure on $c\C$ with the weak equivalences, cofibrations and fibrations defined above.
\item The cofibrations are exactly the chain maps which are levelwise split monomorphisms with projective cokernels. In particular, cofibrant objects are complexes of projectives.
\item The fibrations are exactly the chain maps which are levelwise $\proj$-split maps in positive degree. In particular, all objects are fibrant.
\item Abuse notation by writing $(c\C,\proj)$ for this model category. Then there is a functor $\Gamma: c\C \to s\C$, which gives an equivalence of categories and a left and right Quillen functor of model categories.
\end{enumerate}
\end{thm}
\begin{proof}
\begin{enumerate}[(i)]
\item $\underline{c\C}_{cAb}(S^0P,-)$ commutes with the formation of mapping cones, and by the definition of $\proj$-split, a complex $A$ is exact in our sense if and only if $\underline{c\C}_{cAb}(S^0P,A)$ is exact in $Ab$ for all $P \in \proj$.
\item Argue in the same way as Theorem \ref{modelcat}, with condition (i). Or observe that \cite[Theorem 3.C.1.2.2]{Buhler} applies here, thanks to (i): the proof goes through \textit{mutatis mutandis} in our situation.
\item It is clear that maps in $I$-cell have this form. Retracts of split monomorphisms with projective cokernel also have this form.
\item This is clear from the definition of $J$.
\item $\Gamma$ and the equivalence of categories are known as the Dold-Kan correspondence. This is usually stated for abelian categories; a statement for idempotent complete additive categories is given in \cite[Theorem 1.2.3.7]{Lurie}. Since $\C$ is assumed to be complete and cocomplete, it is idempotent complete. Now choose a fuctor $N$ giving an inverse for $\Gamma$ up to natural isomorphism (that is, $\Gamma N$ and $N \Gamma$ are naturally isomorphic to the identity). These natural isomorphisms make $N$ a right and left adjoint to $\Gamma$. Then, for $\Gamma$ to be a right Quillen functor, we must show it preserves fibrations and trivial fibrations, while for it to be a left Quillen functor we must show it reflects them. Both follow by using the `apply $\underline{c\C}_{cAb}(S^0P,-)$ for all $P \in \proj$' argument, and noting that the usual Dold-Kan functor $cAb \to sAb$ has these properties: this is in \cite[Section 4.1]{SS}.
\end{enumerate}
\end{proof}

So for additive categories with enough functorial projectives, we will often be able to just talk about chain complexes instead of simplicial objects. It is more common in an additive context to discuss homological algebra using the language of triangulated categories, and it is worth mentioning that our categories do fit into this framework. (The convention is not to require that the choice of projective be functorial here, but we do anyway for the sake of consistency.) Indeed, for any additive category $\C$, the category $K(\C)$ of chain complexes in $\C$ with homotopic maps identified is triangulated, by \cite[Theorem 2.1.3]{Buhler}. Similarly for the bounded above, bounded below and bounded subcategories. To localise to a triangulated derived category $D(\C)$, we want a thick (that is, closed under isomorphisms and summands), triangulated subcategory $T$, thanks to \cite[Theorem 1.2.8]{Buhler}. The weak equivalences can be defined to be maps whose mapping cone is in $T$. Here we take $T$ to be the exact complexes, which is triangulated by Theorem \ref{chainmodelcat}(i). It is easily seen to be thick by applying $\underline{c\C}_{cAb}(S^0P,-)$ for all $P \in \proj$. Then we will write $D_+(\C,\proj)$ for the resulting localisation of $c\C$.

We will also write $D_+(\C,reg)$ for the localisation of $c\C$ at maps whose mapping cone is exact in the regular structure. As required, complexes which are exact in the regular structure form a triangulated subcategory, by \cite[Lemma 7.2]{BC}. It is thick by (the dual of) \cite[Corollary 2.18]{Buhler2}: the proof there also holds in our situation.

These localisations give canonical functors $D_+(\C,\proj) \to D_+(\C,\Q)$ whenever we have two classes of enough functorial projectives for $\C$ with $\proj \subseteq \Q$. Similarly, if $\proj$-split maps are regular (see Lemma \ref{regularproj}), we get a canonical functor $D_+(\C,\proj) \to D_+(\C,reg)$. This is just the universal property of localisations.

We are particularly interested in the case of $c(R$-$\U Mod)$. This has at least three interesting model structures, given by three different choices of projective objects: summands of all free modules, written $F(\U)$, summands of free modules on spaces in $\sqcup CH$, written $F(\sqcup CH)$, and summands of free modules on discrete spaces, written $F(Set)$. All three choices give enough projectives, by Remark \ref{projectivesinnature}: all three come, via the free module functor, from classes of enough projectives in $\U$ which give model structures on $\U$ (if we think of the Hurewicz model structure as coming from $\ob \U$, in the sense of Remark \ref{structuresonsU}). So we get three model structures on $c(R$-$\U Mod)$. Compare this to the situation for $s\U$: as noted in Remark \ref{structuresonsU}, it is not clear that we get all three model structures there.

By the universal property of free modules, we see that $F(\U)$-split (respectively, $F(\sqcup CH)$-split, $F(Set)$-split) maps are exactly the maps which as maps of the underlying spaces are split epimorphisms (respectively, $CH$-split epimorphisms, surjections). By Lemma \ref{regularproj}, $\proj$-split maps in $(c(R$-$\U Mod),F(\U))$ and in $(c(R$-$\U Mod), F(\sqcup CH))$ are regular epimorphisms, so they are cokernels, and hence $R$-$\U Mod$ with either $F(\U)$-split maps or $F(\sqcup CH)$-split maps is a left exact category. On the other hand, this is not true in general for $F(Set)$-split maps.

Suppose $\C$ is a regular additive category (equivalently, left quasi-abelian). Then we have seen that $Kan\C_{reg}$ is a category of fibrant objects. Every object in $s\C$ is fibrant in $(s\C,\C)$ by Theorem \ref{additivesc}, and hence fibrant in $(s\C,reg)$ by Lemma \ref{regularproj}, so $Kan\C_{reg} = s\C$. Then the Dold-Kan correspondence tranfers to $c\C$, giving it the structure of a category of fibrant objects, which we write $(c\C,reg)$. It is not hard to check under this correspondence that weak equivalences in $(c\C,reg)$ are maps of chain complexes whose mapping cone is exact in the regular structure, and fibrations are levelwise regular epimorphisms.

For $\C$ a regular additive category, its exact completion $\C_{ex}$ is additive, and hence is an abelian category by \cite[Theorem 3.11]{Barr}. As noted above, $Kan\C_{reg} = s\C$, so given $A \in s\C$ we can take homotopy $\pi_n(A) \in \C_{ex}$. For $B \in c\C$, we can take homology $H_n(B)$ in $\C_{ex}$ as well: just think of $B$ as an object in $c\C_{ex}$ and apply the usual homology functor.

\begin{prop}
$\pi_n(\Gamma(B)) \cong H_n(B)$ as group objects in $\C$, naturally in $B$ and $n$.
\end{prop}
\begin{proof}
Think of $B$ as a chain complex in the abelian category $\C_{ex}$, for which see \cite[Corollary III.2.5]{GJ}.
\end{proof}

Write $F: \U \to R$-$\U Mod$ for the free module functor, left adjoint to the forgetful functor. This extends to a left adjoint to the forgetful functor $sR$-$\U Mod \to s\U$, which applies $F$ degreewise; we will write $F$ for this functor too. $F$ may not preserve weak equivalences in $(s\U,\sqcup CH)$, but it is easy to see that it does give a Quillen adjunction between $(s\U,\sqcup CH)$ and $(sR$-$\U Mod, F(\sqcup CH))$, so we will also use the total left derived functor $LF$ of $F$, given as the composite of $F$ with the cofibrant replacement functor $Q: s\U \to s\U$.

In constructing a homology theory on $\U$, we would like certain axioms to be satisfied: those of a \emph{generalised homology theory}. These axioms are usually listed for homology theories from spaces to abelian groups, but they make sense in our context. A generalised homology theory is a functor $E$ from pairs $(X,Y)$ of spaces $Y \subseteq X$ in $\U$ to chain complexes in $R$-$\U Mod$, with some left exact structure, satisfying:
\begin{enumerate}[(i)]
\item homotopy invariance;
\item exactness: associated naturally to a pair $(X,Y)$ is an exact triangle $E(Y,\emptyset) \to E(X,\emptyset) \to E(X,Y) \to$;
\item additivity: if $(X,A)$ is a disjoint union of pairs $(\bigsqcup X_i, \bigsqcup A_i)$, then the canonical map $\bigoplus E(X_i,A_i) \to E(X,A)$ is a weak equivalence;
\item dimension: $E(\ast,\emptyset)$ is exact in non-zero dimensions;
\item excision: for $U \subseteq A \subseteq X$ with $\bar{U}$ contained in the interior of $A$, the canonical map $E(X \setminus U, A \setminus U) \to E(X,A)$ is a weak equivalence.
\end{enumerate}

Write $N$ for a functor $sR$-$\U Mod \to cR$-$\U Mod$ which is inverse to $\Gamma$ up to natural isomorphism. A naive attempt to define a topological homology theory would use $N \circ F \circ \Sing$: the abstract singular chain complex with the free module topologies on the spaces of singular maps. This is seen quite easily to satisfy the first four axioms. But excision seems to break down here. At any rate, if we give $R$-$\U Mod$ the $F(\U)$ model structure, the usual proof of excision does not work. It seems worthwhile here to point out what goes wrong. For concreteness, we work with the proof given in \cite[Section 2.1]{Hatcher}. The only real obstacle lies in \cite[Proposition 2.21, part (4)]{Hatcher}: each $D_m$ defined there is degreewise continuous, but there is no reason why $D$ should be.

On the other hand, using the $(s\U,CH)$ model structure suggests that we should replace $F$ with $LF$. Here we get better behaviour; indeed, the following results shows that the $F(\sqcup CH)$ model structure and the quasi-abelian structure on $R$-$\U Mod$ are indispensable to its study.

Suppose $X \in \U$. Let $C$ be an open cover of $X \in \U$, and write $C'$ for the poset of finite intersections of sets in $C$, ordered by inclusion. Then we get the following continuous version of \cite[Proposition 2.21]{Hatcher}.

\begin{thm}
$LF(\underline{\Sing}(X))$ is weakly equivalent (in $(s\U Ab,reg)$) to the homotopy colimit of $\{LF(\underline{\Sing}(U))\}_{U \in C'}$ (in $(s\U Ab,F(\sqcup CH))$).
\end{thm}
\begin{proof}
See \cite[Theorem 5.4]{Myself}.
\end{proof}

Inspired by this, we define singular homology on $\U$ to be $H^{\Sing} = N \circ LF \circ \Sing$, and on $s\U$ to ge $H^{\Sing} = N \circ LF$. We also define $n$th homology groups objects in the exact completion of $\U$ via $H^{\Sing}_n = H_n \circ H^{\Sing}$. It is also possible to define homology for pairs $Y \subseteq X$ in $\U$: we just define $H^{\Sing}_n(X,Y)$ to be the mapping cone of $N(LF(\Sing(Y))) \to N(LF(\Sing(X)))$. Homology and cohomology with coefficients may be defined similarly (we are using the fact that $R$-$U Mod$ is coregular to take cohomology: everything works dually to the homological case; see \cite{Schneiders}).

The unit of the adjunction between $F$ and the forgetful functor gives a map $Q(X) \to F(Q(X)) = LF(X)$ for $X \in s\U$, which induces maps $$\pi_n(X) = \pi_n(Q(X)) \to H^{\Sing}_n(Q(X)) = H^{\Sing}_n(X),$$ which we may think of as Hurewicz maps.

As for homotopy groups, we may define several other versions of `topological homology groups', and compare them. The situation for homology is similar to that for homotopy, and we leave it to the interested reader.

From the previous theorem one can easily deduce a version of the Excision Theorem for our homology theory.

\begin{thm}
Given subspaces $A \subseteq B \subseteq X$ in $\U$ with $A$ closed and $B$ open, the inclusion $(X \setminus A, B \setminus A)֓ \to (X,B)$ induces isomorphisms of the homology group objects $H^{\Sing}_n(X \setminus A, B \setminus A)֓ \to H^{\Sing}_n(X,B)$ for all $n$. Equivalently, for open subspaces $A, B \subseteq X$ covering $X$, the inclusion $(B,A \cap B)֓ \to (X,A)$ induces isomorphisms of homology group objects $H^{\Sing}_n(B,A \cap B) \to H^{\Sing}_n(X,A)$ for all $n$.
\end{thm}
\begin{proof}
See \cite[Theorem 5.6]{Myself}.
\end{proof}

Given our axioms for a generalised homology theory, the excision axiom was the only difficult thing to check. We immediately get:

\begin{thm}
$H^{\Sing}$ is a generalised homology theory.
\end{thm}
\begin{proof}
See \cite[Theorem 5.8]{Myself}.
\end{proof}

A lot of the standard results that apply to the usual definition of generalised homology theory immediately follow here. For instance, we immediately get a Mayer--Vietoris sequence for $H^{\Sing}$, by \cite[Theorem 5.7]{Myself}.

%

It is beyond the scope of the current work to investigate to what extent our current axioms for a generalised homology theory characterise such a theory up to natural isomorphism, as they are well-known to do in the classical setting. However, we do have the following result.

Generalise the definition of $\Delta$-complexes given in \cite[p.103]{Hatcher} to allow totally path-disconnected spaces of cells, as follows: a generalised $\Delta$-complex is a space $X \in \U$ with a collection of maps $\sigma_\alpha: \Delta_n \times K_\alpha \to X$, with $n$ varying with $\alpha$ and $K_\alpha$ a totally path-disconnected space, satisfying the following conditions.
\begin{enumerate}[(i)]
\item The restriction of $\sigma_\alpha$ to $(\Delta_n \setminus \partial\Delta_n) \times K_\alpha$ is injective, and each point of $X$ is in the image of exactly one such restriction.
\item Each restriction of $\sigma_\alpha$ to a face $\Delta_{n-1} \times K_\alpha$ of $\Delta_n \times K_\alpha$ factors through a continuous map $$\Delta_{n-1} \times K_\alpha \xrightarrow{\id_{\Delta_{n-1} \times f}} \Delta_{n-1} \times K_\beta \xrightarrow{\sigma_\beta} X.$$
\item A set $A \subseteq X$ is open if and only if $\sigma^{-1}_\alpha(A)$ is open for all $\sigma_\alpha$.
\end{enumerate}

As in the classical case, the definition ensures that such a $\Delta$-complex $X \in \U$ is the geometric realisation of some $X' \in s\U$ such that every $X'_n$ is totally path-disconnected. Call such spaces generalised simplicial complexes. As in the classical case, we define the simplicial homology $H^\text{Simp}(X)$ of $X$ to be the singular homology of $X'$.

\begin{prop}
The singular and simplicial homology theories for generalised $\Delta$-complexes are naturally isomorphic.
\end{prop}
\begin{proof}
See \cite[Proposition 6.6]{Myself}.
\end{proof}

\section{Derived functors}
\label{derived}

At this point we will say some more about derived functors. The convention when using the framework of triangulated categories is not to require the adjunction that is requested for model categories. So here we will drop the adjunction condition. Given $(\C,\proj)$ with enough projectives, a left exact category $\D$, and an additive functor $F: \C \to \D$, the total left derived functor $LF: D_+(\C,\proj) \to D_+(D)$ is given by taking projective resolutions and then applying $F$; this satisfies the expected universal property. Here $D_+(D)$ is the localisation of $K_+(D)$ defined at the end of \cite{BC}.

If $\D$ is regular and the left exact structure on $\D$ is at least as strong as the regular structure -- that is, if the deflations are regular epimorphisms -- then we can also define the $n$th left derived functor $L_nF: \C \to D_{ex}$ of $F$ to be the composite $$\C \to D_+(\C,\proj) \xrightarrow{LF} D_+(\D) \xrightarrow{H_n} \D_{ex},$$ and short exact sequences in $\C$ give long exact sequences in $\D_{ex}$. This gives an obvious notion of $F$-dimension of an object $A \in \C$ as the largest $n$ for which $L_nF(A) \neq 0$, and similarly a notion of dimension for $F$.

This is standard. But there seems to be a gap in the literature around bicomplexes beyond the abelian case, and we will make some more comments in this case. Obvious applications include spectral sequences and balanced bifunctors.

\begin{lem}
Suppose $(\C,\proj)$ has enough functorial projectives and $A$ is a right half-plane complex in $\C$. Suppose the rows of $A$ are exact chain complexes. Then the product total complex $\operatorname{Tot}^{\scriptstyle\prod} A$ is exact.
\end{lem}
\begin{proof}
The bicomplex $\C(P,A_{m,n})$ has exact rows for all $P \in \proj$, so $$\underline{c\C}_{cAb}(S^0P,\operatorname{Tot}^{\scriptstyle\prod} A) = \operatorname{Tot}^{\scriptstyle\prod} \C(P,A_{m,n})$$ is exact for all $P \in \proj$ by \cite[Lemma 2.7.3]{Weibel}, so $\operatorname{Tot}^{\scriptstyle\prod} A$ is exact.
\end{proof}

We call this lemma the Acyclic Assembly Lemma, after \cite[Lemma 2.7.3]{Weibel}.

We can get spectral sequences from a bicomplex $A$ when $\C$ is regular, by thinking of $A$ as a bicomplex in the abelian category $\C_{ex}$, where these things are exhaustively studied. A statement of this is given in \cite[Proposition 5.10]{Boggi}. This can be used to prove a Grothendieck spectral sequence, as stated in \cite[Theorem 5.12]{Boggi}.

As usual, the Acyclic Assembly Lemma shows that, given two composeable additive functors $(\C,\proj) \xrightarrow{F} (\D,\Q) \xrightarrow{G} \mathcal{E}$ where $\C$ and $\D$ have enough projectives, $\mathcal{E}$ is left exact, and $F(\proj) \subseteq \Q$, $L(GF) = LG \circ LF$. We will focus on some other applications.

Given a bifunctor $F: (\C,\proj) \times (\D,\Q) \to \mathcal{E}$, we define $$LF: D_+(\C,\proj) \times D_+(\D,\Q) \to D_+(\mathcal{E})$$ by $LF(A,B) = \operatorname{Tot} F(Q(A),Q(B))$, where $Q(A)$ and $Q(B)$ are projective resolutions of $A$ and $B$ respectively.

\begin{prop}
\label{balanced}
Suppose that $F(P,-)$ is exact for all $P \in \proj$. Then $LF(-,-)$ is naturally isomorphic to $\operatorname{Tot} F(Q(-),-)$ on the derived category.
\end{prop}
\begin{proof}
Consider the exact complex $Q'$ given by $Q(B)$ in non-negative degrees, and $B$ in degree $-1$. Apply the Acyclic Assembly Lemma to $F(Q(A),Q')$ to conclude that the mapping cone of $F(Q(A),Q(B)) \to F(Q(A),B)$ is exact, as required.
\end{proof}

Concrete examples we are interested in here are the bifunctors $\otimes_R$ and $\underline{\U_R}$ on $R$-$k$-modules. Before defining derived functors here, we have to know which model structure we are dealing with. Given $(\C,\proj)$ and $(\C,\Q)$ with enough projectives, if $\proj \subseteq \Q$ then $\Q$-split maps are not $\proj$-split maps in general. So a projective resolution in $(\C,\proj)$ is not necessarily a projective resolution in $(\C,\Q)$, or vice versa.

Recall that $Q$ is a commutative $k$-ring and $R$ is a $Q$-$k$-algebra. With respect to any of $F(\U)$, $F(\sqcup CH)$ or $F(Set)$, we can define the total left derived functor of $$\otimes_R: \U Mod\text{-}R \times R\text{-}\U Mod \to Q\text{-}\U Mod,$$ which is balanced by Proposition \ref{balanced}. Since $Q$-$\U Mod$ is regular too, we can take homology groups as well.

It would be nice to apply all this to $\underline{\U_R}(-,-)$ too, but as far as I know $R$-$\U Mod$ does not have enough injectives. There are some interesting subcategories that satisfy a form of Pontryagin duality: the category of locally compact modules over a locally compact ring is self-dual by \cite{Flood}, and the categories of ind-profinite and pro-discrete modules over a profinite ring are dual to each other by \cite{Boggi}. The ind-profinite modules have enough projectives, so the pro-discrete modules have enough injectives. But there seems to be no obvious way to extend this approach to all of $R$-$\U Mod$, so for now we will have to define the total derived functor of $\underline{\U_R}(-,-)$ to be calculated by taking a projective resolution in the first variable (where, again, the resulting derived functor depends on the choice of class of projectives). Then we can also take cohomology groups, because $Q$-$\U Mod$ is coregular.

What we have defined can be thought of as $\Ext$ and $\Tor$ groups, though we could just as easily have defined these as derived functors of $\otimes_R$ and $\U_R(-,-)$, thought of as functors to $Ab$. As for homology groups, we leave to the reader the job of comparing these to other possible definitions.

\begin{lem}
\label{Pexact}
Let $S$ be any of $\U$, $\sqcup$ or the class of discrete spaces. Consider $R$-$\U Mod$ and $Q$-$\U Mod$, both with the class of projectives given by summands of free $k$-modules on spaces in $S$. Then $\underline{\U_R}(P,-): R$-$\U Mod \to Q$-$\U Mod$ is an exact functor (that is, preserves exact complexes) for all projective $R$-$k$-modules $P$.
\end{lem}
\begin{proof}
The proof actually holds very generally. Given categories $(\C,\proj), (\V,\Q)$ with enough functorial projectives such that $\C$ is enriched and tensored over $\V$, suppose that $Q \odot P \in \proj$ for all $P \in \proj$ and all $Q \in \Q$. Then $\underline{\C}_\V(P,-)$ is exact. This is an exercise in using enrichment and tensors and we leave it to the reader; our specific result follows thanks to the tensoring defined in Theorem \ref{kModenriched}.
\end{proof}

\begin{prop}
Let $\proj$ be any of $F(\U)$, $F(\sqcup CH)$ or $F(Set)$. With $\proj$ as the class of projectives, the total derived functor $$R\underline{\U_R}(-,-): D_+(R\text{-}\U Mod) \times K_+(R\text{-}\U Mod) \to D_+(Q\text{-}\U Mod)$$ preserves weak equivalences in $K_+(R$-$\U Mod)$. That is, we actually get a functor on the derived categories $$D_+(R\text{-}\U Mod) \times D_+(R\text{-}\U Mod) \to D_+(Q\text{-}\U Mod).$$
\end{prop}
\begin{proof}
Given a weak equivalence $f$ in $K_+(R$-$\U Mod)$, $R\underline{\U_R}(A,\operatorname{cone}(f)) = \underline{c\U_R}(P,\operatorname{cone}(f))$, where $P$ is a projective resolution of $A$ and $\operatorname{cone}(f)$ is the mapping cone of $f$. Now consider the double complex $\underline{\U_R}(P_{-m},\operatorname{cone}(f)_n)$: this has exact columns by Lemma \ref{Pexact}, so apply the Acyclic Assembly Lemma (after rotating, to turn the columns into rows and put the double complex in the right half-plane) to get that $\underline{c\U_R}(P,\operatorname{cone}(f) = \operatorname{cone}(\underline{c\U_R}(P,f))$ is exact, as required.
\end{proof}

We will not give here an exhaustive account of the properties of these derived functors, which seem to be harder to calculate explicitly than the usual $\Ext$ and $\Tor$ groups. For a good summary of other attempts to work with (generally less categorical) derived functors for topological modules, see \cite{Stasheff}.

Recall that, for a $k$-group $G$, we get a group ring $Q[G]$; $Q$, with the trivial $G$-action, becomes a $Q[G]$-module. We can define group homology and cohomology for $G$ in the usual way, by considering the derived functors of $Q \otimes_{Q[G]} -$ and $\underline{\U_{Q[G]}}(Q,-)$. In this context it may be easier to take $F(\U)$ as our class of projectives, if only because the bar construction is a projective resolution in this case.

With the generating class $F(CH)$ of projectives for $Q[G]$-$\U Mod$, this also suggests definitions for $G$ to be of type $F(CH)-\FP_n$, $F(CH)-\FP_\infty$ and $F(CH)-\FP$, according to whether $Q$ has a projective resolution by $Q[G]$-modules generated by compact Hausdorff spaces for the first $n$ steps, and so on. Note that Schanuel's lemma still holds in this context, by the usual proof, so these definitions are well-behaved.

This property will be explored further in future research. In particular, there is a definition of type $\FP_n$ for totally disconnected, locally compact groups given in \cite{Ilaria}, using a category of discrete rational modules. In \cite{CCKS} we prove the following results:

\begin{thm}
Let $G$ be a totally disconnected, locally compact group. Then $G$ has type $\FP_n$ in the sense of \cite{Ilaria} if and only if it has type $F(CH)\FP_n$ over $\mathbb{Q}$, for $n \leq \infty$.
\end{thm}

We immediately get, for free, some classes of totally disconnected, locally compact groups of type $F(CH)-\FP_\infty$:
\begin{enumerate}[(i)]
\item Any discrete group of type $\FP_\infty$ over $\mathbb{Q}$;
\item Any profinite group;
\item Any tdlc group acting with compact, open stabilisers on a finite type, contractible cell complex;
\item Any hyperbolic tdlc group;
\item The Neretin group;
\item Any simply-connected semi-simple algebraic group defined over a non-discrete non-archimedean local field.
\end{enumerate}

On the other hand we do not know what can be deduced in our context from a totally disconnected, locally compact group having finite cohomological dimension in the sense of \cite{Ilaria}.

\begin{cor}
A totally disconnected, locally compact group has type $F(CH)-\FP_1$ if and only if it is \emph{compactly generated}, in the sense of being abstractly generated as a group by a compact subspace. If it is \emph{compactly presented} in the sense of \cite[Section 5.8]{Ilaria}, it has type $F(CH)-\FP_2$.
\end{cor}

\begin{rem}
This conflicting use of `compactly generated' is unfortunate, but fixed in the literature. Elsewhere in this paper we will only ever use the topological space definition.
\end{rem}

A useful tool in the homology and cohomology of abstract groups is Shapiro's lemma. The only obstruction to the corresponding result in our context is the requirement that, for $H$ a closed subgroup of $G$, $Q[G]$ be projective (for coinduction) or flat (for induction) as a $Q[H]$-module. Note that in the $k$-group world $Q[G]$ is not automatically free, as it is for abstract groups: for that argument to work we need $G \cong H \times G/H$ as $k$-spaces, which is not always the case. An easy counter-example is given by $G = \mathbb{R}$ with the Euclidean topology, $H = \mathbb{Z}$, so there is no continuous section $\mathbb{R}/\mathbb{Z} \to \mathbb{R}$.

Now suppose $S$ is $\ob\U, \sqcup CH$ or the class of discrete spaces.

\begin{lem}
Suppose that the quotient map $G \to G/H$ has a section in $\U$, and $G/H \in S$. Consider $Q[H]$-$\U Mod$ with $F(S)$ as the class of projectives. Then $Q[G]$ is projective as a $Q[H]$-module.
\end{lem}
\begin{proof}
As $Q[H]$-modules, $Q[G] \cong Q[H \times G/H]$, which is projective when $G/H \in S$.
\end{proof}

For example, when $S = \ob\U$, we only require that $G \to G/H$ has a section; when $S$ is discrete spaces, we must have $H$ open. As remarked earlier, a sufficient condition for $G/H$ to be in $\sqcup CH$ is that it be locally compact and totally disconnected, by van Dantzig's theorem. I known of no interesting cases where $Q[G]$ is projective or flat but not free.

Finally, we prove a Lyndon--Hochschild--Serre spectral sequence, in the following sense.

\begin{thm}
Suppose $S$ is $\ob\U, \sqcup CH$ or the class of discrete spaces. Suppose that $H \lhd G$, $G \to G/H$ has a section and $G/H \in S$. Consider $Q[G]$-$\U Mod$, $Q[H]$-$\U Mod$ and $Q[G/H]$-$\U Mod$ with $F(S)$ as the class of projectives. Then $$R\underline{\U_{Q[G]}}(Q,-) = R\underline{\U_{Q[G/H]}}(Q,R\underline{\U_{Q[H]}}(Q,-))$$ as functors on $D_+(Q[G]$-$\U Mod)$, and for $A \in Q[G]$-$\U Mod$ we get a spectral sequence $$H^p(G/H,H^q(H,A) \Rightarrow H^{p+q}(G,A).$$ Similarly for the total left derived functors $$Q (L\otimes_{Q[G]}) - = Q (L\otimes_{Q[G/H]}) (Q (L\otimes_{Q[H]}) -),$$ and we get a corresponding homology spectral sequence.
\end{thm}
\begin{proof}
Thanks to our hypotheses, a projective resolution of $Q$ by $Q[G]$-modules becomes a projective resolution of $Q$ by $Q[H]$-modules under restriction of the $G$-action. Then, thanks to what we have already shown, the proof of \cite[Section 3.5]{Benson} goes through unchanged.
\end{proof}

\end{document}